\documentclass[reqno]{amsart} 
\usepackage{amsaddr}

\usepackage[english]{babel}
\usepackage[latin1]{inputenc}  
\usepackage[T1]{fontenc}

\usepackage{amsfonts,amssymb,amsmath,amsthm} 
\usepackage{stmaryrd}
\numberwithin{equation}{section}
\usepackage{dsfont}
\usepackage{multirow}
\usepackage{url}
\usepackage{algorithmic,algorithm}
\usepackage{graphics}
\usepackage{booktabs}
\usepackage{epsfig}
\usepackage{epstopdf}
\usepackage{psfrag}                
\usepackage{mathrsfs}
\usepackage{mathtools}             
\usepackage{subfigure}
\usepackage{longtable}	
\usepackage{paralist}
\usepackage{comment}
\usepackage[inline]{enumitem}
\usepackage{color}
\usepackage{nicefrac}

\usepackage[colorlinks=true,linkcolor=blue,citecolor=blue]{hyperref}

\newcommand{\bbE}{\mathbb{E}}

\newcommand{\bbM}{\mathbb{M}}
\newcommand{\bbN}{\mathbb{N}}
\newcommand{\bbP}{\mathbb{P}}

\newcommand{\bbR}{\mathbb{R}}
\newcommand{\bbT}{\mathbb{T}}

\newcommand{\cA}{\mathcal{A}}
\newcommand{\cB}{\mathcal{B}}
\newcommand{\cC}{\mathcal{C}}

\newcommand{\cF}{\mathcal{F}}
\newcommand{\cG}{\mathcal{G}}
\newcommand{\cH}{\mathcal{H}}
\newcommand{\cI}{\mathcal{I}}

\newcommand{\cL}{\mathcal{L}}

\newcommand{\cP}{\mathcal{P}}

\newcommand{\cT}{\mathcal{T}}

\newcommand{\Omegat}{\widetilde{\Omega}} 
\newcommand{\omegat}{\widetilde{\omega}} 
\newcommand{\cAt}{\widetilde{\cA}} 
\newcommand{\bbPt}{\widetilde{\bbP}} 
\newcommand{\bbEt}{\widetilde{\bbE}} 

\newcommand{\rvec}{\boldsymbol{r}}
\newcommand{\etavec}{\boldsymbol{\eta}} 
\newcommand{\xivec}{\boldsymbol{\xi}} 

\newcommand{\xk}{\otimes^k} 
\newcommand{\xks}{\otimes^{k,s}}
\newcommand{\xkeps}{\otimes^k_\varepsilon}
\newcommand{\xkseps}{\otimes^{k,s}_{\varepsilon_s}}
\newcommand{\xkpi}{\otimes^k_\pi}
\newcommand{\xkspis}{\otimes^{k,s}_{\pi_s}}

\newcommand{\tenvec}{U}

\newcommand{\Mk}{\bbM^k_\varepsilon} 
\newcommand{\Mkpi}{\bbM^k_\pi} 

\newcommand{\norm}[2]{     \| #1       \|_{ #2 }}
\newcommand{\seminorm}[2]{  | #1        |_{ #2 }}

\newcommand{\normiii}[2]{\vert\kern-0.25ex\vert\kern-0.25ex\vert #1 \vert\kern-0.25ex \vert\kern-0.25ex\vert_{ #2 }}
\newcommand{\Normiii}[2]{\left\vert\kern-0.25ex\left\vert\kern-0.25ex\left\vert #1 \right\vert\kern-0.25ex\right\vert\kern-0.25ex\right\vert_{ #2 }}
\newcommand{\scalar}[2]{     ( #1       )_{ #2 }}

\newcommand{\duality}[1]{     \langle #1       \rangle}

\newcommand{\dual}[1]{{#1}^{\prime}}

\newcommand{\rd}{\mathrm{d}}

\newcommand{\dom}{\mathrm{D}}
\newcommand{\Wo}{\mathring{W}}

\newcommand{\from}{\colon}

\newcommand{\eps}{\varepsilon}
\newcommand{\sym}{s}
\newcommand{\CSL}{C^{\,\sf SL}_{q,p,k}}
\newcommand{\CML}{C^{\,\sf ML}_{q,p,k}}
\newcommand{\Cdiff}{C^{\,\sf diff}_{q,p,k}}
\newcommand{\Cstab}{C_{\sf stab}}

\newtheorem{lemma}{Lemma}[section]
\newtheorem{proposition}[lemma]{Proposition}
\newtheorem{theorem}[lemma]{Theorem}
\newtheorem{corollary}[lemma]{Corollary}

\theoremstyle{remark}
\newtheorem{remark}[lemma]{Remark}

\theoremstyle{definition}
\newtheorem{definition}[lemma]{Definition}
\newtheorem{assumption}[lemma]{Assumption}
\newtheorem{example}[lemma]{Example}

\begin{document}

\title[Monte Carlo methods for moments in Banach spaces]
	{Monte Carlo convergence rates for $k$th moments  
		in Banach spaces}  


\author{Kristin Kirchner$^{{\lowercase\mathrm{a}},*}$ 
	\and 
	Christoph Schwab$^{\lowercase\mathrm{b}}$}

\address[Kristin Kirchner]{$\,^{\lowercase\mathrm{a}}$Delft 
	Institute of Applied Mathematics, 
	Delft University of Technology\\
	P.O.~Box 5031,  2600$\,$GA Delft, 
	The Netherlands\\
	$\,^{*}$Corresponding author; 
	e-mail: {\normalfont\texttt{k.kirchner@tudelft.nl}}}

\address[Christoph Schwab]{
	$\,^{\lowercase\mathrm{b}}$Seminar for Applied Mathematics, 
	ETH Z\"urich \\
	R\"amistrasse 101, 8092 Z\"urich,  
	Switzerland}


\begin{abstract}
We formulate standard and multilevel Monte Carlo methods for  
the $k$th moment $\mathbb{M}^k_\varepsilon[\xi]$ 
of a Banach space valued random variable 
$\xi\colon\Omega\to E$, 
interpreted as an element 
of the $k$-fold injective 
tensor product space~$\otimes^k_\varepsilon E$. 
For the standard Monte Carlo estimator of 
$\mathbb{M}^k_\varepsilon[\xi]$, 
we prove the $k$-independent convergence rate 
$1-\tfrac{1}{p}$ 
in the $L_q(\Omega;\otimes^k_\varepsilon E)$-norm,  
provided that 
(i) $\xi\in L_{kq}(\Omega;E)$  
and 
(ii) $q\in[p,\infty)$, where $p\in[1,2]$ is  
the Rademacher type of $E$. 
By using the fact that 
Rademacher averages 
are dominated by Gaussian sums 
combined with 
a version of Slepian's inequality for Gaussian processes 
due to Fernique, 
we moreover derive 
corresponding results 
for multilevel Monte Carlo methods, 
including a rigorous error estimate 
in the $L_q(\Omega;\otimes^k_\varepsilon E)$-norm 
and the optimization of the computational cost  
for a given accuracy.  
Whenever the type of the 
Banach space~$E$ is $p=2$, 
our findings coincide with known results 
for Hilbert space valued random variables. 

We illustrate the abstract results 
by three model problems: 
second-order elliptic PDEs with 
random forcing or random coefficient, 
and stochastic evolution equations.
In these cases, the 
solution processes naturally take values 
in non-Hilbertian Banach spaces. 
Further applications, where physical modeling
constraints impose a setting in Banach spaces 
of type $p<2$, are indicated. 
\end{abstract}




\date{}

\maketitle

\section{Introduction}\label{section:intro} 

\subsection{Background and motivation}\label{subsec:motivation} 

Many applications in uncertainty quantification 
require the estimation of statistical moments. 
The first statistical moment, 
that is to say the mean, 
is often itself the quantity of interest, 
whereas higher-order moments 
are needed to infer 
certain characteristics about the 
probability distribution of 
the underlying real- or 
vector-valued random variable. 
In the case that this 
distribution is Gaussian, 
it is fully determined by 
the first two statistical moments. 
Third-order and fourth-order 
moments, which define the 
skewness and kurtosis of the probability distribution,  
play for instance an important role 
for tests if 
the distribution is Gaussian, 
see e.g.\ \cite{JarqueBera1987,Mardia1970}. 

In order to estimate statistical moments 
one resorts to sampling strategies, i.e., Monte Carlo methods. 
It is well-known that 
for estimating the mean  
the convergence 
rate $\nicefrac{1}{2}$ in the number of samples 
is achieved as long as 
the random variable $\xi$ is 
square-integrable in Bochner sense 
with values in a Hilbert space $H$, 
i.e., $\xi\in L_2(\Omega;H)$. 
Moreover, this result extends 
to statistical moments of an arbitrary order $k\in\bbN$ 
when interpreted as elements of the 
Hilbert tensor product space $H^{(k)}$\!,  
provided that the random variable 
exhibits sufficient integrability 
in $L_{2k}(\Omega;H)$. 

Vector-valued random variables occur, for instance,  
in the context of differential equations 
involving randomness. 
Here, numerical methods  
for generating samples of 
approximate solutions often 
allow for a hierarchical multilevel structure 
corresponding to different degrees
of refinement of the discretization parameters. 
The idea of multilevel Monte Carlo (MLMC) methods 
is to reduce the computational 
cost for achieving a given target accuracy 
by optimizing the number of samples 
used on each level 
to compute the MLMC estimator. 
To the best of our knowledge 
this approach was first formulated by Giles~\cite{Giles2008}  
for stochastic ordinary differential equations (SDEs)
after having previously been introduced 
by Heinrich~\cite{Heinrich2001} 
in the context of numerical integration. 
Since then MLMC methods 
have been used to approximate means 
of Hilbert space valued random variables 
for a variety of problems in uncertainty quantification, 
including but not limited to 
SDEs~\cite{CollierEtAl2015, 
	GilesMilstein2008,
	GilesDebrabantRossler2019,  
	GilesSzpruch2014, 
	XiaGiles2012}, 
partial differential equations (PDEs) with random 
coefficients~\cite{BarthSchwabZollinger2011, 
	CharrierScheichlTeckentrup2013, 
	CliffeGilesScheichlTeckentrup2011, 
	CollierEtAl2015, 
	GittelsonKonnoSchwabStenberg2013, 
	GrahamScheichlUllmann2016, 
	HajiAliEtAl2016,  
	TeckentrupScheichlGilesUllmann2013}, 
stochastic  
PDEs~\cite{BarthLangSchwab2013, 
	GilesReisinger2012}, 
and 
hyperbolic PDEs with random fluxes 
or uncertainties in the initial 
data~\cite{MishraRisebroSchwabTokareva2016, 
	MishraSchwab2012, 
	MishraSchwab2017}. 

With regard to higher-order moments, MLMC strategies  
have been applied to estimate 
\emph{diagonals of central 
	statistical moments} in \cite{BierigChernov2015,BierigChernov2016}, 
and combined with sparse tensor techniques  
to approximate (\emph{full}) moments 
in \cite{BarthSchwabZollinger2011,
	MishraSchwab2012, 
	MishraSchwab2017}; 
the latter approach has been refined and generalized 
by the multi-index Monte Carlo method 
in~\cite{MIMC2016}.  

All of the previously mentioned references 
have in common that the error analysis of the 
(multilevel) Monte Carlo estimators proceeds in Hilbert spaces. 
For the first statistical moment, it is known that 
estimating means of random variables taking 
values in a Banach space~$E$ 
via the standard Monte Carlo method 
\emph{does in general not converge at the rate~$\nicefrac{1}{2}$}, 
even in the presence of high Bochner integrability. 
More specifically, the rate of convergence 
depends on geometric properties of the Banach space~$E$: 
If $E$ has \emph{Rademacher type}  
$p\in[1,2]$ and the random variable is an 
element of the Bochner space~$L_q(\Omega;E)$, $q\geq 1$, 
the standard Monte Carlo method will, 
in general, converge only 
at the rate~$1-\tfrac{1}{\min\{p,q\}}$, 
see~\cite[Proposition~9.11]{LedouxTalagrand2011}. 
This behavior necessitates 
tailoring of MLMC methods 
not only to the discretization 
of a particular problem, but also  
to the type of the Banach space. 
The only two references known to the authors 
addressing this issue are~\cite{KoleyRisebroSchwabWeber2017},
where a scalar, degenerate conservation law with random data
is discretized by a MLMC finite difference method, 
and~\cite{CoxEtAl2021}, 
where the authors perform a MLMC  analysis in 
H\"older spaces for solutions to 
stochastic evolution equations. 

Besides the aforementioned issue 
of type-dependent convergence rates, 
another difficulty occurs when considering 
higher-order moments of Banach space valued 
random variables: 
As opposed to the Hilbert space case, 
there is no canonical choice for  
the norm on the tensor product space~$E\otimes E$. 
Two options which are widely used in the literature 
are the \emph{projective} and 
\emph{injective tensor product norms}, 
mostly caused by the fact that any reasonable 
cross norm is bounded from above, respectively from below, 
by these norms, see~\cite[Proposition~6.1(a)]{Ryan2002}. 
Janson and Kaijser~\cite{JansonKaijser2015} 
defined and analyzed 
statistical moments of order $k\in\bbN$ 
(in Bochner, Dunford or Pettis sense) 
as elements of projective and injective 
tensor product spaces. 
One of the findings \cite[Theorem~3.8]{JansonKaijser2015} 
shows that 
both the \emph{projective} and \emph{injective $k$th 
	moment} of $\xi\from\Omega\to E$ exist in Bochner sense 
(and coincide) whenever~$\xi\in L_k(\Omega;E)$. 

Clearly, the choice of the tensor product norm 
will play a crucial role in the error analysis 
of (multilevel) Monte Carlo estimation 
of higher-order statistical moments. 
A simple argumentation 
(see Example~\ref{example:projective-bad})
shows that, 
no matter how ``good'' the (Rademacher) type 
of the Banach space $E$ is, 
Monte Carlo methods for the second moment 
will in general not converge 
in the projective tensor product 
of~$E$.  

While this work is devoted to the numerical analysis
of Monte Carlo (sampling) methods 
to approximate higher-order moments 
of vector-valued random variables, 
we remark that 
certain \emph{linear} (or \emph{linearized}) 
stochastic equations allow for 
alternative approaches 
to \emph{deterministically compute}  
approximations to $k$-point correlations
of random solutions. 
In this context, we mention
\cite{ChernovSchwab2013, Kirchner2020,KirchnerLangLarsson2017,
	KovacsLangPetersson2022,LangLarssonSchwab2013}
and the references therein. 
This methodology does not raise  
the mathematical issue of the type of 
a Banach space and its impact 
on the convergence of 
numerical approximations. 
In the present paper, we
shall not pursue this direction further.

\subsection{Contributions}\label{subsec:contributions} 

We consider the injective $k$th moment $\Mk[\xi]$ 
of a Banach space valued random variable 
$\xi\in L_k(\Omega;E)$ and, for the first time, 
formulate standard and multilevel Monte Carlo methods 
for this higher-order moment 
in the Banach space setting. 
We prove that the standard Monte Carlo estimator 
for $\Mk[\xi]$ converges in the 
$L_q(\Omega;\xkeps E)$-norm 
at the rate $1-\tfrac{1}{p}$ 
provided that 
(i) $\xi\in L_{kq}(\Omega;E)$  
and 
(ii) $q\in[p,\infty)$, where $p\in[1,2]$ is  
the Rademacher type of $E$, 
see Theorem~\ref{thm:MC-k}. 
Here, $\xkeps E$ denotes the 
$k$-fold injective tensor product of $E$. 
Note, in particular, that\linebreak  
this convergence rate 
is independent of the order $k$ 
of the statistical moment~$\Mk[\xi]$. 
This result readily implies error estimates 
and convergence rates for abstract 
\emph{single-level Monte Carlo} methods 
for $\Mk[\xi]$, see Corollary~\ref{cor:MC-k}. 

By means of replacing Rademacher sums  
by Gaussian averages and exploiting 
a version of Slepian's inequality for Gaussian processes 
due to Fernique \cite{Fernique1975}, 
we are furthermore able to formulate 
corresponding abstract results 
for \emph{multilevel Monte Carlo} methods. 
This includes a rigorous error estimate 
in the $L_q(\Omega;\xkeps E)$-norm, 
see Theorem~\ref{thm:MLMC-k}, 
and the optimization of the computational cost  
for a given accuracy 
(in the MLMC context also known as 
``$\alpha\beta\gamma$ theorem''), 
see Theorem~\ref{thm:alpha-beta-gamma}. 

We apply these abstract results 
to several classes of problems, 
where the stochastic solution 
processes naturally take values in 
(non-Hilbertian) Banach spaces: 
second-order elliptic PDEs  
\begin{enumerate*}[label=(\alph*)]
	\item with random forcing  
	taking values in $L_p$ 
	for some general $p\in(1,\infty)$, or 
	\item with log-Gaussian diffusion coefficient 
	and right-hand side in $L_p$; and 
	\item stochastic evolution equations, 
	where we extend the MLMC analysis in H\"older 
	norms, performed in \cite[Section~5]{CoxEtAl2021} 
	for mean values of solution processes,  
	to their $k$th moments, 
	being $k$th order (spatio-)temporal
	correlation functions. 
\end{enumerate*}

\subsection{Layout}
\label{subsec:outline} 

In Section~\ref{section:prelim} 
we introduce the necessary notation, 
see Subsection~\ref{subsec:notation}, 
as well as the analytical preliminaries on  
tensor products of a Banach space~$E$, 
with particular emphasis on the full and symmetric 
$k$-fold injective tensor product of $E$, 
see Subsection~\ref{subsec:k-tensor}.  
We close this section with the definition 
of the injective $k$th moment of a Banach space 
valued random variable $\xi\in L_k(\Omega;E)$ 
in Subsection~\ref{subsec:moments}. 
Section~\ref{section:abstract} is dedicated 
to the analysis of Monte Carlo 
methods for the injective $k$th moment. 
For this purpose, 
we first need to formulate several auxiliary 
results for Rademacher and Gaussian averages 
in Subsection~\ref{subsec:auxiliary}. 
We then perform the error analysis for the 
standard (and single-level) Monte Carlo method 
in Subsection~\ref{subsec:slmc} 
and for the multilevel Monte Carlo method 
in Subsection~\ref{subsec:mlmc}. 
In Section~\ref{section:applications} 
we discuss several applications of our 
convergence results. 
Section~\ref{section:conclusions} 
gives an outlook on extensions 
and further applications, where 
due to essential restrictions in the physical modeling 
a (non-Hilbertian) Banach space setting cannot 
be avoided.

\section{Preliminaries}\label{section:prelim}   

\subsection{General notation and setting}\label{subsec:notation} 

Given parameter sets  
$\mathscr{P},\,\mathscr{Q}$, 
and mappings 
$\mathscr{F},\mathscr{G}\from \mathscr{P}\times\mathscr{Q}\to\bbR$, 
we use the notation  
$\mathscr{F} (p,q) \lesssim_{q} \mathscr{G} (p,q)$ 
to indicate that for each $q \in \mathscr{Q}$ 
there exists a constant ${C_q \in(0,\infty)}$ 
such that 
$\mathscr{F} (p,q) \leq C_q \, \mathscr{G} (p,q)$ 
holds 
for all ${p\in \mathscr{P}}$. 
Whenever both relations, 
$\mathscr{F} (p,q) \lesssim_q \mathscr{G}(p,q)$ 
and $\mathscr{G} (p,q) \lesssim_q \mathscr{F}(p,q)$, 
hold simultaneously, 
we write $\mathscr{F} (p,q) \eqsim_q \mathscr{G} (p,q)$.   

For a Banach space $(F,\norm{\,\cdot\,}{F})$ over $\bbR$, 
we write 
$B_F:= \{x\in F : \norm{x}{F} \leq 1 \}$ 
for its closed unit ball, 
and 
$\cB(F)$ for the Borel $\sigma$-algebra 
on 
$(F,\norm{\,\cdot\,}{F})$, 
that is the $\sigma$-algebra generated by 
the open sets. 
The dual space 
of all continuous linear functionals 
$g \from F \to \bbR$ 
is denoted by~$\dual{F}$\!.  
We write $g(x)$ or $\duality{g, x}$ 
for the duality pairing between $g\in \dual{F}$ 
and $x\in F$, and 
$\norm{g}{\dual{F}} := \sup_{x\in B_F} | g(x) |$ 
for the norm on~$\dual{F}$\!. 

Throughout this article, we assume that 
$(\Omega,\cA,\bbP)$ is 
a complete probability space
with expectation operator $\bbE$, 
and we mark statements which hold 
$\bbP$-almost surely with $\bbP$-a.s. 
For a second probability space 
$(\Omegat,\cAt,\bbPt)$ with expectation operator $\bbEt$, 
$(\Omega\times\Omegat,\cA\otimes\cAt,\bbP\otimes\bbPt)$ 
denotes the product probability space, i.e.,  
$\Omega\times\Omegat$ is the 
set of all tuples $(\omega,\omegat)$ with  
$\omega\in\Omega$, $\omegat\in\Omegat$, 
$\cA\otimes\cAt$ is the product $\sigma$-algebra generated by 
all sets of the form $A\times\widetilde{A}$ with $A\in\cA$, 
$\widetilde{A}\in\cAt$, 
and $\bbP\otimes\bbPt$ is the uniquely defined 
product measure satisfying 
$(\bbP\otimes\bbPt)(A\times\widetilde{A})
= 
\bbP(A)\bbPt(\widetilde{A})$ 
for all $A\in\cA$ and $\widetilde{A}\in\cAt$. 
The expectation operator on 
$(\Omega\times\Omegat,\cA\otimes\cAt,\bbP\otimes\bbPt)$ 
will be denoted by $\bbE\otimes\bbEt$. 

In addition, we let   
$(E,\norm{\,\cdot\,}{E})$ be a Banach space 
over~$\bbR$, the set $\bbN$ 
contains all 
(strictly) positive integers 
and, unless otherwise stated, 
$k \in \bbN$ is a fixed positive integer 
which indicates the order of (statistical) moments.

\subsection{\texorpdfstring{$k$-fold}{k-fold} tensor products  
	of Banach spaces}\label{subsec:k-tensor}    

In this subsection we define 
(full and symmetric) $k$-fold 
tensor products 
of the Banach space $(E,\norm{\,\cdot\,}{E})$, 
with the aim of 
obtaining a new Banach space 
satisfying the following: 
\begin{enumerate}[label={\normalfont(\roman*)}, leftmargin = 1cm]
	\item It contains the set of all 
	$k$th moments  
	(in Bochner sense)
	of Bochner integrable random 
	variables $\xi\from\Omega\to E$ 
	satisfying $\bbE\bigl[ \norm{\xi}{E}^k \bigr] < \infty$. 
	\item The topology on this space 
	(prescribed by its norm) allows to 
	quantify the convergence 
	of Monte Carlo estimation for  
	statistical moments of order $k$. 
\end{enumerate} 
For this purpose, symmetry of moments 
will be particularly important.  

We start by defining the 
(full) \emph{$k$-fold algebraic tensor product of $E$}, 
\[
\xk E = \underbrace{E \otimes \cdots \otimes E}_{
	\text{\normalfont $k$ times}}, 
\] 
that is the vector space consisting of all finite sums of the form
\[
\sum_{j=1}^{M} 
x_{j,1} 
\otimes 
\cdots 
\otimes 
x_{j,k} 
= 
\sum_{j=1}^{M} 
\bigotimes_{n=1}^k 
x_{j,n}, 
\qquad 
x_{j,n}\in E,  
\quad 
1\leq j \leq M, 
\quad 
1\leq n \leq k, 
\]
equipped with the algebraic operations 
rendering the tensor product  
linear in each of the $k$ components, 
see \cite[Section~1.1]{Floret1997}. 
The \emph{injective tensor norm} $\norm{\,\cdot\,}{\eps}$ 
for an element 
$\tenvec\in\xk E$ with representation 
$\tenvec=\sum_{j=1}^{M} 
x_{j,1} 
\otimes 
\cdots 
\otimes 
x_{j,k}$ 
is defined by 
\begin{equation}\label{eq:epsnorm} 
\norm{\tenvec}{\eps} 
:= 
\sup\Biggl\{
\biggl| 
\sum_{j=1}^{M} 
\prod_{n=1}^{k} 
f_n(x_{j,n})  
\biggr| 
\; \Bigg| \;  
f_1,\ldots,f_k\in B_{\dual{E}} 
\Biggr\}, 
\end{equation} 
cf.~\cite[Section~3.1]{Ryan2002} 
and \cite[Section~2.3.2]{JansonKaijser2015}. 
Note that 
the value of
the supremum in~\eqref{eq:epsnorm} 
is independent of the choice of the representation 
of $\tenvec\in\xk E$, since 
\[
\biggl| 
\sum_{j=1}^{M} 
\prod_{n=1}^{k} 
f_n(x_{j,n})  
\biggr|  
= 
\biggl| 
\sum_{j=1}^{M} 
\duality{ (f_1 \otimes \cdots \otimes f_k), 
	(x_{j,1} 
	\otimes 
	\cdots 
	\otimes 
	x_{j,k})}  
\biggr| 
= 
|\duality{(f_1 \otimes \cdots \otimes f_k) , \tenvec}|.   
\]
For $k=1$, the Hahn--Banach theorem shows that 
$\norm{\,\cdot\,}{\eps} = \norm{\,\cdot\,}{E}$. 
We call the completion of the $k$-fold 
algebraic tensor product space $\xk E$ 
with respect to $\norm{\,\cdot\,}{\eps}$ in \eqref{eq:epsnorm} 
the (full) \emph{$k$-fold injective tensor product of $E$} 
and denote it by~$\xkeps E$. 

In the context of moments of 
$E$-valued random variables, 
we will be interested in subspaces 
of $\xk E$ and $\xkeps E$ containing  
only their symmetric elements. 
To this end, we first introduce for 
$x_1\otimes\cdots\otimes x_k \in \xk E$ 
its \emph{symmetrization} 
\begin{equation}\label{eq:symmetrization} 
\sym(x_1\otimes \cdots\otimes x_k) 
:= 
\frac{1}{k!} 
\sum_{\sigma\in S_k} 
x_{\sigma(1)} \otimes \cdots \otimes x_{\sigma(k)}, 
\end{equation} 
where 
$S_k$ is the group of permutations 
of the set $\{1, . . . , k\}$.
The \emph{$k$-fold symmetric algebraic tensor product of $E$}, 
denoted by $\xks E$,  
is then defined as the linear span 
of $\{\sym(x_1\otimes \cdots \otimes x_k) : x_1,\ldots,x_k\in E\}$ 
in $\xk E$, i.e., 
\begin{align*} 
\xks E 
&:= 
\Biggl\{
\sum_{j=1}^M 
\sym(x_{j,1} \otimes\cdots\otimes x_{j,k}) 
\; \Bigg| \; 
M\in\bbN,  
\; 
x_{j, n}\in E, 
\; 
1\leq j \leq M, 
\; 
1\leq n \leq k
\Biggr\}
\\
&\textcolor{white}{:}= 
\Biggl\{
\sum_{j=1}^M 
\lambda_j \, {\xk} x_j  
\;\Bigg| \;
M\in\bbN, 
\; 
\lambda_j \in \bbR, 
\; 
x_j \in E, 
\; 
1\leq j\leq M
\Biggr\}
\\
&\textcolor{white}{:}= 
\Biggl\{
\sum_{j=1}^M 
\delta_j \, {\xk} x_j  
\;\Bigg| \;
M\in\bbN, 
\; 
\delta_j \in \mathscr{E}(k), 
\; 
x_j \in E, 
\;  
1\leq j\leq M
\Biggr\}, 
\end{align*} 
see \cite[Section~1.5]{Floret1997} 
or \cite[Section~1.1]{Floret2001}. 
Here, we set 
\[
{\xk} x := \underbrace{x\otimes \cdots \otimes x}_{\normalfont 
	\text{$k$ times}} 
\quad 
\forall x \in E, 
\quad 
\text{and} 
\quad  
\mathscr{E}(k) := 
\begin{cases} 
\{-1,1\} 
&\text{if $k$ is even}, 
\\
\{1\} 
&\text{if $k$ is odd}. 
\end{cases} 
\]
The \emph{symmetric injective tensor norm} 
$\norm{\,\cdot\,}{\eps_s}$ on 
the $k$-fold symmetric algebraic 
tensor product space $\xks E$ 
is given by 
(see \cite[Section~3.1]{Floret1997}) 
\begin{equation}\label{eq:sepsnorm} 
\norm{\tenvec}{\eps_s} 
:= 
\sup\Biggl\{
\biggl| 
\sum_{j=1}^M 
\lambda_j f(x_j)^k
\biggr| 
\; \Bigg| \; 
f\in B_{\dual{E}} 
\Biggr\}, 
\end{equation} 
if $\tenvec=\sum_{j=1}^M \lambda_j \, {\xk} x_j$, 
where $\lambda_j\in\bbR$, $x_j\in E$ 
for $1\leq j \leq M$.  
Note that, 
as for the injective tensor norm $\norm{\,\cdot\,}{\eps}$, 
this definition does 
not depend on the choice of 
the representation of $\tenvec$. 
The \emph{$k$-fold symmetric injective tensor product of $E$}, 
denoted by $\xkseps E$, 
is the completion of $\xks E$ with respect to 
the norm $\norm{\,\cdot\,}{\eps_s}$ in \eqref{eq:sepsnorm}. 

The symmetrization  
$\sym$ in \eqref{eq:symmetrization} 
extends linearly 
to a projection $\sym \from {\xk} E \to \xks E$  
and, since for 
$\tenvec=\sum_{j=1}^M x_{j,1}\otimes \cdots \otimes x_{j,k} \in \xk E$ 
we moreover have that 
\begin{equation}\label{eq:norm-rel} 
\begin{split}  
&\norm{\sym(\tenvec)}{\eps_s} 
\leq 
\norm{\sym(\tenvec)}{\eps} 
= 
\biggl\| 
\sum_{j=1}^M 
\frac{1}{k!} 
\sum_{\sigma\in S_k} 
x_{j,\sigma(1)} \otimes 
\cdots \otimes x_{j,\sigma(k)} 
\biggr\|_{\eps} 
\\
&\quad\leq 
\frac{1}{k!} 
\sum_{\sigma\in S_k} 
\biggl\| 
\sum_{j=1}^M 
x_{j,\sigma(1)} \otimes 
\cdots \otimes x_{j,\sigma(k)} 
\biggr\|_{\eps}  
=
\biggl\| 
\sum_{j=1}^M 
x_{j,1} \otimes 
\cdots \otimes x_{j,k} 
\biggr\|_{\eps} 
= 
\norm{\tenvec}{\eps}, 
\end{split}
\end{equation} 
it also has a unique continuous extension 
to a linear projection 
$\sym_\eps\from {\xkeps} E \to \xkseps E$. 
Furthermore, 
on $\xkseps E$ 
the injective tensor norm $\norm{\,\cdot\,}{\eps}$ 
and the symmetric injective tensor norm $\norm{\,\cdot\,}{\eps_s}$ 
are equivalent, with 
$k$-dependent equivalence constants, 
\begin{equation}\label{eq:eps-norm-equivalence} 
\norm{ \tenvec }{\eps_s} 
\leq 
\norm{ \tenvec }{\eps} 
\leq 
\frac{k^k}{k!} 
\norm{ \tenvec }{\eps_s} 
\quad 
\forall 
\tenvec \in \xkseps E. 
\end{equation} 
see \cite[Section~2.3 and~2.7]{Floret2001}.

\begin{remark}\label{rem:projective-norm} 
	There are several meaningful options to define norms  
	on the $k$-fold algebraic tensor product spaces~$\xk E$ 
	and~$\xks E$. 
	Besides the injective tensor norm, 
	a common choice is the 
	\emph{projective tensor norm} $\norm{\,\cdot\,}{\pi}$, 
	defined for $\tenvec \in \xk E$ by 
	\begin{equation}\label{eq:pinorm}
	\norm{\tenvec}{\pi} 
	:= 
	\inf\Biggl\{ 
	\sum_{j=1}^M \prod_{n=1}^k \norm{x_{j,n}}{E} 
	\; \Bigg| \;  
	M\in\bbN, \, 
	\tenvec = 
	\sum_{j=1}^{M} 
	x_{j,1} 
	\otimes 
	\cdots 
	\otimes 
	x_{j,k} 
	\Biggr\}. 
	\end{equation}
	The \emph{symmetric projective tensor norm} 
	on~$\xks E$ is given by 
	(see \cite[Section~2.2]{Floret1997}) 
	\begin{equation}\label{eq:spinorm}
	\norm{\tenvec}{\pi_s} 
	:= 
	\inf\Biggl\{ 
	\sum_{j=1}^M  
	|\lambda_j| \,
	\norm{x_j}{E}^k 
	\; \Bigg| \;  
	M\in\bbN, \, 
	\tenvec = 
	\sum_{j=1}^{M} 
	\lambda_j \, {\xk} x_j  
	\Biggr\}. 
	\end{equation}
	Then, for every $\tenvec\in \xk E$,  
	$\norm{\tenvec}{\pi}\geq \norm{\tenvec}{\eps}$ 
	and 
	$\norm{\sym(\tenvec)}{\pi_s} 
	\geq 
	\max\{ \norm{\sym(\tenvec)}{\eps_s}, \norm{\sym(\tenvec)}{\pi}\}$ hold. 
	The closures of $\xk E$ and of $\xks E$ 
	with respect to the norms~$\norm{\,\cdot\,}{\pi}$ 
	and~$\norm{\,\cdot\,}{\pi_s}$, respectively, 
	yield well-defined Banach spaces, 
	the \emph{full} and \emph{symmetric 
		$k$-fold projective tensor product of $E$}, 
	denoted by $\xkpi E$ and $\xkspis E$. 
	As shown in Example~\ref{example:projective-bad}, 
	the projective norm is not suitable 
	for the error analysis of Monte Carlo methods. 
\end{remark}

\subsection{Moments of Banach space valued random variables} 
\label{subsec:moments} 

The purpose of this subsection is to generalize the notion 
of the $k$th moment, defined for a real-valued 
random variable $X\from\Omega\to\bbR$ as 
\[
\bbM^k[X] 
:= 
\bbE\bigl[ X^k \bigr] 
= 
\int_{\Omega} X(\omega)^k \, \rd\bbP(\omega), 
\]
to Banach space valued random variables 
$\xi\from\Omega\to E$. 
To this end, we first specify 
the concept of vector-valued integration 
which we imply when taking expected values 
of $E$-valued random variables. 

\subsubsection{Vector-valued integration}

An $E$-valued random variable $\xi$ defined on the 
probability space $(\Omega,\cA,\bbP)$ is a mapping 
$\xi\from\Omega\to E$ 
which is measurable in a certain sense. 
Specifically, we consider the 
class of Bochner measurable 
random variables; these are all mappings 
$\xi\from\Omega\to E$ which are 
\begin{enumerate*} 
	\item measurable 
	with respect to the Borel $\sigma$-algebra $\cB(E)$ 
	on $E$, and 
	\item almost surely separably valued, i.e., 
	$\xi\in E_0$ $\bbP$-a.s.\ 
	for some separable subspace $E_0\subseteq E$. 
\end{enumerate*} 
A Bochner measurable random variable $\xi$  
is often also called \emph{strongly measurable}, 
emphasizing the contrast to the notion 
of weak measurability, which only requires 
the real-valued random variable 
$\duality{ f, \xi }$ to be measurable 
for every $f\in \dual{E}$\!. 
Note that 
these notions are equivalent  
whenever $\xi$ is 
almost surely separably valued 
(e.g., in the case of  
a separable Banach space $E$), 
see \cite[Theorem~2.3]{JansonKaijser2015}.

Furthermore, it turns out that 
$\xi\from\Omega\to E$ is 
Bochner measurable if and only if there 
exists a sequence of 
Borel measurable simple functions 
$\xi_n\from\Omega\to E$, $n\in\bbN$, 
such that $\xi_n \to \xi$, $\bbP$-a.s. 
This characterization facilitates the definition 
of the Bochner integral 
\[
\int_{\Omega} \xi(\omega)\, \rd\bbP(\omega) 
=:  
\bbE[\xi] \in E, 
\]
whenever $\xi$ is Bochner measurable 
and $\bbE\bigl[ \norm{\xi}{E} \bigr] <\infty$. 

We close this subsection with introducing 
the corresponding Bochner $L_q$-spaces. 
For a real Banach space $(F,\norm{\,\cdot\,}{F})$ 
and $q\in[1,\infty)$, 
$L_q(\Omega;F):=L_q(\Omega,\cA,\bbP;F)$ 
is the space of all 
(equivalence classes of) $F$-valued 
Bochner measurable random variables 
$\eta\from\Omega\to F$ 
such that $\bbE\bigl[ \norm{\eta}{F}^q \bigr] < \infty$, 
with norm given by 
\[
\norm{\eta}{L_q(\Omega;F)} 
:= 
\left(
\bbE\bigl[
\norm{\eta}{F}^q
\bigr]
\right)^{\nicefrac{1}{q}} 
= 
\biggl( 
\int_{\Omega} 
\norm{\eta(\omega)}{F}^q 
\, \rd\bbP(\omega)
\biggr)^{\nicefrac{1}{q}} \!. 
\]

\subsubsection{Moments of order $k$}

For an $E$-valued random variable $\xi$, 
its \emph{injective $k$th moment $\Mk[\xi]$} is defined as 
the expectation 
(see e.g.~\cite[Section~3.1]{JansonKaijser2015})
\begin{equation}\label{eq:Mk} 
\Mk[\xi] 
:= 
\bbE\bigl[ \xk \xi \bigr] 
= 
\int_{\Omega}  \xk \xi(\omega) \, \rd \bbP(\omega) 
= 
\int_\Omega 
\underbrace{\xi(\omega) \otimes \cdots \otimes \xi(\omega)}_{ 
	\text{\normalfont $k$ times}}
\, \rd\bbP(\omega), 
\end{equation} 
whenever this integral exists (in Bochner sense)
in the $k$-fold injective tensor product space $\xkeps E$. 
In what follows, we will always assume 
that $\xi \in L_q(\Omega;E)$ holds for some $q\in[k,\infty)$.  
This guarantees that $\Mk[\xi]$ exists in Bochner sense: 
Firstly, 
Bochner measurability~of~$\xi$ 
implies that also 
$\xk \xi \from \Omega \to \xkeps E$ 
is Bochner measurable 
since the non-linear 
mapping $E\ni x \mapsto \xk x \in \xkeps E$ 
is continuous and, secondly, 
\[
\bbE\bigl[ \norm{ {\xk} \xi }{\eps} \bigr] 
= 
\bbE \bigl[ \norm{ \xi }{E}^k \bigr] 
\leq 
\bbE \bigl[ \norm{ \xi }{E}^q \bigr] 
= 
\norm{\xi}{L_q(\Omega;E)}^q
< 
\infty. 
\]

Of particular relevance 
in the present context is that 
the injective $k$th moment 
is an element of the 
$k$-fold \emph{symmetric} 
injective tensor product 
space $\xkseps E$, since 
\[
\sym_\eps 
\bigl( 
\Mk[\xi]
\bigr)
= 
\sym_\eps
\bigl( 
\bbE\bigl[ \xk \xi \bigr]
\bigr)
= 
\bbE\bigl[ \sym_\eps \bigl( \xk \xi \bigr) \bigr] 
= 
\bbE\bigl[ \xk \xi \bigr] 
= 
\Mk[\xi]. 
\]
Here, we have used 
continuity of the symmetrization 
$\sym_\eps\from {\xkeps} E \to \xkseps E$, 
see \eqref{eq:symmetrization} and 
\eqref{eq:norm-rel}, to exchange the order 
of $\sym_\eps(\,\cdot\,)$ and $\bbE[\,\cdot\,]$. 

\begin{remark}\label{rem:pi-moment} 
	The assumption $\xi \in L_k(\Omega;E)$ 
	does not only guarantee the existence of 
	the injective $k$th moment $\Mk[\xi]$
	but also that of the 
	projective $k$th moment $\Mkpi[\xi]$, 
	i.e., the integral in \eqref{eq:Mk} 
	converges in Bochner sense 
	also in the stronger projective tensor norm 
	defined in \eqref{eq:pinorm}, 
	see Remark~\ref{rem:projective-norm}. 
	This observation follows from the identities 
	$\norm{{\xk} \xi}{\pi} 
	= 
	\norm{\xi}{E}^k
	=
	\norm{{\xk} \xi}{\eps}$ 
	showing that the above arguments 
	may be translated verbatim, 
	see also \cite[Theorem~3.8]{JansonKaijser2015}. 
\end{remark} 

\section{Monte Carlo estimation of the \texorpdfstring{$k$th}{kth} moment} 
\label{section:abstract}   

This section treats 
the analysis of abstract standard, single-level and 
multilevel Monte Carlo 
methods to estimate the injective $k$th moment $\Mk[\xi]$ 
of a Banach space valued random variable $\xi\in L_k(\Omega;E)$.  
In Subsection~\ref{subsec:auxiliary} 
we first provide necessary definitions, 
including those of 
Rademacher and orthogaussian families 
as well as the type of a Banach space.   
Furthermore, we formulate auxiliary 
results based on comparison theorems 
for Rademacher and Gaussian averages.  
These observations facilitate 
the error analysis for  
standard and single-level Monte Carlo estimation  
in Subsection~\ref{subsec:slmc} 
and for the multilevel Monte Carlo method 
in Subsection~\ref{subsec:mlmc}. 

\subsection{Auxiliary results on Rademacher and Gaussian averages} 
\label{subsec:auxiliary} 

\begin{definition}[Rademacher family]\label{def:rade-family} 
	Let $(\Omegat,\cAt,\bbPt)$ 
	be a complete probability space and 
	$r_j \from \Omegat\to\{-1,1\}$, $j\in J\subseteq\bbN$, 
	be a family of independent random variables such that  
	$\bbPt(r_j = -1)=\bbPt(r_j = 1)=\tfrac{1}{2}$ for all $j\in J$. 
	Then, the collection of random variables 
	$(r_j)_{j\in J}$ is called a \emph{Rademacher family}.
\end{definition} 

\begin{definition}[Orthogaussian family]\label{def:orthogaussian} 
	Let $(\Omegat,\cAt,\bbPt)$ 
	be a complete probability space 
	with expectation~$\bbEt$. 
	Suppose that 
	$g_j \from \Omegat\to\bbR$, $j\in J\subseteq\bbN$, 
	are independent standard Gaussian random variables, 
	i.e., $\bbEt[g_j] = 0$, $\bbEt\bigl[g_j^2 \bigr] = 1$ 
	for all $j\in J$, 
	and $\bbEt[g_i g_j] = 0$ for $i\neq j$.  
	Then, the collection $(g_j)_{j\in J}$ is called an 
	\emph{orthogaussian family}. 
\end{definition} 

Assuming that $(z_j)_{j=1}^M$ is a Rademacher 
or orthogaussian family 
and $(x_j)_{j=1}^M$ are vectors in the Banach space $E$, 
the $E$-norm of the finite random sum 
$\sum_{j=1}^M z_j x_j$ 
is the supremum of a real-valued 
(Rademacher or Gaussian) stochastic process. 
More specifically, the Hahn--Banach theorem allows us to rewrite 
the norm as follows, 
\begin{equation}\label{eq:rade-gaussian-norm} 
\biggl\| \sum_{j=1}^M z_j x_j \biggr\|_E 
= 
\sup_{f\in B_{\dual{E}}}  
\biggl| \sum_{j=1}^M z_j f(x_j) \biggr| 
= 
\sup_{(t_1,\ldots,t_M)\in T} 
\biggl| \sum_{j=1}^M z_j t_j \biggr| , 
\end{equation}
where $T$ is the compact subset of $\bbR^M$\! given by 
$T:=\{ (f(x_1),\ldots,f(x_M)) : f\in B_{\dual{E}} \}$.  

The next lemma summarizes comparison theorems 
for Gaussian and Rademacher averages, see 
\cite[Corollary~3.17 \& Theorem~4.12]{LedouxTalagrand2011}. 
It will facilitate generalizing 
the equality~\eqref{eq:rade-gaussian-norm} 
to an upper bound for $L_q$-norms of 
random variables of the form 
\[
\sup_{f\in B_{\dual{E}}}  
\biggl| \sum_{j=1}^M r_j f(x_j)^{k_j} \biggr| 
\quad \text{and} \quad 
\sup_{f\in B_{\dual{E}}}  
\biggl| \sum_{j=1}^M g_j f(x_j)^{k_j} \biggr|,
\qquad  
k_j\in\bbN, \quad 1\leq j \leq M. 
\]

\begin{lemma}\label{lem:rade-gaussian} 
	Let $M\in\bbN$, 
	and let $z := (z_j)_{j=1}^M$ be 
	a Rademacher or 
	orthogaussian family 
	on a complete probability space 
	$(\Omegat,\cAt,\bbPt)$. 
	Suppose that $(\varphi_j)_{j=1}^M$ are  
	functions on~$\bbR$ such that, 
	for every $1\leq j\leq M$,  
	\begin{equation}\label{eq:ass:contraction}
	\varphi_j(0) = 0 
	\qquad\text{and}\qquad 
	|\varphi_j(s) - \varphi_j(t)| 
	\leq 
	|s-t| 
	\quad \forall s,t\in\bbR. 
	\end{equation}
	Assume further that $G\from [0,\infty) \to [0,\infty)$ 
	is convex and increasing. Then, 
	we have for any bounded subset $T\subset\bbR^M$\!  
	the relation   
	\begin{equation}\label{eq:lem:rade-gaussian} 
	\widetilde{\bbE}
	G\Biggl( \frac{1}{2} 
	\sup_{(t_1,\ldots,t_M)\in T} 
	\biggl| 
	\sum_{j=1}^M 
	z_j \varphi_j(t_j) \biggr| 
	\Biggr) 
	\leq 
	\widetilde{\bbE} 
	G\Biggl(  
	C_z 
	\sup_{(t_1,\ldots,t_M)\in T} 
	\biggl| 
	\sum_{j=1}^M 
	z_j t_j \biggr| 
	\Biggr),   
	\end{equation} 
	where $\bbEt$ denotes the expectation operator 
	on $(\Omegat, \cAt, \bbPt)$, and $C_z$ is given by 
	\begin{equation}\label{eq:C-rade-gaussian} 
	C_z 
	:= 
	\begin{cases} 
	1 &\text{ if $z$ is a Rademacher family} , 
	\\
	2 &\text{ if $z$ is an orthogaussian family}. 
	\end{cases} 
	\end{equation}
\end{lemma}  

\begin{proposition}\label{prop:rade-gaussian}  
	Let $M\in\bbN$, 
	$x_1,\ldots,x_M\in E$, 
	and $(z_j)_{j=1}^M$ be 
	a Rademacher or 
	orthogaussian family 
	on a complete probability space 
	$(\Omegat,\cAt,\bbPt)$ 
	(with expectation $\bbEt$). 
	In addition, assume that 
	$G\from [0,\infty) \to [0,\infty)$ 
	is a convex and increasing function. 
	\begin{enumerate}[leftmargin=1cm] 
		\item\label{prop:rade-gaussian-poly}  
		For integers $k_1, \ldots, k_M\in\bbN$,  
		we have the relation  
		\begin{equation}\label{eq:prop:rade-gaussian-poly} 
		\bbEt 
		G\Biggl( 
		\sup_{f\in B_{\dual{E}}}  
		\biggl| 
		\sum_{j=1}^M 
		z_j f(x_j)^{k_j} \biggr| 
		\Biggr) 
		\leq 
		\bbEt 
		G\Biggl(  
		2 C_z \, 
		\biggl\| 
		\sum_{j=1}^M 
		z_j k_j \norm{x_j}{E}^{k_j-1} x_j \biggr\|_E 
		\Biggr).  
		\end{equation}
		\item\label{prop:rade-gaussian-q}   
		For general exponents 
		$q_1,\ldots,q_M\in[1,\infty)$,  
		the following holds: 
		\begin{equation}\label{eq:prop:rade-gaussian-q} 
		\bbEt 
		G\Biggl(  
		\sup_{f\in B_{\dual{E}}}  
		\biggl| 
		\sum_{j=1}^M 
		z_j |f(x_j)|^{q_j} \biggr| 
		\Biggr) 
		\leq 
		\bbEt 
		G\Biggl(  
		2 C_z \, 
		\biggl\| 
		\sum_{j=1}^M 
		z_j q_j \norm{x_j}{E}^{q_j-1}  x_j \biggr\|_E 
		\Biggr). 
		\end{equation} 
	\end{enumerate}
	Here, the constant $C_z\in\{1,2\}$ is defined 
	as in \eqref{eq:C-rade-gaussian}. 
\end{proposition} 

\begin{remark} 
	Part~\ref{prop:rade-gaussian-poly} of Proposition~\ref{prop:rade-gaussian} 
	is a generalization  
	of the observation 
	made by Ledoux and Talagrand 
	in~\cite[Equation~(4.19)]{LedouxTalagrand2011}; 
	there formulated for 
	Rademacher averages, 
	$G(x):=x$, and 
	$k_1=\ldots=k_M=2$. 
	That is, we generalize to  
	higher-order polynomials on one hand, 
	and to Gaussian averages on the other~hand. 
\end{remark} 

\begin{proof}[Proof of Proposition~\ref{prop:rade-gaussian}] 
	We will prove~\ref{prop:rade-gaussian-poly} 
	and~\ref{prop:rade-gaussian-q} 
	using Lemma~\ref{lem:rade-gaussian}. 
	Without loss of generality 
	we may assume that $x_j\neq 0$ 
	for all $1\leq j\leq M$. 
	
	To derive \ref{prop:rade-gaussian-poly},  
	let $M\in\bbN$ and 
	$k_j\in\bbN$,  
	$\widetilde{x}_j \in E\setminus\{0\}$ 
	for all $1\leq j\leq M$. 
	Furthermore, for $1\leq j \leq M$, 
	define 
	$I_j := \bigl[ -\norm{\widetilde{x}_j}{E}, \norm{\widetilde{x}_j}{E} \bigr]$ 
	and 
	\begin{equation}\label{eq:varphi-tilde}
	\widetilde{\varphi}_j
	\from 
	I_j \to \bbR, 
	\qquad 
	\widetilde{\varphi}_j(s) 
	:= 
	k_j^{-1} \norm{\widetilde{x}_j}{E}^{1 - k_j}  
	s^{k_j}, 
	\quad
	s\in I_j. 
	\end{equation}
	The function $\widetilde{\varphi}_j$ is continuously differentiable 
	on the interior of $I_j$ 
	with $|\widetilde{\varphi}_j'(s)|\leq 1$. 
	Let 
	$P_j \from \bbR\to I_j$ denote the projection 
	$P_j(s) := \max\{-\norm{\widetilde{x}_j}{E}, \min\{s,\norm{\widetilde{x}_j}{E}\}\}$ 
	onto the interval $I_j$.  
	Then, for every $1\leq j\leq M$, 
	the function 
	\[
	\varphi_j \from \bbR\to \bbR, 
	\qquad 
	\varphi_j(s) := \widetilde{\varphi}_j(P_j(s)), 
	\]
	satisfies the assumptions \eqref{eq:ass:contraction} 
	of Lemma~\ref{lem:rade-gaussian}.  
	Therefore, 
	letting the bounded set $T\subset \bbR^M$ 
	be given by 
	$T:= \{ (f(\widetilde{x}_1), \ldots, f(\widetilde{x}_M)) : f\in B_{\dual{E}} \}$, 
	we may combine the observation \eqref{eq:rade-gaussian-norm} 
	with \eqref{eq:lem:rade-gaussian} 
	and by also noting that $T\subseteq I_1\times \ldots \times I_M$ 
	we obtain that, for every convex and increasing 
	function $G\from [0,\infty) \to [0,\infty)$, 
	\begin{align} 
	\widetilde{\bbE}  
	G\Biggl( \frac{1}{2} 
	\sup_{f\in B_{\dual{E}}} 
	&\biggl| 
	\sum_{j=1}^M 
	z_j k_j^{-1} \norm{\widetilde{x}_j}{E}^{1-k_j} f(\widetilde{x}_j)^{k_j} \biggr| 
	\Biggr) 
	= 
	\widetilde{\bbE}   
	G\Biggl( \frac{1}{2} 
	\sup_{(t_1,\ldots,t_M)\in T} 
	\biggl| 
	\sum_{j=1}^M 
	z_j \varphi_j(t_j) \biggr| 
	\Biggr)  
	\notag 
	\\
	&\quad\leq 
	\widetilde{\bbE}    
	G\Biggl(  
	C_z 
	\sup_{(t_1,\ldots,t_M)\in T} 
	\biggl| 
	\sum_{j=1}^M 
	z_j t_j \biggr| 
	\Biggr)
	=  
	\widetilde{\bbE}   
	G\Biggl(  
	C_z \, 
	\biggl\| 
	\sum_{j=1}^M 
	z_j \widetilde{x}_j \biggr\|_E
	\Biggr).  
	\label{eq:proof:prop:rade-gaussian}   
	\end{align} 
	Finally, for $x_1,\ldots,x_M\in E\setminus\{0\}$, 
	we choose  
	$\widetilde{x}_j := 2 k_j \norm{x_j}{E}^{k_j-1} x_j \in E\setminus\{0\}$ 
	for all $1\leq j\leq M$ 
	and \eqref{eq:proof:prop:rade-gaussian} shows 
	\eqref{eq:prop:rade-gaussian-poly}. 
	
	For~\ref{prop:rade-gaussian-q} 
	we modify the above arguments 
	by replacing $k_j$ 
	and 
	the interval $I_j$ 
	in the definition 
	\eqref{eq:varphi-tilde} 
	of $\widetilde{\varphi}_j$
	with $q_j$ 
	and $I_j := \bigl[ 0, \norm{\widetilde{x}_j}{E} \bigr]$, 
	respectively, 
	and 
	the projection $P_j\from\bbR\to I_j$ 
	with 
	$P_j(s) := \min\{|s|, \norm{\widetilde{x}_j}{E}\}$. 
	Then, the function $\varphi_j(s) := \widetilde{\varphi}_j( P_j(s) )$ 
	satisfies the assumptions~\eqref{eq:ass:contraction} 
	of Lemma~\ref{lem:rade-gaussian}, 
	since $\varphi_j(0) = 0$, 
	$|\widetilde{\varphi}_j'(s)| \leq 1$ 
	for all $s\in(0,\norm{\widetilde{x}_j}{E})$ and 
	by the mean value theorem 
	combined with the reverse triangle inequality 
	we thus have that 
	\[
	|\varphi_j(s) - \varphi_j(t) | 
	= 
	|\widetilde{\varphi}_j(P_j(s)) - \widetilde{\varphi}_j(P_j(t)) | 
	\leq 
	| P_j(s) - P_j(t) | 
	\leq 
	\bigl| |s| - |t| \bigr| 
	\leq 
	|s-t|. 
	\]
	Furthermore, for every $f\in B_{\dual{E}}$, we have that  
	$P_j(f(\widetilde{x}_j)) = |f(\widetilde{x}_j)|$ 
	and we obtain the analogue 
	of \eqref{eq:proof:prop:rade-gaussian}, 
	\[
	\widetilde{\bbE}  
	G\Biggl( \frac{1}{2} 
	\sup_{f\in B_{\dual{E}}} 
	\biggl| 
	\sum_{j=1}^M 
	z_j q_j^{-1} \norm{\widetilde{x}_j}{E}^{1-q_j} |f(\widetilde{x}_j)|^{q_j} \biggr| 
	\Biggr) 
	\leq 
	\widetilde{\bbE}   
	G\Biggl(  
	C_z \, 
	\biggl\| 
	\sum_{j=1}^M 
	z_j \widetilde{x}_j \biggr\|_E
	\Biggr).  
	\]
	The choice 
	$\widetilde{x}_j := 2 q_j \norm{x_j}{E}^{q_j-1} x_j$, 
	$1\leq j\leq M$, completes the proof 
	of \eqref{eq:prop:rade-gaussian-q}. 
\end{proof} 

The next lemma shows that we may 
symmetrize independent random variables 
with vanishing mean,
when bounding $L_q$-norms of their sum. 
This result can be found, e.g., 
in~\cite[Lemma~6.3]{LedouxTalagrand2011}  
or~\cite[Lemma~5.9]{CoxEtAl2021}.  

\begin{lemma}[Symmetrization]
	\label{lem:symmetrization} 
	Let  $q\in[1,\infty)$, $M\in\bbN$, $(r_j)_{j=1}^M$ be 
	a Rademacher family on a complete probability space 
	$(\Omegat,\cAt,\bbPt)$, 
	and let $\eta_1,\ldots,\eta_M\in L_q(\Omegat;F)$ be 
	centered random variables, 
	i.e., $\bbEt[\eta_j]=0$ for $1\leq j \leq M$, 
	taking values 
	in a real Banach space $(F,\norm{\,\cdot\,}{F})$ 
	such that $\eta_1,\ldots,\eta_M, r_1, \ldots, r_M$ are independent. 
	Then, 
	\[
	\biggl\| 
	\sum_{j=1}^M \eta_j 
	\biggr\|_{L_q(\Omegat;F)} 
	\leq 
	2 \, 
	\biggl\| 
	\sum_{j=1}^M r_j \eta_j 
	\biggr\|_{L_q(\Omegat;F)} . 
	\]
\end{lemma} 

\begin{definition}[Kahane--Khintchine constants]
	\label{def:Kqp-constant}
	Assume that $p,q\in[1,\infty)$. 
	The  $(q,p)$ Kahane--Khintchine constant $K_{q,p}$ 
	is the smallest constant $K\in(0,\infty)$ 
	such that 
	for any Rademacher family $(r_j)_{j\in\bbN}$ 
	on a complete probability space 
	$(\Omegat,\cAt,\bbPt)$, 
	for any real Banach space $(F,\norm{\,\cdot\,}{F})$,  
	for all $n\in\bbN$, 
	and every $x_1,\ldots,x_n\in F$, 
	\begin{equation}\label{eq:kahane-khintchine} 
	\biggl\| 
	\sum_{j=1}^n r_j x_j
	\biggr\|_{L_q(\Omegat;F)} 
	\leq K \,
	\biggl\| 
	\sum_{j=1}^n r_j x_j
	\biggr\|_{L_p(\Omegat;F)}. 
	\end{equation} 
\end{definition} 

\begin{remark} 
	For the case $q\leq p$, 
	H\"older's inequality 
	shows 
	that ${K_{q,p}=1}$. 
	Finiteness of the constant 
	$K_{q,p}$ in the non-trivial case $q>p$ 
	was derived by Kahane~\cite{Kahane1985}; 
	it implies that, for Rademacher sums,  
	all $L_q$-norms 
	with $q\in[1,\infty)$     
	are equivalent. 
\end{remark} 

\begin{remark}\label{rem:kahane-khintchine-gaussian} 
	By invoking an argument 
	based on the central limit theorem 
	(see \cite[p.~103]{LedouxTalagrand2011}) 
	the Kahane--Khintchine inequality 
	for Rademacher sums \eqref{eq:kahane-khintchine} implies 
	a corresponding result for Gaussian averages: 
	For all $p,q\in[1,\infty)$, 
	for any orthogaussian family $(g_j)_{j\in\bbN}$ 
	on a complete probability space 
	$(\Omegat,\cAt,\bbPt)$, 
	for any real Banach space $(F,\norm{\,\cdot\,}{F})$,  
	for all $n\in\bbN$, 
	and every $x_1,\ldots,x_n\in F$, we have that 
	\begin{equation}\label{eq:kahane-khintchine-gaussian} 
	\biggl\| 
	\sum_{j=1}^n g_j x_j
	\biggr\|_{L_q(\Omegat;F)} 
	\leq K_{q,p} \,
	\biggl\| 
	\sum_{j=1}^n g_j x_j
	\biggr\|_{L_p(\Omegat;F)}. 
	\end{equation} 
\end{remark} 

\begin{definition}[Type $p$ constant]\label{def:type-p-constant}
	A real Banach space $(F,\norm{\,\cdot\,}{F})$ 
	is said to be of \emph{(Rademacher) type} $p\in[1,2]$ 
	if there exists a constant $\tau\in(0,\infty)$ 
	such that 
	for any Rademacher family $(r_j)_{j\in\bbN}$ 
	on a complete probability space 
	$(\Omegat,\cAt,\bbPt)$ 
	(with expectation $\bbEt$), 
	for every $n\in\bbN$, 
	and all vectors $x_1,\ldots,x_n\in F$, 
	\begin{equation}\label{eq:type} 
	\biggl\| 
	\sum_{j=1}^n r_j x_j 
	\biggr\|_{L_p(\Omegat;F)}  
	=
	\Biggl( 
	\bbEt   
	\Biggl[ 
	\biggl\|
	\sum_{j=1}^n r_j x_j 
	\biggr\|_F^p 
	\Biggr] 
	\Biggr)^{\nicefrac{1}{p}} 
	\leq 
	\tau 
	\Biggl( 
	\sum_{j=1}^n \norm{x_j}{F}^p 
	\Biggr)^{\nicefrac{1}{p}} \!. 
	\end{equation} 
	In this case, the smallest constant $\tau\in(0,\infty)$ 
	in \eqref{eq:type} 
	is called the \emph{type $p$ constant of $F$} 
	and will be denoted by $\tau_p(F)$. 
\end{definition}

\begin{remark}
	The definition of the type of a Banach space~$(F,\norm{\,\cdot\,}{F})$
	is often complemented with the notion of 
	its cotype: 
	$F$ has \emph{cotype} $q\in[2,\infty]$ 
	if there exists a constant $c\in(0,\infty)$ 
	such that 
	for any Rademacher family $(r_j)_{j\in\bbN}$ 
	on a complete probability space $(\Omegat,\cAt,\bbPt)$, 
	for every $n\in\bbN$, 
	and all vectors $x_1,\ldots,x_n\in F$, 
	\begin{align*}
	\qquad\quad
	\Biggl( 
	\sum_{j=1}^n \norm{x_j}{F}^q 
	\Biggr)^{\nicefrac{1}{q}} 
	&\leq 
	c \, 
	\biggl\| 
	\sum_{j=1}^n r_j x_j 
	\biggr\|_{L_q(\Omegat;F)} 
	&&
	\text{if } q \in [2,\infty), 
	\qquad\quad
	\\
	\sup_{1\leq j \leq n} 
	\norm{x_j}{F}
	&\leq 
	c \, 
	\biggl\| 
	\sum_{j=1}^n r_j x_j 
	\biggr\|_{L_1(\Omegat;F)}  
	&&\text{if } 
	q=\infty. 
	\end{align*} 
\end{remark}  

\begin{remark} 
	Every Banach space has type $1$ 
	and cotype~$\infty$
	by the triangle inequality and L\'evy's inequality 
	(see e.g.\ \cite[Proposition~2.3]{LedouxTalagrand2011}), 
	respectively.  
	Conversely, by the (classical) Khintchine inequalities 
	(see e.g.\ \cite[Lemma~4.1]{LedouxTalagrand2011}) 
	there exist constants $A_q, B_q\in(0,\infty)$ 
	depending only on $q\in[1,\infty)$ 
	such that, for any Rademacher family $(r_j)_{j\in\bbN}$ 
	on~$(\Omegat,\cAt,\bbPt)$ 
	and any finite sequence $(a_j)_{j=1}^n$ 
	of real numbers, 
	\begin{equation}\label{eq:khintchine} 
	\hspace{-1.5mm} 
	A_q 
	\Biggl( \sum_{j=1}^n a_j^2 \Biggr)^{\nicefrac{1}{2}} 
	\leq 
	\Biggl( \bbEt \Biggl[ \biggl| \sum_{j=1}^n r_j a_j \biggr|^q \Biggr] \Biggr)^{\nicefrac{1}{q}} 
	= 
	\biggl\| 
	\sum_{j=1}^n r_j a_j 
	\biggr\|_{L_q(\Omegat;\bbR)}  
	\leq 
	B_q 
	\Biggl( \sum_{j=1}^n a_j^2 \Biggr)^{\nicefrac{1}{2}} 
	\hspace{-1.5mm} 
	\end{equation} 
	which implies that the type cannot be larger than $2$ 
	and the cotype cannot be smaller than $2$.  
	Kwapie\'n~\cite{Kwapien1972}~showed that 
	a Banach space has type $2$ and cotype $2$ if and only if 
	it is isomorphic to a Hilbert space. 
\end{remark} 

\begin{example}\label{example:type-tensor}  
	Let $(H, \scalar{\,\cdot\,, \,\cdot\,}{H})$ 
	be a real separable Hilbert space. 
	In this case, 
	the \emph{Hilbert tensor product space} $\otimes_{w_2}^2 H$ 
	(see Appendix~\ref{appendix:norm-computation} for the definition) 
	is again a Hilbert space 
	and, consequently, 
	has type $p=2$. 
	However, 
	none of the tensor product spaces 
	$\otimes_\pi^2 H$, $\otimes^{2,s}_{\pi_s}H$, 
	$\otimes_\eps^2 H$ or $\otimes^{2,s}_{\eps_s}H$ 
	has a non-trivial type~${p>1}$. 
	
	This can be seen by the following counterexample: 
	Let $(e_j)_{j\in\bbN}$ be an orthonormal basis for~$H$. 
	Then, for all $p\in[1,\infty)$ and every $n\in\bbN$, we have 
	\[
	\Biggl( 
	\sum_{j=1}^n \norm{e_j\otimes e_j}{\pi}^p 
	\Biggr)^{\nicefrac{1}{p}} 
	=
	\Biggl( 
	\sum_{j=1}^n \norm{e_j\otimes e_j}{\eps}^p 
	\Biggr)^{\nicefrac{1}{p}} 
	= 
	\Biggl( 
	\sum_{j=1}^n \norm{e_j}{H}^{2p}  
	\Biggr)^{\nicefrac{1}{p}} 
	=
	n^{\nicefrac{1}{p}} . 
	\]
	Moreover, the
	calculations in Appendix~\ref{appendix:norm-computation} 
	(see the identities \eqref{eq:diagonal-pi-norms} 
	and \eqref{eq:diagonal-eps-norms} 
	of Lemma~\ref{lem:diagonal-norms})
	imply that 
	for any Rademacher family $(r_j)_{j\in\bbN}$ 
	on a complete probability space 
	$(\Omegat,\cAt,\bbPt)$, 
	for all $p\in[1,\infty)$ 
	and 
	for every $n\in\bbN$,  
	\begin{align*} 
	\biggl\| 
	\sum_{j=1}^n r_j \, e_j \otimes e_j  
	\biggr\|_{L_p(\Omegat; \otimes^2_\pi H)}  
	=
	n 
	\qquad 
	\text{and} 
	\qquad   
	\biggl\| 
	\sum_{j=1}^n r_j \, e_j \otimes e_j  
	\biggr\|_{L_p(\Omegat; \otimes^2_\eps H)}  
	=
	1, 
	\end{align*} 
	and the same statements hold when 
	replacing $\otimes^2_\pi H$ by the symmetric projective tensor product 
	space $\otimes^{2,s}_{\pi_s} H$ 
	and 
	$\otimes^2_\eps H$ by the symmetric injective tensor product 
	space $\otimes^{2,s}_{\eps_s} H$, respectively. 
	This shows that (i) neither 
	$\otimes^2_\pi H$ nor $\otimes^{2,s}_{\pi_s} H$
	have~a~non-trivial type $p>1$, 
	and (ii) neither  
	$\otimes^2_\eps H$ nor $\otimes^{2,s}_{\eps_s} H$ 
	have a
	non-trivial cotype~$q<\infty$. 
	Thus, 
	$\otimes^2_\eps H$ and $\otimes^{2,s}_{\eps_s} H$ 
	do not have a non-trivial type either, 
	cf.~\cite[Theorem~7.1.14]{AnalysisInBanachSpacesII2017}. 
\end{example} 

In the next subsections we will see 
that the type $p\in[1,2]$ of a Banach space~$E$
determines the rate of convergence 
when approximating statistical moments 
of $E$-valued random variables 
by means of Monte Carlo methods  
and, moreover, that this convergence rate 
does not depend on the order $k$ of the moment. 
However, as the above example illustrates, 
to derive this finding, 
it is not possible to argue 
by transferring the type of 
a Banach space to its $k$-fold tensor product. 

\subsection{Standard and single-level Monte Carlo estimation} 
\label{subsec:slmc} 

The next proposition is the key result 
for proving convergence 
of Monte Carlo methods for means, 
i.e., statistical 
moments of order $k=1$. 
It can be found, e.g., in \cite[Proposition~9.11]{LedouxTalagrand2011} 
for the case $q=p$ and 
in this generality in \cite[Proposition~5.10]{CoxEtAl2021}. 

\begin{proposition}\label{prop:MC-1} 
	Assume that  
	$(E,\norm{\,\cdot\,}{E})$ 
	is of Rademacher type $p\in[1,2]$. 
	Let $q\in[p,\infty)$, $M\in\bbN$ 
	and $\eta_1,\ldots,\eta_M \in L_q(\Omega;E)$ be 
	independent $E$-valued random variables 
	with vanishing mean, 
	$\bbE[\eta_j]=0$ for all $1\leq j \leq M$. 
	Then, 
	\[ 
	\biggl\| 
	\sum_{j=1}^M 
	\eta_j 
	\biggr\|_{L_q(\Omega;E)} 
	\leq 
	2 K_{q,p} \tau_p(E) 
	\Biggl( 
	\sum_{j=1}^M
	\norm{ \eta_j }{L_q(\Omega; E)}^p 
	\Biggr)^{\nicefrac{1}{p}} \!.
	\]
\end{proposition}  

\begin{corollary}\label{cor:MC-1} 
	Assume that  
	$(E,\norm{\,\cdot\,}{E})$ 
	is of Rademacher type $p\in[1,2]$ 
	and let $\eta\in L_1(\Omega;E)$. 
	In addition, let $q\in[p,\infty)$, $M\in\bbN$ 
	and $\xi_1,\ldots,\xi_M\in L_q(\Omega;E)$ be 
	independent and identically distributed 
	$E$-valued random variables.  
	Then, 
	\begin{align*} 
	\biggl\| 
	\bbE[\eta] 
	- \frac{1}{M}
	\sum_{j=1}^M 
	\xi_j 
	\biggr\|_{L_q(\Omega;E)} 
	&\leq 
	\bigl\| \bbE[ \eta - \xi_1] \bigr\|_{E}  
	\\
	&\quad 
	+ 
	2 K_{q,p} \tau_p(E) \, 
	M^{-\left( 1-\frac{1}{p} \right)} 
	\bigl\| \xi_1 - \bbE[\xi_1] \bigr\|_{L_q(\Omega; E)}. 
	\end{align*}
\end{corollary} 

\begin{proof}
	By applying the triangle inequality on $L_q(\Omega;E)$ 
	and Proposition~\ref{prop:MC-1} 
	(noting that $\xi_j - \bbE[\xi_1]$, $1\leq j\leq M$, 
	are independent and centered), 
	we find that 
	\begin{align*} 
	\biggl\| 
	\bbE[\eta] - \frac{1}{M} \sum_{j=1}^M \xi_j 
	&\biggr\|_{L_q(\Omega;E)} 
	\leq 
	\bigl\| \bbE[ \eta - \xi_1] \bigr\|_{E}  
	+ 
	\frac{1}{M} \, 
	\biggl\| 
	\sum_{j=1}^M 
	\bigl( \xi_j - \bbE[\xi_1] \bigr)
	\biggr\|_{L_q(\Omega;E)} 
	\\
	&\leq 
	\bigl\| \bbE[ \eta - \xi_1] \bigr\|_{E}  
	+ 
	2 K_{q,p} \tau_p(E) \, 
	M^{-1}
	\Biggl( 
	\sum_{j=1}^M
	\bigl\| \xi_j - \bbE[\xi_1] \bigr\|_{L_q(\Omega; E)}^p 
	\Biggr)^{\nicefrac{1}{p}} \!, 
	\end{align*} 	
	and the claim follows by 
	the identical distribution 
	of $\xi_1,\ldots,\xi_M$. 
\end{proof}

The remainder of this subsection is devoted to generalizing 
the approximation result of Monte Carlo estimation 
for the first statistical moment in Corollary~\ref{cor:MC-1} 
to (injective) statistical moments $\Mk[\eta]$ 
of an arbitrary order $k\in\bbN$. 
Example~\ref{example:type-tensor} shows that 
it is in general not possible to argue via the type 
of the tensor product space. 
We therefore prove the convergence rates of Monte Carlo methods 
directly by means of the auxiliary results derived in  
Subsection~\ref{subsec:auxiliary}. 

\begin{theorem}\label{thm:MC-k} 
	Assume that  
	$(E,\norm{\,\cdot\,}{E})$ 
	is of Rademacher type $p\in[1,2]$. 
	Let $q\in[p,\infty)$, $k,M\in\bbN$ 
	and $\xi_1,\ldots,\xi_M\in L_{kq}(\Omega;E)$ be 
	independent and identically distributed 
	$E$-valued random variables. 
	Then, 
	\begin{equation}\label{eq:thm:MC-k}  	
		\biggl\| 
		\Mk[\xi_1] 
		-
		\frac{1}{M} 
		\sum_{j=1}^M 
		\xk \xi_j 
		\biggr\|_{L_q(\Omega;\xkseps E)} 
		\leq 
		\CSL \,  
		M^{-\left( 1-\frac{1}{p} \right)} 
		\norm{ \xi_1 }{L_{kq}(\Omega; E)}^k .  
	\end{equation} 
	Here, we recall  
	the constant $B_q\in(0,\infty)$   
	from the classical 
	Khintchine inequalities \eqref{eq:khintchine},  
	as well as the Kahane--Khintchine constant $K_{q,p}$   
	and type $p$ constant $\tau_p(E)$  
	from Definitions~\ref{def:Kqp-constant} and~\ref{def:type-p-constant}, 
	respectively, and set 
	\begin{equation}\label{eq:CSL} 
	\CSL := 2 (2 k  K_{q,p} \tau_p(E) + B_q). 
	\end{equation}
\end{theorem} 

\begin{proof} 
	Assume that $(r_j)_{j=1}^M$ is 
	a Rademacher family on a complete probability space 
	$(\Omegat,\cAt,\bbPt)$ 
	and, for $j\in\{1,\ldots,M\}$, 
	let $\xivec_j\from\Omega\times\Omegat\to E$ 
	and $\rvec_j\from\Omega\times\Omegat\to \{-1,1\}$ denote the mappings 
	that satisfy $\xivec_j(\omega,\omegat) = \xi_j(\omega)$ 
	and $\rvec_j(\omega,\omegat) = r_j(\omegat)$ 
	for all $(\omega,\omegat)\in\Omega\times\Omegat$. 
	Notice that 
	on $(\Omega \times \Omegat,\cA \otimes \cAt,\bbP\otimes\bbPt)$ 
	the random variables 
	$(\rvec_j)_{j=1}^M$ form a Rademacher family and 
	$\xivec_1,\ldots,\xivec_M,\rvec_1,\ldots,\rvec_M$ 
	are independent. 
	Moreover,  
	$\xi_j \in L_{kq}(\Omega;E)$ implies that 
	$\Mk[\xi_1] \in \xkseps E$ and 
	$\xk \xivec_j \in L_q(\Omega\times\widetilde{\Omega};\xkseps E)$ 
	are well-defined. 
	We further note that by the identical distribution 
	of $\xi_1,\ldots,\xi_M$, 
	\[
	(\bbE\otimes\bbEt)\bigl[  \xk \xivec_j  -  \Mk[\xi_1] \bigr]  
	= 
	\bbE\bigl[ \xk \xi_j -  \Mk[\xi_1] \bigr]  
	=
	\bbE\bigl[ \xk \xi_j \bigr] -  \Mk[\xi_j]  
	= 
	0. 
	\]  
	This shows that the independent random variables 
	$\xk \xivec_j  -  \Mk[\xi_1]\from \Omega\times\Omegat \to \xkseps E$, 
	$1\leq j\leq M$, 
	are centered. 
	Therefore, 
	we can use Lemma~\ref{lem:symmetrization}  
	on the 
	probability space 
	$(\Omega \times \Omegat,\cA \otimes \cAt,\bbP\otimes\bbPt)$ 
	and for the Banach space $\xkseps E$ to deduce that 
	\[
	\biggl\| 
	\sum_{j=1}^M 
	\bigl( \xk \xivec_j  
	-  \Mk[\xi_1] \bigr)
	\biggr\|_{L_q(\Omega\times\widetilde{\Omega};\xkseps E)} 
	\leq 
	2\, 
	\biggl\| 
	\sum_{j=1}^M 
	\rvec_j 
	\bigl( \xk \xivec_j  
	-  \Mk[\xi_1] \bigr)
	\biggr\|_{L_q(\Omega\times\widetilde{\Omega};\xkseps E)} . 
	\]
	By the triangle inequality 
	on $L_q(\Omega\times\widetilde{\Omega};\xkseps E)$ 
	we then obtain 
	that 
	\begin{equation}\label{eq:proof:thm:MC-k:A-B} 
	\begin{split}  
	\biggl\| 
	\sum_{j=1}^M 
	&\bigl(\xk \xi_j 
	-  \Mk[\xi_1] \bigr)
	\biggr\|_{L_q(\Omega;\xkseps E)} 
	= 
	\biggl\| 
	\sum_{j=1}^M 
	\bigl( \xk \xivec_j  
	-  \Mk[\xi_1] \bigr)
	\biggr\|_{L_q(\Omega\times\widetilde{\Omega};\xkseps E)} 
	\\
	&\leq 
	2\, 
	\biggl\| 
	\sum_{j=1}^M 
	\rvec_j 
	\bigl( \xk \xivec_j  
	-  \Mk[\xi_1] \bigr)
	\biggr\|_{L_q(\Omega\times\widetilde{\Omega};\xkseps E)}  
	\\
	&\leq 
	2\, 
	\biggl\| 
	\sum_{j=1}^M 
	\rvec_j 
	\xk\! \xivec_j  
	\biggr\|_{L_q(\Omega\times\widetilde{\Omega};\xkseps E)}   
	+ 
	2\, 
	\biggl\| 
	\sum_{j=1}^M 
	\rvec_j 
	\Mk[\xi_1] 
	\biggr\|_{L_q(\Omega\times\widetilde{\Omega};\xkseps E)}   
	\\
	&=: 
	2\text{(A)} + 2\text{(B)}. 
	\end{split} 
	\end{equation} 

	To bound term (A) from above, 
	we apply Fubini's theorem and 
	the Kahane--Khintchine inequality~\eqref{eq:kahane-khintchine} 
	for the Banach space $F:=\xkseps E$ 
	and obtain that 
	\begin{align*} 
	\text{(A)} 
	&=
	\Biggl( 
	\int_\Omega 
	\biggl\| 
	\sum_{j=1}^M 
	r_j (\,\cdot\,) 
	\xk\! \xi_j(\omega)  
	\biggr\|_{L_q(\Omegat;\xkseps E)}^q  
	\rd\bbP(\omega)
	\Biggr)^{\nicefrac{1}{q}} 
	\\
	&\leq 
	K_{q,p} 
	\Biggl( 
	\int_\Omega 
	\biggl\| 
	\sum_{j=1}^M 
	r_j (\,\cdot\,) 
	\xk\! \xi_j(\omega) 
	\biggr\|_{L_p(\Omegat;\xkseps E)}^q  
	\rd\bbP(\omega)
	\Biggr)^{\nicefrac{1}{q}} \!. 
	\end{align*} 
	Upon inserting the definition~\eqref{eq:sepsnorm}  
	of the symmetric injective tensor norm,  
	we use Proposition~\ref{prop:rade-gaussian}\ref{prop:rade-gaussian-poly}
	for the convex increasing function $G(t) := t^p$, $t\geq0$,
	and 
	the fact that the Banach space $E$ has type $p\in[1,2]$ 
	with type constant $\tau_p(E)\in(0,\infty)$, 
	to conclude that 
	{\allowdisplaybreaks
	\begin{align*} 
	\text{(A)} 
	&\leq 
	K_{q,p} 
	\Biggl( 
	\int_\Omega 
	\Biggl( 
	\bbEt 
	\Biggl[ 
	\biggl( 
	\sup_{f\in B_{\dual{E}}} 
	\biggl| 
	\sum_{j=1}^M 
	r_j (\,\cdot\,) 
	f\bigl( \xi_j(\omega) \bigr)^k 
	\biggr| \biggr)^p 
	\Biggr] 
	\Biggr)^{\nicefrac{q}{p}} 
	\rd\bbP(\omega)
	\Biggr)^{\nicefrac{1}{q}}  
	\\
	&\leq  
	2 k K_{q,p}  
	\Biggl( 
	\int_\Omega 
	\Biggl( 
	\bbEt 
	\Biggl[ 
	\biggl\| 
	\sum_{j=1}^M 
	r_j (\,\cdot\,) 
	\norm{ \xi_j(\omega) }{E}^{k-1} 
	\xi_j(\omega)
	\biggr\|_E^p 
	\Biggr] 
	\Biggr)^{\nicefrac{q}{p}} 
	\rd\bbP(\omega)
	\Biggr)^{\nicefrac{1}{q}} 
	\\
	\text{(A)} 
	&\leq 
	2 k K_{q,p} \tau_p(E) 
	\Biggl( 
	\int_\Omega 
	\Biggl( 
	\sum_{j=1}^M 
	\norm{ \xi_j(\omega) }{E}^{kp} 
	\Biggr)^{\nicefrac{q}{p}} 
	\rd\bbP(\omega)
	\Biggr)^{\nicefrac{1}{q}} 
	\\
	&= 
	2 k K_{q,p} \tau_p(E) 
	\biggl\| 
	\sum_{j=1}^M 
	\norm{\xi_j}{E}^{kp} 
	\biggr\|_{L_{\nicefrac{q}{p}}(\Omega;\bbR)}^{\nicefrac{1}{p}}. 
	\end{align*} 
	Since} $q\geq p$, 
	we can use the triangle inequality 
	on $L_{\nicefrac{q}{p}}(\Omega;\bbR)$, yielding 
	\begin{equation}\label{eq:proof:thm:MC-k:A} 
	\text{(A)} 
	\leq 
	2 k K_{q,p} \tau_p(E) 
	\Biggl(  
	\sum_{j=1}^M 
	\norm{\xi_j}{L_{kq}(\Omega;E)}^{kp} 
	\biggr)^{\nicefrac{1}{p}} 
	=  
	2 k K_{q,p} \tau_p(E) 
	M^{\nicefrac{1}{p}} 
	\norm{\xi_1}{L_{kq}(\Omega; E)}^k, 
	\end{equation} 
	where we also used the  
	identical distribution of $\xi_1,\ldots,\xi_M$. 
	
	For term (B) we use the estimate  
	\[
	\bigl\| \Mk[\xi_1] \bigr\|_{\eps_s}  
	= 
	\bigl\| \bbE\bigl[ \xk \xi_1 \bigr] \bigr\|_{\eps_s} 
	\leq 
	\bbE\bigl[ \norm{ {\xk} \xi_1 }{\eps_s} \bigr]
	= 
	\bbE\bigl[ \norm{ \xi_1 }{E}^k \bigr] 
	\leq 
	\norm{ \xi_1 }{L_{kq}(\Omega;E)}^k , 
	\]
	as well as the classical 
	Khintchine inequalities \eqref{eq:khintchine}
	so that 
	\[
	\biggl\| 
	\sum_{j=1}^M 
	r_j 
	\biggr\|_{L_q(\widetilde{\Omega};\bbR)} 
	\leq 
	B_q \, M^{\nicefrac{1}{2}} 
	\leq 
	B_q \, M^{\nicefrac{1}{p}} \!, 
	\]
	and conclude that 
	\begin{equation}\label{eq:proof:thm:MC-k:B}   
	\begin{split} 
	\text{(B)}  
	&= 
	\biggl\| 
	\Mk[\xi_1] 
	\sum_{j=1}^M 
	\rvec_j 
	\biggr\|_{L_q(\Omega\times\widetilde{\Omega};\xkseps E)}   
	= 
	\Biggl( 
	\int_{\widetilde{\Omega}} 
	\biggl\| 
	\Mk[\xi_1] 
	\sum_{j=1}^M 
	r_j (\omegat) 
	\biggr\|_{\eps_s}^q   
	\, \rd\bbPt(\omegat) 
	\Biggr)^{\nicefrac{1}{q}} 
	\\
	&=
	\bigl\| 
	\Mk[\xi_1]  
	\bigr\|_{\eps_s} 
	\Biggl( 
	\int_{\widetilde{\Omega}} 
	\biggl| 
	\sum_{j=1}^M 
	r_j(\widetilde{\omega}) 
	\biggr|^q 
	\, \rd\bbPt(\omegat) 
	\Biggr)^{\nicefrac{1}{q}} 
	\leq 
	B_q \, M^{\nicefrac{1}{p}} 
	\norm{ \xi_1 }{L_{kq}(\Omega; E)}^k . 
	\end{split} 
	\end{equation} 
	Finally, combining \eqref{eq:proof:thm:MC-k:A-B}, 
	\eqref{eq:proof:thm:MC-k:A} and 
	\eqref{eq:proof:thm:MC-k:B} shows that 
	\begin{align*} 
	\biggl\| 
	\Mk[\xi_1] 
	-
	\frac{1}{M} 
	\sum_{j=1}^M 
	\xk \xi_j 
	\biggr\|_{L_q(\Omega;\xkseps E)} 
	&=
	M^{-1} \, 
	\biggl\| 
	\sum_{j=1}^M 
	\bigl(\xk \xi_j -  \Mk[\xi_1] \bigr)
	\biggr\|_{L_q(\Omega;\xkseps E)} 
	\\
	&\leq 
	2 
	( 2k K_{q,p} \tau_p(E) + B_q) \, 
	M^{-\left(1-\frac{1}{p}\right)} 
	\norm{ \xi_1 }{L_{kq}(\Omega;E)}^k, 
	\end{align*} 
	which along with 
	the definition \eqref{eq:CSL} 
	of~$\CSL$ 
	completes the proof 
	of \eqref{eq:thm:MC-k}. 
\end{proof} 

The estimate~\eqref{eq:proof:thm:MC-k:A} 
of term (A) in the proof of Theorem~\ref{thm:MC-k} 
reveals the following analogue of 
Proposition~\ref{prop:MC-1} 
for independent (not necessarily identically distributed) random variables 
$\eta_1,\ldots,\eta_M\in L_{kq}(\Omega;E)$ 
with vanishing $k$th moment. 
	
\begin{corollary} 
	Assume that  
	$(E,\norm{\,\cdot\,}{E})$ 
	is of Rademacher type $p\in[1,2]$. 
	Let $q\in[p,\infty)$, $k,M\in\bbN$ 
	and $\eta_1,\ldots,\eta_M\in L_{kq}(\Omega;E)$ be 
	independent $E$-valued random variables 
	with vanishing $k$th moment, 
	$\Mk[\eta_j] = 0$ for all $1\leq j\leq M$. 
	Then, 
	\begin{equation}\label{eq:0-kth-moment}
		\biggl\| 
		\sum_{j=1}^M 
		\xk \eta_j 
		\biggr\|_{L_q(\Omega;\xkseps E)} 
		\leq 
		4 k K_{q,p} \tau_p(E) \,  
		\Biggl( 
		\sum_{j=1}^M
		\norm{ \eta_j }{L_{kq}(\Omega; E)}^{kp}  
		\Biggr)^{\nicefrac{1}{p}} \!. 
	\end{equation} 
\end{corollary} 
	
\begin{remark} 
	Optimality of the convergence rate 
	$1-\frac{1}{p}$ in \eqref{eq:thm:MC-k} 
	is ultimately related to the question  
	whether it is \emph{necessary} that the Banach space~$E$ 
	has Rademacher type $p\in[1,2]$ 
	for \eqref{eq:0-kth-moment} 
	to hold for all finite sequences 
	${\eta_1,\ldots,\eta_M\in L_{kq}(\Omega;E)}$ 
	of independent $E$-valued random variables 
	with vanishing $k$th moment. 
	 
	For the first moment, $k=1$, it is evident that 
	the choice $\eta_j := r_j x_j$ 
	in \eqref{eq:0-kth-moment}, 
	where $(r_j)_{j\in\bbN}$ is a Rademacher family 
	on $(\Omega,\cA,\bbP)$  
	and  $x_1,x_2,\ldots\in E$, 
	implies that the Banach space~$E$ 
	has Rademacher type $p$. 
	However, for higher-order moments, 
	this question is  
	more complex due to the injective 
	tensor norm on the left-hand side. 
	For odd orders $k\in\bbN$ and 
	the space~$E:=\ell_1$ 
	of summable real-valued sequences 
	(which only has Rademacher 
	type $p = 1$), 
	the choice $\eta_j := r_j e_j$ 
	shows that \eqref{eq:0-kth-moment} 
	cannot hold for any $p > 1$. 
	Here, $(e_j)_{j\in\bbN}$ denote the 
	standard unit vectors in $\ell_1$.  
	In addition, 
	the classical Khintchine inequalities \eqref{eq:khintchine} 
	imply that, for any Banach space $E$,
	the convergence rate in \eqref{eq:thm:MC-k} 
	cannot be better that $\nicefrac{1}{2}$ 
	(one may take, e.g., $\xi_j := g_j \, x$, 
	where $(g_j)_{j\in\bbN}$ is an 
	orthogaussian family on $(\Omega,\cA,\bbP)$ 
	and $x\neq 0$ is a non-zero vector in~$E$). 
	Sharpness of the rate 
	$1-\frac{1}{p}$ in \eqref{eq:thm:MC-k} 
	and necessity of the Rademacher type~$p$ 
	for~\eqref{eq:0-kth-moment} 
	for the case $p\in(1,2)$ 
	remains an open question. 
\end{remark}

The next lemma complements Theorem~\ref{thm:MC-k} 
when deriving convergence rates of   
single-level Monte Carlo methods 
for approximating injective $k$th moments. 

\begin{lemma}\label{lem:Mk-diff} 
	Let $k\in\bbN$ and 
	suppose that $\eta,\xi\in L_k(\Omega;E)$. Then, 
	\[
	\bigl\| \Mk[\eta] - \Mk[\xi] \bigr\|_{\eps_s}  
	\leq 
	\bigl\| \Mk[\eta] - \Mk[\xi] \bigr\|_{\eps} 
	\leq 
	\norm{ \eta - \xi }{L_k(\Omega;E)} 
	\sum_{i=0}^{k-1} 
	\Bigl[ 
	\norm{ \eta }{L_k(\Omega;E)}^i 
	\norm{ \xi }{L_k(\Omega;E)}^{k-i-1} 
	\Bigr] . 
	\]
\end{lemma} 

\begin{proof} 
	The first inequality of the assertion is trivial. 
	We next note that also the remaining relation 
	is evident in the case $k=1$, since  
	\[
	\bigl\| \Mk[\eta] - \Mk[\xi] \bigr\|_{\eps} 
	= 
	\bigl\| \bbE[\eta - \xi] \bigr\|_{E} 
	\leq 
	\bbE\bigl[ \norm{ \eta - \xi }{E} \bigr] 
	= 
	\norm{ \eta - \xi }{L_k(\Omega;E)} 
	\quad 
	\text{if}
	\quad  
	k=1. 
	\]
	We now assume that $k\geq 2$ and 
	observe that 
	\begin{align*} 
	{\xk} \eta 
	- 
	{\xk} \xi 
	&= 
	\sum_{i=0}^{k-1} 
	\Bigl[ \bigl( \otimes^{i+1} \eta \bigr) \otimes \bigl( \otimes^{k-(i+1)} \xi \bigr) 
	- 
	\bigl( \otimes^i \eta\bigr) \otimes \bigl( \otimes^{k-i} \xi \bigr) \Bigr]
	\\
	&= 
	\sum_{i=0}^{k-1} 
	\Bigl[ 
	\bigl( \otimes^i \eta\bigr) \otimes (\eta-\xi) 
	\otimes \bigl( \otimes^{k-(i+1)} \xi \bigr)\Bigr] . 
	\end{align*} 
	Therefore, we may estimate as follows, 
	\begin{align*} 
	&\bigl\| \Mk[\eta] - \Mk[\xi] \bigr\|_{\eps}  
	= 
	\bigl\| \bbE\bigl[ {\xk} \eta - {\xk} \xi \bigr] \bigr\|_{\eps}
	\leq 
	\bbE\Bigl[ \bigl\| {\xk} \eta - {\xk} \xi \bigr\|_{\eps} \Bigr]
	\\
	&\leq
	\bbE\Biggl[ \, 
	\sum_{i=0}^{k-1} 
	\bigl\|
	\bigl( \otimes^i \eta\bigr) \otimes (\eta-\xi) \otimes \bigl( \otimes^{k-i-1} \xi \bigr) 
	\bigr\|_{\eps} \Biggr] 
	= 
	\sum_{i=0}^{k-1} 
	\bbE\Bigl[ 
	\norm{ \eta }{E}^i \norm{ \eta -\xi }{E} \norm{ \xi }{E}^{k-i-1} 
	\Bigr]. 
	\end{align*} 
	Combined with the H\"older inequality 
	this completes the proof, 
	since 
	\begin{align*}
	\bbE\Bigl[ \norm{ \eta -\xi }{E} \norm{ \xi }{E}^{k-1} \Bigr] 
	&\leq 
	\bigl( \bbE\bigl[ \norm{ \eta -\xi }{E}^k \bigr] \bigr)^{\frac{1}{k}}  
	\bigl( \bbE\bigl[ \norm{ \xi }{E}^k \bigr] \bigr)^{\frac{k-1}{k}}  , 
	\\ 
	\bbE\Bigl[ \norm{ \eta }{E}^{k-1} \norm{ \eta -\xi }{E} \Bigr] 
	&\leq 
	\bigl( \bbE\bigl[ \norm{ \eta -\xi }{E}^k \bigr] \bigr)^{\frac{1}{k}}  
	\bigl( \bbE\bigl[ \norm{ \eta }{E}^k \bigr] \bigr)^{\frac{k-1}{k}}  , 
	\end{align*} 
	and, whenever $k\geq 3$, 
	we obtain for all $i\in\{1,\ldots,k-2\}$ 
	\begin{align*} 
	\bbE\Bigl[ \norm{ \eta }{E}^i \norm{ \eta -\xi }{E} \norm{ \xi }{E}^{k-i-1} \Bigr] 
	\leq 		
	\bigl( \bbE\bigl[ \norm{ \eta }{E}^k \bigr] \bigr)^{\frac{i}{k}}  
	\bigl( \bbE\bigl[ \norm{ \eta -\xi }{E}^k \bigr] \bigr)^{\frac{1}{k}}  
	\bigl( \bbE\bigl[ \norm{ \xi }{E}^k \bigr] \bigr)^{\frac{k-i-1}{k}}  
	\end{align*} 
	by a triple H\"older inequality 
	with 
	$\bigl(\tfrac{k}{i}\bigr)^{-1} 
	+ 
	k^{-1}
	+ 
	\bigl(\tfrac{k}{k-i-1}\bigr)^{-1} 
	= 
	1$. 
\end{proof} 

We are now ready to state the main result of this subsection: 
an abstract convergence rate bound in $L_q(\Omega; \xkseps E)$ 
for single-level Monte Carlo 
estimation of the injective $k$th moment~$\Mk[\eta]$, 
assuming at our disposal $M$ independent  
samples of an approximation 
$\xi_1\in L_{kq}(\Omega;E)$ to $\eta\in L_k(\Omega;E)$. 

\begin{corollary}\label{cor:MC-k}
	Assume that  
	$(E,\norm{\,\cdot\,}{E})$ 
	is of Rademacher type $p\in[1,2]$. 
	Let $q\in[p,\infty)$, $k,M\in\bbN$ 
	and $\xi_1,\ldots,\xi_M\in L_{kq}(\Omega;E)$ be 
	independent and identically distributed 
	$E$-valued random variables. 
	Then, for every $\tenvec \in \xkseps E$, we have 
	\[
	\biggl\| 
	\tenvec 
	-
	\frac{1}{M} 
	\sum_{j=1}^M 
	\xk \xi_j 
	\biggr\|_{L_q(\Omega;\xkseps E)} 
	\leq 
	\bigl\| \tenvec - \Mk[\xi_1] \bigr\|_{\eps_s}
	+ 
	\CSL \,  
	M^{-\left( 1-\frac{1}{p} \right)} 
	\norm{ \xi_1 }{L_{kq}(\Omega; E)}^k 
	,
	\]
	where the constant $\CSL$ is defined as in \eqref{eq:CSL}. 
	
	In particular, 
	for all $\eta\in L_k(\Omega;E)$,
	we have   
	\begin{align*} 
	\biggl\| 
	\Mk[\eta]
	-
	\frac{1}{M} 
	\sum_{j=1}^M 
	\xk \xi_j 
	\biggr\|_{L_q(\Omega;\xkseps E)} 
	&\leq 
	\norm{ \eta - \xi_1 }{L_k(\Omega;E)} 
	\sum_{i=0}^{k-1} 
	\Bigl[ \norm{ \eta }{L_k(\Omega;E)}^i \norm{ \xi_1 }{L_k(\Omega;E)}^{k-i-1} \Bigr] 
	\\
	&\quad + 
	\CSL \,  
	M^{-\left( 1-\frac{1}{p} \right)} 
	\norm{ \xi_1 }{L_{kq}(\Omega; E)}^k 
	.
	\end{align*}
\end{corollary} 

\begin{proof} 
	By the triangle inequality on 	
	$L_q(\Omega;\xkseps E)$ we obtain,  
	for every $\tenvec\in\xkseps E$, 
	\begin{align*}
	\biggl\| 
	\tenvec -
	\frac{1}{M} 
	\sum_{j=1}^M 
	\xk \xi_j 
	\biggr\|_{L_q(\Omega;\xkseps E)} 
	&\leq 
	\bigl\| \tenvec - \Mk[\xi_1] \bigr\|_{\eps_s} 
	\\
	&\quad +
	\biggl\| 
	\Mk[\xi_1] 
	-
	\frac{1}{M} 
	\sum_{j=1}^M 
	\xk \xi_j 
	\biggr\|_{L_q(\Omega;\xkseps E)} , 
	\end{align*} 
	and the first claim follows 
	by applying the estimate \eqref{eq:thm:MC-k} 
	of Theorem~\ref{thm:MC-k}. 
	Subsequently, we derive the second assertion 
	by combining this result with Lemma~\ref{lem:Mk-diff}
	which we use to bound  
	the difference of the $k$th moments 
	$\bigl\| \Mk[\eta] - \Mk[\xi_1] \bigr\|_{\eps_s}$. 
\end{proof} 

We close this subsection with a counterexample 
which shows that the convergence results for the standard  
and single-level Monte Carlo estimators in 
Theorem~\ref{thm:MC-k} and Corollary~\ref{cor:MC-k}
can, in general, not hold when measuring the error 
in the (symmetric or full) projective tensor norm. 
More specifically, we discuss 
this for 
second-order moments 
of Hilbert space valued random variables, i.e., 
the random variables take values in a Banach space 
of type $p=2$. 

\begin{example}\label{example:projective-bad} 
	Let $(H, \scalar{\,\cdot\,, \,\cdot\,}{H})$ 
	be a real separable Hilbert space, 
	and let $(e_j)_{j\in\bbN}$ be an orthonormal basis for~$H$. 
	For $n\in\bbN$, consider the $H$-valued random variable 
	$\xi_n\from \Omega\to H$ on $(\Omega,\cA,\bbP)$, 
	whose (discrete uniform) distribution 
	is defined by 
	\[
	\forall i \in \{1,\ldots,n\}: 
	\quad 
	\bbP( \{\omega\in\Omega : \xi_n(\omega) = e_i\} ) 
	= n^{-1}\!. 
	\]
	Then, for all $n\in\bbN$, 
	both the projective and injective second moments 
	of $\xi_n$ exist,   
	\[
	\bbM^2_\pi[\xi_n] 
	= 
	\bbM^2_\eps[\xi_n] 
	= 
	\bbE\bigl[\xi_n\otimes \xi_n\bigr] 
	= 
	\frac{1}{n} \sum_{i=1}^n e_i \otimes e_i, 
	\quad 
	\bigl\| \bbM^2_\pi[\xi_n] \bigr\|_{\pi}
	= 
	\bigl\| \bbM^2_\pi[\xi_n] \bigr\|_{\pi_s}
	=
	1, 
	\]
	see \eqref{eq:diagonal-pi-norms} in Lemma~\ref{lem:diagonal-norms} 
	of Appendix~\ref{appendix:norm-computation} for the  
	norm identities. 
	
	In addition,  
	for  all $q\in[1,\infty)$ 
	and every 
	$n\in\bbN$,  
	we have that 
	\[
	\norm{\xi_n}{L_q(\Omega;H)}^q 
	= 
	\bbE\bigl[ \norm{\xi_n}{H}^q \bigr] 
	= 
	\frac{1}{n} 
	\sum_{i=1}^n 
	\norm{e_i}{H}^q  
	= 
	1. 
	\]
	We let $q\in[1,\infty)$, 
	$\xi_{n,1},\ldots,\xi_{n,M}$ be $M\in\bbN$ independent 
	copies of $\xi_n$ and estimate 
	\begin{align*} 
	\mathrm{err}_{q,\pi}^{(n)} 
	:=
	\biggl\| 
	\bbM^2_\pi [\xi_n] 
	&-
	\frac{1}{M} 
	\sum_{j=1}^M 
	\otimes^2 \xi_{n,j} 
	\biggr\|_{L_q(\Omega; \otimes_\pi^2 H)}^q 
	= 
	\bbE
	\Biggl[ 
	\biggl\| 
	\bbM^2_\pi [\xi_n] 
	- 
	\frac{1}{M} 
	\sum_{j=1}^M 
	\otimes^2 \xi_{n,j} 
	\biggr\|_\pi^q 
	\Biggr] 
	\\
	&= 
	\sum_{\nu_1 = 1}^n  
	\cdots
	\sum_{\nu_M = 1}^n 
	\frac{1}{n^M} 
	\biggl\| \, 
	\frac{1}{n} 
	\sum_{i=1}^n 
	( e_i \otimes e_i )
	- 
	\frac{1}{M} 
	\sum_{j=1}^M 
	\bigl( e_{\nu_j} \otimes e_{\nu_j} \bigr)
	\biggr\|_\pi^q 
	\\
	&\geq 
	\sum_{\substack{1 \leq \nu_1,\ldots,\nu_M \leq n \\ \text{pairwise distinct}}} 
	\frac{1}{n^M} 
	\biggl\| \, 
	\frac{1}{n} 
	\sum_{i=1}^n 
	( e_i \otimes e_i )
	- 
	\frac{1}{M} 
	\sum_{j=1}^M 
	\bigl( e_{\nu_j} \otimes e_{\nu_j} \bigr)
	\biggr\|_\pi^q . 
	\end{align*} 
	Thus, 
	assuming that $n\geq M$, 
	again by \eqref{eq:diagonal-pi-norms} 
	in Lemma~\ref{lem:diagonal-norms} 
	we obtain that 
	\begin{align*} 
	\mathrm{err}_{q,\pi}^{(n)} 
	&\geq 
	\sum_{\substack{1 \leq \nu_1,\ldots,\nu_M \leq n \\ \text{pairwise distinct}}} 
	n^{-M} 
	\bigl[ 
	M \bigl( \tfrac{1}{M} - \tfrac{1}{n} \bigr) 
	+ 
	(n-M) \tfrac{1}{n} 
	\bigr]^q 
	\\
	&= 
	2^q  
	\bigl( 1 - \tfrac{M}{n} \bigr)^q 
	n^{-M} 
	[ n \cdots (n-M+1) ] 
	\geq 
	2^q 
	\bigl( 1 - \tfrac{M}{n} \bigr)^{q+M} \! . 
	\end{align*} 

	Given $q\in[1,\infty)$ and 
	$M\in\bbN$, we choose an integer 
	$n_\star = n_\star(q,M)\in \bbN$ 
	such that 
	\[
	n_\star 
	\geq 
	M \bigl( 1 - 2^{-\nicefrac{q}{(q+M)}}  \bigr)^{-1} 
	\qquad 
	\Longrightarrow 
	\qquad 
	\bigl( 1 - \tfrac{M}{n_\star} \bigr)^{q+M}  
	\geq 2^{-q} . 
	\]
	This proves that, 
	for all $q\in[1,\infty)$ and 
	every $M\in\bbN$, there exists 
	${n_\star=n_\star(q,M)\in\bbN}$ 
	such that 
	\[
	\biggl\| 
	\bbM^2_\pi [ \xi_{n_\star} ] 
	-
	\frac{1}{M} 
	\sum_{j=1}^M 
	\otimes^2 \xi_{n_\star,j} 
	\biggr\|_{L_q(\Omega; \otimes_{\pi_s}^{2,s} H)} 
	\geq 
	\biggl\| 
	\bbM^2_\pi [ \xi_{n_\star} ] 
	-
	\frac{1}{M} 
	\sum_{j=1}^M 
	\otimes^2 \xi_{n_\star,j} 
	\biggr\|_{L_q(\Omega; \otimes_\pi^2 H)} 
	\geq 1 . 
	\]
	Since $\norm{\xi_{n_\star}}{L_q(\Omega;H)} = 1$ 
	is also true 
	for all $q\in[1,\infty)$ and 
	since $H$ has type $p=2$, 
	this shows that an analogue 
	of \eqref{eq:thm:MC-k} cannot hold with respect to 
	the (full or symmetric) 
	projective tensor norm. 
\end{example}

\subsection{Multilevel Monte Carlo estimation} 
\label{subsec:mlmc} 

Assuming that $(X_\ell)_{\ell=1}^L$ is 
a family of $E$-valued random variables 
corresponding to $L\in\bbN$ 
different refinement \emph{levels}  
of underlying discretization parameters, 
translating the idea of multilevel Monte Carlo (MLMC)
estimation, as formulated 
e.g.\ in~\cite[p.~5]{CliffeGilesScheichlTeckentrup2011} 
for means of 
Hilbert space valued random variables, 
to higher-order moments of Banach space 
valued random variables 
results in exploiting the telescopic sum 
(here: $X_0:=0\in E$)
\[
\bbE\bigl[ {\xk} X_L \bigr]
= 
\Mk[X_L] 
= 
\sum_{\ell=1}^L 
\bigl( \Mk[X_\ell] - \Mk[X_{\ell-1}] \bigr) 
= 
\sum_{\ell=1}^L 
\bbE\bigl[ {\xk} X_{\ell} - {\xk} X_{\ell-1} \bigr], 
\]
and estimating $\bbE[ {\xk} X_{\ell} - {\xk} X_{\ell-1} ]$ 
for each $1\leq\ell\leq L$ 
via Monte Carlo sampling
instead of $\bbE[ {\xk} X_L ]$. 
As we will see in Theorem~\ref{thm:alpha-beta-gamma} 
and Remark~\ref{rem:comparison-SLMC}, 
this approach considerably reduces the computational cost. 

Evidently, the corresponding error analysis 
requires a Monte Carlo convergence result for 
estimating differences of injective $k$th moments, i.e., 
for expected values of the form 
$\bbE[{\xk}\eta - {\xk}\xi] = \Mk[\eta] - \Mk[\xi]$, 
via standard Monte Carlo methods. 
This auxiliary result is derived in Proposition~\ref{prop:MLMC-k} 
by means of the next lemma, Lemma~\ref{lem:MLMC-k}, 
which acts as the analogue of the Rademacher 
type estimate \eqref{eq:type} for Rademacher sums 
of differences 
${\xk} x_j - {\xk} y_j$, $1\leq j\leq M$. 

\begin{lemma}\label{lem:MLMC-k}
	Let $(r_j)_{j=1}^M$ be   
	a Rademacher family 
	on a complete probability space 
	$(\Omegat,\cAt,\bbPt)$ 
	with expectation $\bbEt$.  
	Assume that  
	$(E,\norm{\,\cdot\,}{E})$ 
	has Rademacher type ${p\in[1,2]}$, and 
	let ${q\in[p,\infty)}$, $k,M\in\bbN$ and 
	$x_1,\ldots,x_M, y_1,\ldots, y_M \in E$. 
	Then,  
	\begin{equation}\label{eq:lem:MLMC-k} 
	\biggl\| 
	\sum_{j=1}^M 
	r_j \bigl( \xk x_j - \xk y_j \bigr) 
	\biggr\|_{L_q(\Omegat;\xkseps E)}  
	\leq 
	\Cdiff 
	\sum_{i=1}^{k}
	\binom{k}{i} 
	\Biggl[  
	\sum_{j=1}^M 
	\norm{ x_j - y_j }{E}^{ip} \norm{ y_j }{E}^{(k-i)p}  
	\Biggr]^{\nicefrac{1}{p}} \!. 
	\end{equation}
	Here,  
	$\binom{k}{i} := \frac{k!}{i!(k-i)!}$ 
	is the binomial coefficient and 
	the constant $\Cdiff$ is given by 
	\begin{equation}\label{eq:Cdiff}
	\Cdiff := 16 k \sqrt{\pi} K_{q,p} \tau_{p}(E) K_{q,2}, 
	\end{equation}
	with the Kahane--Khintchine and type $p$ constants 
	from Definitions~\ref{def:Kqp-constant} 
	and~\ref{def:type-p-constant}. 
\end{lemma} 

\begin{proof} 
	For  the proof of \eqref{eq:lem:MLMC-k}, we 
	assume that $(g_j)_{j=1}^M$, $(\widetilde{g}_j)_{j=1}^M$  
	are two independent orthogaussian families on 
	$(\Omegat,\cAt,\bbPt)$. 
	We first note that 
	by \cite[Lemma~4.5 and (4.8)]{LedouxTalagrand2011}, 
	applied for the convex function $t\mapsto t^q$ 
	and the Banach space $\xkseps E$, 
	and by the definition~\eqref{eq:sepsnorm} 
	of the symmetric injective 
	tensor norm we have that 
	\begin{equation}\label{eq:proof:lem:MLMC-k-1} 
	\begin{split} 
	\biggl\|  
	\sum_{j=1}^M 
	r_j \bigl( \xk x_j &- \xk y_j \bigr) 
	\biggr\|_{L_q(\Omegat;\xkseps E)} 
	\leq 
	\sqrt{\frac{\pi}{2}} \, 
	\biggl\| 
	\sum_{j=1}^M 
	g_j \bigl( \xk x_j - \xk y_j \bigr) 
	\biggr\|_{L_q(\Omegat;\xkseps E)} 
	\hspace{-1cm}
	\\
	&= 
	\sqrt{\frac{\pi}{2}} \, 
	\Biggl\| 
	\sup_{f\in B_{\dual{E}}} 
	\biggl| 
	\sum_{j=1}^M 
	g_j \bigl( f(x_j)^k - f(y_j)^k \bigr) 
	\biggr|  
	\Biggr\|_{L_q(\Omegat;\bbR)}
	\\
	&= 
	\sqrt{\frac{\pi}{2}} \,
	\Biggl\| 
	\sup_{f\in B_{\dual{E}}} 
	\biggl| 
	\sum_{j=1}^M 
	g_j 
	\sum_{i=1}^k
	\binom{k}{i} 
	f(x_j-y_j)^i f(y_j)^{k-i}  
	\biggr| 
	\Biggr\|_{L_q(\Omegat; \bbR)}
	\\
	&\leq 
	\sum_{i=1}^k
	\sqrt{\frac{\pi}{2}}  
	\binom{k}{i} 
	\Biggl\| 
	\sup_{f\in B_{\dual{E}}} 
	\biggl| 
	\sum_{j=1}^M 
	g_j 
	f(x_j-y_j)^i f(y_j)^{k-i}  
	\biggr| 
	\Biggr\|_{L_q(\Omegat; \bbR)} \! . 
	\end{split} 
	\end{equation}
	Here, we also used the binomial expansion 
	for $f(x_j)^k = [f(x_j-y_j) + f(y_j)]^k$ 
	and the triangle inequality on $L_q(\Omegat;\bbR)$. 
	We now claim that, for every $i\in\{1,\ldots,k-1\}$, 
	\begin{equation}\label{eq:proof:lem:MLMC-k-2} 
	\begin{split}  
	&\Biggl\| 
	\sup_{f\in B_{\dual{E}}} 
	\biggl| 
	\sum_{j=1}^M 
	g_j 
	f(x_j-y_j)^i f(y_j)^{k-i}  
	\biggr| 
	\Biggr\|_{L_q(\Omegat;\bbR)}
	\\
	&\leq  
	2\sqrt{2} \, 
	\Biggl\| 
	\sup_{f\in B_{\dual{E}}} 
	\biggl| 
	\sum_{j=1}^M 
	\bigl( 
	g_j f(x_j - y_j)^i \norm{y_j}{E}^{k-i}
	+ 
	\widetilde{g}_j \norm{x_j - y_j}{E}^i f(y_j)^{k-i} 
	\bigr) 
	\biggr| 
	\Biggr\|_{L_q(\Omegat; \bbR)} \!.  
	\hspace{-1cm}
	\end{split} 
	\end{equation} 
	
	To establish \eqref{eq:proof:lem:MLMC-k-2}, 
	we set $\delta_j := x_j - y_j\in E$ for all $1\leq j \leq M$, and 
	consider for a fixed $i\in\{1,\dots,k-1\}$ the following two real-valued 
	centered Gaussian processes 
	$\cG_{i,1}, \cG_{i,2} \from B_{\dual{E}}\! \times \Omegat \to \bbR$, 
	which are 
	indexed by $f\in B_{\dual{E}}$, 
	\begin{align} 
	\cG_{i,1}(f)
	&:= 
	\sum_{j=1}^M 
	g_j 
	f( \delta_j )^i f( y_j )^{k-i} , 
	\label{eq:Gi1}
	\\
	\cG_{i,2}(f)
	&:= 
	\sqrt{2} \, 
	\sum_{j=1}^M 
	\bigl( 
	g_j 
	f( \delta_j )^i \norm{ y_j }{E}^{k-i}  
	+ 
	\widetilde{g}_j 
	\norm{ \delta_j }{E}^i f( y_j )^{k-i}   
	\bigr). 
	\label{eq:Gi2} 
	\end{align} 
	For all $i\in\{1,\dots,k-1\}$ and 
	every $f,h\in B_{\dual{E}}$, 
	we then obtain by independence of the 
	standard Gaussian random variables 
	$g_1, \ldots, g_M, \widetilde{g}_1, \ldots, \widetilde{g}_M$ 
	the following estimate,  
	\begin{align*} 
	\bbEt \bigl[ |\cG_{i,1}(f) &- \cG_{i,1}(h)|^2 \bigr] 
	=  
	\sum_{j=1}^M 
	\bigl( f( \delta_j )^i f( y_j )^{k-i}   
	- 
	h( \delta_j )^i h( y_j )^{k-i} \bigr)^2 
	\\
	&= 
	\sum_{j=1}^M 
	\bigl( \bigl[ f( \delta_j )^i - h( \delta_j )^i \bigr] f( y_j )^{k-i}   
	+ 
	h( \delta_j )^i \bigl[ f( y_j )^{k-i} - h( y_j )^{k-i} \bigr] \bigr)^2 
	\\
	&\leq 
	2
	\sum_{j=1}^M 
	\Bigl( 
	\bigl[ f( \delta_j )^i - h( \delta_j )^i \bigr]^2 f( y_j )^{2(k-i)}   
	+ 
	h( \delta_j )^{2i} \bigl[ f( y_j )^{k-i} - h( y_j )^{k-i}\bigr]^2 
	\Bigr) 
	\\
	&\leq 
	2
	\sum_{j=1}^M 
	\Bigl( 
	\bigl[ f( \delta_j )^i - h( \delta_j )^i \bigr]^2 \norm{ y_j }{E}^{2(k-i)}   
	+ 
	\norm{ \delta_j }{E}^{2i} \bigl[ f( y_j )^{k-i} - h( y_j )^{k-i} \bigr]^2 
	\Bigr) 
	\\
	&= 
	\bbEt\bigl[ | \cG_{i,2}(f) - \cG_{i,2}(h) |^2 \bigr]. 
	\end{align*} 
	Furthermore, for every $i\in\{1,\dots,k-1\}$, 
	we have 
	\[
	\cG_{i,1}(f) 
	= 
	\sum_{j=1}^M g_j \Psi_i \bigl( f(\delta_j), f(y_j) \bigr), 
	\quad 
	f\in B_{\dual{E}\!}. 
	\]
	Here, $(t_1,t_2)\mapsto \Psi_i(t_1,t_2) := t_1^i \, t_2^{k-i}$ 
	is continuous on $\bbR^2$ 
	and satisfies $\Psi_i(0,0) = 0$ 
	for all $i\in\{1,\dots,k-1\}$. 
	
	We thus may apply the comparison result 
	derived in Lemma~\ref{lem:slepian} 
	(see Appendix~\ref{appendix:details-fernique}) 
	for every ${i\in\{1,\ldots,k-1\}}$, 
	which shows that, 
	for all $q\in[1,\infty)$, 
	\begin{align*} 
	\bbEt\Bigl[
	\bigl(  
	\sup\nolimits_{f \in B_{\dual{E}} \!} 
	|\cG_{i,1}(f)| \bigr)^q 
	\Bigr] 
	\leq 
	2^q \, 
	\bbEt\Bigl[ 
	\bigl( \sup\nolimits_{f \in B_{\dual{E}} \!} 
	|\cG_{i,2}(f) | \bigr)^q \Bigr] . 
	\end{align*} 
	Taking the $q$th root 
	on both sides of this inequality 
	and inserting the definitions 
	of $\cG_{i,1}$ and $\cG_{i,2}$ 
	from \eqref{eq:Gi1}--\eqref{eq:Gi2} 
	completes the proof 
	of \eqref{eq:proof:lem:MLMC-k-2}. 
	
	Next, combining \eqref{eq:proof:lem:MLMC-k-1} 
	with \eqref{eq:proof:lem:MLMC-k-2} and the triangle 
	inequality on $L_q(\Omegat;\bbR)$ yields 
	\begin{align*} 
	&\biggl\|  
	\sum_{j=1}^M 
	r_j \bigl( \xk x_j - \xk y_j \bigr) 
	\biggr\|_{L_q(\Omegat;\xkseps E)} 
	\leq 
	\sum_{i=1}^{k-1}
	\widetilde{C}_{k,i} \, 
	\Biggl\| 
	\sup_{f\in B_{\dual{E}}} 
	\biggl| 
	\sum_{j=1}^M 
	g_j f( \delta_j )^i \norm{y_j}{E}^{k-i} 
	\biggr| 
	\Biggr\|_{L_q(\Omegat; \bbR)}  
	\\
	&+ 
	\sqrt{\frac{\pi}{2}} \, 
	\Biggl\| 
	\sup_{f\in B_{\dual{E}}} 
	\biggl| 
	\sum_{j=1}^M 
	g_j f( \delta_j )^k  
	\biggr| 
	\Biggr\|_{L_q(\Omegat; \bbR)} 
	\hspace*{-1.06mm} +
	\sum_{i=1}^{k-1}
	\widetilde{C}_{k,i} \, 
	\Biggl\| 
	\sup_{f\in B_{\dual{E}}} 
	\biggl| 
	\sum_{j=1}^M 
	\widetilde{g}_j \norm{ \delta_j }{E}^i f(y_j)^{k-i} 
	\biggr| 
	\Biggr\|_{L_q(\Omegat; \bbR)} 
	\!, 
	\end{align*} 
	where we set 
	$\widetilde{C}_{k,i}:=2 \sqrt{ \pi }  \binom{k}{i}$. 
	By noting that 
	$\sqrt{\pi}/ \sqrt{2} \leq \widetilde{C}_{k,k}$ 
	and estimating 
	the $L_q(\Omegat;\bbR)$-norms 
	on the right-hand side for all $i\in\{1,\ldots,k\}$ 
	using 
	Proposition~\ref{prop:rade-gaussian}\ref{prop:rade-gaussian-poly}, 
	with $k_j = i$ 
	and the vectors 
	$\norm{ y_j }{E}^{\frac{k-i}{i}} \delta_j$ 
	(respectively,   
	for every $i\in\{1,\ldots,k-1\}$ 
	with 
	$k_j = k-i$ and 
	$\norm{ \delta_j }{E}^{\frac{i}{k-i}} y_j$) 
	for all $1\leq j \leq M$, 
	we find that 
	\begin{align*} 
	\biggl\|  
	\sum_{j=1}^M 
	r_j \bigl( \xk x_j - \xk y_j \bigr) 
	&\biggr\|_{L_q(\Omegat;\xkseps E)} 
	\leq 
	\sum_{i=1}^k
	\widetilde{C}_{k,i} \, 
	4 i \, 
	\biggl\| 
	\sum_{j=1}^M 
	g_j \norm{ \delta_j }{E}^{i-1} \norm{y_j}{E}^{k-i} \delta_j 
	\biggr\|_{L_q(\Omegat; E)}  
	\\
	&+ 
	\sum_{i=1}^{k-1} 
	\widetilde{C}_{k,i} \, 
	4 (k-i) \,
	\biggl\| 
	\sum_{j=1}^M 
	\widetilde{g}_j \norm{ \delta_j }{E}^i \norm{ y_j }{E}^{k-i-1} y_j  
	\biggr\|_{L_q(\Omegat; E)} 
	\!. 
	\end{align*} 
	Finally, since $q\in[p,\infty)$ is assumed, 
	we may use Proposition~\ref{prop:MC-1} 
	for the independent, centered
	$E$-valued random variables 
	\[
	\eta_j 
	:= 
	g_j \norm{ \delta_j }{E}^{i-1} \norm{y_j}{E}^{k-i} \delta_j 
	\quad \text{resp.} \quad 
	\widetilde{\eta}_j 
	:=
	\widetilde{g}_j \norm{ \delta_j }{E}^i \norm{ y_j }{E}^{k-i-1} y_j  , 
	\qquad 
	1\leq j\leq M, 
	\]
	to conclude that 
	(recall the definitions 
	$\widetilde{C}_{k,i} = 2 \sqrt{ \pi }  \binom{k}{i}$ 
	and 
	$\delta_j = x_j - y_j$) 
	\begin{align*} 
	\Biggl\|  
	\sum_{j=1}^M 
	r_j &\bigl( \xk x_j - \xk y_j \bigr) 
	\biggr\|_{L_q(\Omegat;\xkseps E)} 
	\\
	&\leq 
	\sum_{i=1}^k 
	\left[ 
	8 k \sqrt{\pi} \binom{k}{i} \, 
	2 K_{q,p} \tau_p(E) \, 
	\norm{g_1}{L_q(\Omegat;\bbR)} 
	\Biggl[ 
	\sum_{j=1}^M 
	\norm{ x_j - y_j }{E}^{ip} \norm{y_j}{E}^{(k-i)p} 
	\Biggr]^{\nicefrac{1}{p}}  
	\right], 
	\end{align*} 
	which completes the proof of \eqref{eq:lem:MLMC-k}, since 
	$\norm{g_1}{L_q(\Omegat;\bbR)} 
	\leq 
	K_{q,2} \norm{g_1}{L_2(\Omegat;\bbR)}  
	= 
	K_{q,2}$ 
	follows from \eqref{eq:kahane-khintchine-gaussian}, 
	and 
	$\Cdiff 
	= 
	16 k \sqrt{\pi} K_{q,p} \tau_{p}(E) K_{q,2}$ 
	by \eqref{eq:Cdiff}. 
\end{proof}

\begin{proposition}\label{prop:MLMC-k} 
	Assume that  
	$(E,\norm{\,\cdot\,}{E})$ 
	is of Rademacher type $p\in[1,2]$. 
	Let $q\in[p,\infty)$, $k,M\in\bbN$  
	and $\eta_1, \ldots, \eta_M, \xi_1,\ldots,\xi_M\in L_{kq}(\Omega;E)$ be 
	$E$-valued random variables 
	such that the tuples
	$(\eta_1,\xi_1),\ldots,(\eta_M,\xi_M)$ 
	are independent and identically distributed. 
	Then, 
	\begin{align*} 
	\biggl\| 
	\Mk[\eta_1] 
	-
	\Mk[\xi_1] 
	&-
	\frac{1}{M} 
	\sum_{j=1}^M 
	\bigl( 
	\xk \eta_j 
	- 
	\xk \xi_j 
	\bigr) 
	\biggr\|_{L_q(\Omega;\xkseps E)} 
	\\
	&\leq 
	2 \Cdiff \,  
	M^{-\left( 1-\frac{1}{p} \right)} 
	\sum_{i=1}^k 
	\Bigl[ 
	{\textstyle\binom{k}{i}} 
	\norm{ \eta_1 - \xi_1 }{L_{kq}(\Omega;E)}^i 
	\norm{ \xi_1 }{L_{kq}(\Omega;E)}^{k-i} 
	\Bigr] 
	\\
	&\quad + 
	2 B_q \, 
	M^{-\nicefrac{1}{2}} 
	\norm{ \eta_1 - \xi_1 }{L_k(\Omega;E)} 
	\sum_{i=0}^{k-1} 
	\Bigl[ 
	\norm{ \eta_1 }{L_k(\Omega;E)}^i 
	\norm{ \xi_1 }{L_k(\Omega;E)}^{k-i-1} 
	\Bigr] 
	, 
	\end{align*} 
	where $B_q, \Cdiff\in(0,\infty)$ 
	are defined as in \eqref{eq:khintchine} 
	and \eqref{eq:Cdiff}, respectively. 
\end{proposition} 

\begin{proof} 
	We proceed similarly as in 
	the proof of Theorem~\ref{thm:MC-k}.  
	We pick 
	a Rademacher family $(r_j)_{j=1}^M$ 
	on a complete probability space 
	$(\Omegat,\cAt,\bbPt)$,  
	and define 
	the following random variables on 
	the product probability space 
	$(\Omega \times \Omegat,\cA \otimes \cAt,\bbP\otimes\bbPt)$: 
	For every $j\in\{1,\ldots,M\}$,  we set  
	\[
	\etavec_j(\omega,\omegat) := \eta_j(\omega), 
	\quad 
	\xivec_j(\omega,\omegat) := \xi_j(\omega), 
	\quad 
	\rvec_j(\omega,\omegat) := r_j(\omegat) 
	\qquad 
	\forall (\omega,\omegat) \in \Omega\times\Omegat, 
	\]
	where we note that 
	$(\rvec_j)_{j=1}^M$ is a Rademacher family 
	on $(\Omega \times \Omegat,\cA \otimes \cAt,\bbP\otimes\bbPt)$, 
	and that 
	$(\etavec_1,\xivec_1), \ldots,(\etavec_M,\xivec_M), \rvec_1,\ldots,\rvec_M$ 
	are independent. 
	Furthermore,  
	the random variables 
	$\xk \etavec_j  - \xk \xivec_j  -  \Mk[\eta_1] + \Mk[\xi_1]$ 
	are centered for all $1\leq j \leq M$ 
	so that by Lemma~\ref{lem:symmetrization} 
	and by the triangle inequality on 
	$L_q(\Omega\times\widetilde{\Omega};\xkseps E)$ 
	we find that 
	\begin{equation}\label{eq:proof:prop:MLMC-k}  
	\begin{split} 
	\biggl\| 
	\sum_{j=1}^M 
	\bigl(\xk \eta_j &- \xk \xi_j - \Mk [\eta_1] + \Mk[\xi_1] \bigr)
	\biggr\|_{L_q(\Omega;\xkseps E)} 
	\\  
	&= 
	\biggl\| 
	\sum_{j=1}^M 
	\bigl( \xk \etavec_j - \xk \xivec_j - \Mk[\eta_1] + \Mk[\xi_1] \bigr)
	\biggr\|_{L_q(\Omega\times\widetilde{\Omega};\xkseps E)} 
	\\
	&\leq 
	2\, 
	\biggl\| 
	\sum_{j=1}^M 
	\rvec_j 
	\bigl( \xk \etavec_j  - \xk \xivec_j  \bigr) 
	\biggr\|_{L_q(\Omega\times\widetilde{\Omega};\xkseps E)}   
	\\
	&\quad + 
	2\, 
	\biggl\| 
	\sum_{j=1}^M 
	\rvec_j 
	\bigl( \Mk[\eta_1] - \Mk[\xi_1] \bigr)
	\biggr\|_{L_q(\Omega\times\widetilde{\Omega};\xkseps E)}   
	=: 
	2\text{(A)} + 2\text{(B)}. 
	\end{split} 
	\end{equation}
	
	For term (A) we use Fubini's theorem 
	as well as Lemma~\ref{lem:MLMC-k} 
	to find that 
	\begin{align*} 
	\text{(A)} 
	&= 
	\biggl\| 
	\sum_{j=1}^M 
	\rvec_j 
	\bigl( \xk \etavec_j  - \xk \xivec_j  \bigr) 
	\biggr\|_{L_q(\Omega\times\widetilde{\Omega};\xkseps E)}   
	\\
	&= 
	\Biggl( 
	\int_\Omega 
	\,\biggl\| 
	\sum_{j=1}^M 
	r_j (\,\cdot\,) 
	\bigl( \xk \eta_j(\omega) - \xk \xi_j(\omega) \bigr) 
	\biggr\|_{L_q(\widetilde{\Omega};\xkseps E)}^q 
	\rd\bbP(\omega)
	\Biggr)^{\nicefrac{1}{q}} 
	\\
	&\leq 
	\Cdiff  
	\Biggl( 
	\int_\Omega 
	\,\Biggl| 
	\sum_{i=1}^k 
	\binom{k}{i} 
	\biggl[  
	\sum_{j=1}^M 
	\norm{ \eta_j(\omega) - \xi_j(\omega)}{E}^{ip} 
	\norm{ \xi_j(\omega) }{E}^{(k-i)p} 
	\biggr]^{\nicefrac{1}{p}} 
	\,\Biggr|^q 
	\rd\bbP(\omega)
	\Biggr)^{\nicefrac{1}{q}} 
	\end{align*}
	and, hence, 
	\[
	\text{(A)} 
	\leq 
	\Cdiff \, 
	\Biggl\| 
	\sum_{i=1}^k 
	\binom{k}{i} 
	\biggl[  
	\sum_{j=1}^M 
	\norm{ \eta_j - \xi_j }{E}^{ip} 
	\norm{ \xi_j }{E}^{(k-i)p} 
	\biggr]^{\nicefrac{1}{p}} 
	\Biggr\|_{L_q(\Omega;\bbR)} . 
	\]
	Next, we use the triangle inequality on $L_q(\Omega;\bbR)$ 
	as well as the fact that $q\in[p,\infty)$ so that 
	also on $L_{\nicefrac{q}{p}}(\Omega;\bbR)$ 
	we may apply the triangle inequality and 
	conclude 
	\begin{align*} 
	\text{(A)} 
	&\leq 
	\Cdiff \, 
	\sum_{i=1}^k 
	\binom{k}{i} 
	\biggl\| 
	\sum_{j=1}^M 
	\norm{ \eta_j - \xi_j }{E}^{ip} 
	\norm{ \xi_j }{E}^{(k-i)p} 
	\biggr\|_{L_{\nicefrac{q}{p}}(\Omega;\bbR)}^{\nicefrac{1}{p}} 
	\\
	&\leq 
	\Cdiff \, 
	\sum_{i=1}^k 
	\binom{k}{i} 
	\Biggl( 
	\sum_{j=1}^M 
	\Bigl\| 
	\norm{ \eta_j - \xi_j }{E}^{ip} 
	\norm{ \xi_j }{E}^{(k-i)p} 
	\Bigr\|_{L_{\nicefrac{q}{p}}(\Omega;\bbR)}
	\Biggr)^{\nicefrac{1}{p}} 
	\\
	&= 
	\Cdiff \, 
	M^{\nicefrac{1}{p}} 
	\sum_{i=1}^k 
	\binom{k}{i} 
	\Bigl\| 
	\norm{ \eta_1 - \xi_1 }{E}^{ip} 
	\norm{ \xi_1 }{E}^{(k-i)p} 
	\Bigr\|_{L_{\nicefrac{q}{p}}(\Omega;\bbR)}^{\nicefrac{1}{p}} , 
	\end{align*} 
	where the last step follows from the identical 
	distribution of 
	$(\eta_1,\xi_1), \ldots,(\eta_M,\xi_M)$. 
	In addition, we observe that, 
	for every $i\in\{1,\ldots,k-1\}$, 
	by H\"older's inequality 
	\begin{align*} 
	\Bigl\| 
	&\norm{ \eta_1 - \xi_1 }{E}^{ip} 
	\norm{ \xi_1 }{E}^{(k-i)p} 
	\Bigr\|_{L_{\nicefrac{q}{p}}(\Omega;\bbR)}^{\nicefrac{1}{p}}  
	= 
	\Bigl( 
	\bbE\Bigl[ 
	\norm{ \eta_1 - \xi_1 }{E}^{iq} 
	\norm{ \xi_1 }{E}^{(k-i)q} 
	\Bigr] 
	\Bigr)^{\nicefrac{1}{q}} 
	\\
	&\leq 
	\Bigl( 
	\bigl( \bbE\bigl[ \norm{ \eta_1 - \xi_1 }{E}^{kq} \bigr] \bigr)^{\frac{i}{k}} 
	\bigl( \bbE\bigl[ \norm{ \xi_1 }{E}^{kq} \bigr] \bigr)^{\frac{k-i}{k}} 
	\Bigr)^{\nicefrac{1}{q}} 
	= 
	\norm{ \eta_1 - \xi_1 }{L_{kq}(\Omega;E)}^i
	\norm{ \xi_1 }{L_{kq}(\Omega;E)}^{k-i} 
	\end{align*} 
	which completes the bound for term (A), 
	\begin{equation}\label{eq:proof:prop:MLMC-k-A}  
	\text{(A)} 
	\leq 
	\Cdiff \, 
	M^{\nicefrac{1}{p}} 
	\sum_{i=1}^k 
	\Bigl[ 
	{\textstyle\binom{k}{i}} 
	\norm{ \eta_1 - \xi_1 }{L_{kq}(\Omega;E)}^i
	\norm{ \xi_1 }{L_{kq}(\Omega;E)}^{k-i}  
	\Bigr]. 
	\end{equation} 
	
	For term (B) we obtain 
	by the Khintchine inequalities~\eqref{eq:khintchine} 
	and by Lemma~\ref{lem:Mk-diff} 
	the following estimate,  
	\begin{equation}\label{eq:proof:prop:MLMC-k-B} 
	\begin{split}  
	\text{(B)} 
	&= 
	\biggl\| 
	\sum_{j=1}^M r_j 
	\biggr\|_{L_q(\widetilde{\Omega};\bbR)} 
	\bigl\| \Mk[\eta_1] - \Mk[\xi_1] \bigr\|_{\eps_s} 
	\leq 
	B_q \, M^{\nicefrac{1}{2}} 
	\bigl\| \Mk[\eta_1] - \Mk[\xi_1] \bigr\|_{\eps_s} 
	\\
	&\leq 
	B_q \, M^{\nicefrac{1}{2}}  
	\| \eta_1 - \xi_1 \|_{L_{k}(\Omega;E)} 
	\sum_{i=0}^{k-1} 
	\Bigl[ 
	\| \eta_1 \|_{L_{k}(\Omega;E)}^i 
	\| \xi_1 \|_{L_{k}(\Omega;E)}^{k-i-1} 
	\Bigr] . 
	\end{split} 
	\end{equation} 
	
	The claim now follows by 
	combining \eqref{eq:proof:prop:MLMC-k} with the estimates 
	\eqref{eq:proof:prop:MLMC-k-A}, \eqref{eq:proof:prop:MLMC-k-B} 
	for the terms (A) and (B),  
	upon dividing  the resulting inequality by $M$. 
\end{proof}  

We are now ready to formulate our convergence 
result for abstract multilevel Monte Carlo methods 
to estimate higher-order statistical moments of Banach 
space valued random variables. 

\begin{theorem}\label{thm:MLMC-k}
	Let 
	$(E,\norm{\,\cdot\,}{E})$ 
	be of Rademacher type $p\in[1,2]$,  
	${q\in[p,\infty)}$ and ${k,L\in\bbN}$. 
	Suppose further that, for every $\ell\in\{1,\ldots,L\}$, 
	${X_\ell \in L_{kq}(\Omega;E)}$,  
	$M_\ell\in\bbN$, and 
	$\xi_{\ell,1},\ldots,\xi_{\ell,M_\ell}$ 
	are independent 
	copies of the $\xkseps E$-valued 
	random variable 
	\[
	\xi_\ell 
	:= 
	\xk X_\ell - \xk X_{\ell-1} \in L_q ( \Omega; \xkseps E ), 
	\qquad 
	X_0 := 0\in E. 
	\]
	Then, for every $\tenvec \in \xkseps E$, 
	\begin{align*} 
	\biggl\| 
	\tenvec
	&-
	\sum_{\ell = 1}^{L} 
	\frac{1}{M_\ell} 
	\sum_{j=1}^{M_\ell}  
	\xi_{\ell,j}
	\biggr\|_{L_q(\Omega;\xkseps E)} 
	\leq 
	\bigl\| 
	\tenvec 
	-
	\Mk[ X_L ] 
	\bigr\|_{\eps_s}
	\\
	&+ 
	\CML  
	\sum_{\ell=1}^{L}   
	\biggl[ 
	M_\ell^{-\left( 1-\frac{1}{p} \right)}  
	\norm{ X_\ell - X_{\ell-1} }{L_{kq}(\Omega;E)} 
	\\
	&\hspace{17mm} \cdot  
	\sum_{i=0}^{k-1} 
	\Bigl[ \Bigl(		
	{\textstyle\binom{k}{i+1}} 
	\norm{ X_\ell - X_{\ell-1} }{L_{kq}(\Omega;E)}^i 
	+ 
	\norm{ X_\ell }{L_k(\Omega;E)}^i 
	\Bigr) 
	\norm{ X_{\ell-1} }{L_{kq}(\Omega;E)}^{k-i-1} 
	\Bigr] 
	\biggr] ,
	\end{align*}
	where 
	$\CML := 
	2 \max\bigl\{ \Cdiff, B_q \bigr\}$ 
	and the constants 
	$B_q, \Cdiff\in(0,\infty)$ 
	are defined as in \eqref{eq:khintchine} 
	and \eqref{eq:Cdiff}. 
\end{theorem} 

\begin{proof} 
	First note that, for every $\ell\in\{1,\ldots,L\}$ 
	the random variables 
	$\xi_{\ell,1},\ldots,\xi_{\ell,M_\ell}$ 
	are identically distributed 
	and we have that 
	\[ 
	\bbE\Biggl[ 
	\sum_{\ell = 1}^{L} 
	\frac{1}{M_\ell} 
	\sum_{j=1}^{M_\ell}  
	\xi_{\ell,j} \Biggr] 
	= 
	\sum_{\ell = 1}^{L} 
	\bbE[ \xi_\ell ] 
	=
	\sum_{\ell = 1}^{L} 
	\bigl( \Mk[ X_\ell ] - \Mk[ X_{\ell-1} ] \bigr) 
	= 
	\Mk[X_L] . 
	\] 
	Thus, we find by the triangle inequality 
	on $L_q(\Omega;\xkseps E)$ that, 
	for every $\tenvec\in\xkseps E$, 
	\begin{align*} 
	\biggl\| 
	\tenvec 
	-
	\sum_{\ell = 1}^{L} 
	\frac{1}{M_\ell} 
	\sum_{j=1}^{M_\ell}  
	\xi_{\ell,j}
	\biggr\|_{L_q(\Omega;\xkseps E)} 
	&\leq 
	\bigl\| 
	\tenvec 
	-
	\Mk[ X_L ] 
	\bigr\|_{\eps_s} 
	+ 
	\sum_{\ell=1}^L 
	\mathrm{err}^{\sf\, SL}_{q,\eps_s}(\xi_\ell), 
	\end{align*}
	where, for $\ell\in\{1,\ldots,L\}$, we define  
	\[
	\mathrm{err}^{\sf\, SL}_{q,\eps_s}(\xi_\ell) 
	:= 
	\biggl\| 
	\bbE[\xi_\ell ] - 
	\frac{1}{M_\ell} 
	\sum_{j=1}^{M_\ell}  
	\xi_{\ell,j}
	\biggr\|_{L_q(\Omega;\xkseps E)} 
	. 
	\]
	For every $\ell\in\{1,\ldots,L\}$, 
	we let the tuples $(X_{\ell-1,1}, X_{\ell,1}), \ldots, 
	(X_{\ell-1,M_\ell}, X_{\ell,M_\ell})$ 
	be $M_\ell$ independent  
	copies of $(X_{\ell-1}, X_\ell)$ 
	and observe that 
	\begin{align*} 
	\mathrm{err}^{\sf\, SL}_{q,\eps_s}(\xi_\ell) 
	= 
	\biggl\| 
	\Mk[X_\ell] 
	-
	\Mk[X_{\ell-1}] 
	- 
	\frac{1}{M_\ell} 
	\sum_{j=1}^{M_\ell}
	\bigl( 
	\xk X_{\ell,j} 
	- 
	\xk X_{\ell-1,j}
	\bigr) 
	\biggr\|_{L_q(\Omega;\xkseps E)} . 
	\end{align*} 
	We are thus in the position to 
	apply Proposition~\ref{prop:MLMC-k} 
	on every level $\ell\in\{1,\ldots,L\}$, 
	\begin{align*} 
	\mathrm{err}^{\sf\, SL}_{q,\eps_s}(\xi_\ell) 
	&\leq  
	2 \Cdiff \,  
	M_\ell^{-\left( 1-\frac{1}{p} \right)} 
	\sum_{i=1}^k 
	\Bigl[ 
	{\textstyle\binom{k}{i}}
	\norm{ X_\ell - X_{\ell-1} }{L_{kq}(\Omega;E)}^i 
	\norm{ X_{\ell-1} }{L_{kq}(\Omega;E)}^{k-i} 
	\Bigr] 
	\\
	&\quad + 
	2 B_q \, 
	M_\ell^{-\nicefrac{1}{2}} 
	\norm{ X_\ell - X_{\ell-1} }{L_k(\Omega;E)} 
	\sum_{i=0}^{k-1} 
	\Bigl[ 
	\norm{ X_\ell }{L_k(\Omega;E)}^i 
	\norm{ X_{\ell-1} }{L_k(\Omega;E)}^{k-i-1} 
	\Bigr] 
	, 
	\end{align*} 
	which after recalling that 
	$p\in[1,2]$ and $q\geq p \geq 1$  
	as well as combining the two sums  
	completes the proof of the assertion. 
\end{proof} 

The error estimate of Theorem~\ref{thm:MLMC-k} 
facilitates optimizing the number 
of levels $L$ as well as 
the number of samples on each level, 
$M_1,\ldots, M_L$, 
to reduce the computational cost 
for achieving a target accuracy $\epsilon>0$ 
of the MLMC estimator in the 
$L_q(\Omega;\xkseps E)$-norm. 
This optimization is subject of the 
following ``$\alpha\beta\gamma$ theorem''.

\begin{theorem}\label{thm:alpha-beta-gamma}
	Assume that  
	$(E,\norm{\,\cdot\,}{E})$ 
	is of Rademacher type $p\in(1,2]$. 
	Let $q\in[p,\infty)$, 
	$k\in\bbN$, 
	$X \in L_k(\Omega; E)$,  
	$( X_\ell )_{\ell\in\bbN} \subset L_{kq}(\Omega;E)$ 
	be a sequence of 
	$E$-valued random variables 
	and, for every $\ell\in\bbN$, define   
	\begin{equation}\label{eq:def:xi-ell} 
	\xi_\ell 
	:= 
	\xk X_\ell - \xk X_{\ell-1} \in L_q ( \Omega; \xkseps E ), 
	\qquad 
	X_0 := 0\in E. 
	\end{equation}
	For $\ell\in\bbN$, 
	let $\cC_\ell$ denote the cost 
	(number of floating point operations)
	to generate one sample of 
	the random variable $\xi_\ell$ 
	in~\eqref{eq:def:xi-ell}, 
	and suppose 
	that there exist 
	a sequence $(N_\ell)_{\ell\in\bbN}$
	of positive integers 
	and constants 
	${\alpha,\beta,\gamma,C_\alpha,C_\beta,C_\gamma,\Cstab\in(0,\infty)}$, 
	$A\in(1,\infty)$ such that  
	$N_\ell \eqsim A^\ell$ 
	for all $\ell\in\bbN$ and, moreover, 
	\begin{align} 
	\quad 
	\forall \ell\in\bbN : 
	&&
	\bigl\| \Mk[X] - \Mk[X_\ell] \bigr\|_{\eps_s}  
	&\leq 
	C_\alpha N_\ell^{-\alpha} \!,
	\qquad
	\tag{$\alpha$} 
	\label{eq:ass:alpha} 
	\\
	\quad
	\forall \ell\in\bbN : 
	&&
	\| X_\ell - X_{\ell-1} \|_{L_{kq}(\Omega;E)} 
	&\leq 
	C_\beta N_\ell^{-\beta} \!, 
	\qquad 
	\tag{$\beta$} 
	\label{eq:ass:beta} 
	\\
	\quad
	\forall \ell\in\bbN : 
	&&
	\cC_\ell 
	&\leq 
	C_\gamma N_\ell^\gamma , 
	\qquad 
	\tag{$\gamma$} 
	\label{eq:ass:gamma} 
	\\
	\quad 
	\forall \ell\in\bbN : 
	&&
	\max\bigl\{ 
	\norm{X}{L_k(\Omega;E)} , \, 
	&\norm{X_\ell}{L_{kq}(\Omega;E)} 
	\bigr\}
	\leq 
	\Cstab. 
	\quad
	\tag{$\sf stab$} 
	\label{eq:ass:stab}  
	\end{align} 
	For each $\ell\in\bbN$, 
	let $( \xi_{\ell,j} )_{j\in\bbN} \subset L_q ( \Omega; \xkseps E )$ 
	be a sequence of independent 
	copies of 
	the $\xkseps E$-valued random variable 
	$\xi_\ell$ in \eqref{eq:def:xi-ell}. 
	
	Then, for every $\epsilon\in(0,\nicefrac{1}{2}]$, 
	there exist integers 
	$L\in\bbN$ 
	and 
	$M_1,\ldots,M_L \in \bbN$ such 
	that the $L_q$-accuracy $\epsilon$ of 
	the multilevel Monte Carlo 
	estimator for $\Mk[X]$, 
	\begin{equation}\label{eq:epsilon} 
	\mathrm{err}^{k, \sf ML}_{q,\eps_s}(X)
	:=  
	\biggl\| 
	\Mk[X] 
	-
	\sum_{\ell = 1}^{L} 
	\frac{1}{M_\ell} 
	\sum_{j=1}^{M_\ell}  
	\xi_{\ell,j}
	\biggr\|_{L_q(\Omega;\xkseps E)} 
	<\epsilon , 
	\tag{$\epsilon$}
	\end{equation} 
	can be achieved at computational 
	costs of the order 
	\begin{equation}\label{eq:cC} 
	\cC^{k, \sf ML}_{q,\eps_s}(X)
	\lesssim_{(\alpha,\beta,\gamma,A,p,q)}  
	\begin{cases} 
	\epsilon^{-\frac{\gamma}{\alpha}} 
	+ 
	\epsilon^{-p'}  
	& \text{if}\quad \beta p' > \gamma, 
	\\
	\epsilon^{-\frac{\gamma}{\alpha}} 
	+ 
	\epsilon^{-p'} 
	|\log_A \epsilon|^{p'+1} 
	&\text{if}\quad \beta p' = \gamma, 
	\\
	\epsilon^{-\frac{\gamma}{\alpha}} 
	+ 
	\epsilon^{-p' - \frac{\gamma - \beta p' }{\alpha}} 
	& \text{if}\quad 
	\beta p' < \gamma, 
	\end{cases}
	\tag{$\cC$} 
	\end{equation}
	where $p'\in[2,\infty)$ is such that 
	$\tfrac{1}{p} + \tfrac{1}{p'} = 1$. 
	The constant implied in $\lesssim$ may also 
	depend on the constants $C_\alpha,C_\beta,C_\gamma$ 
	and $C_{\mathsf{stab}}$ from the assumptions above. 
\end{theorem} 

\begin{proof} 
	We will show by explicit construction that, for every $\epsilon\in(0,\nicefrac{1}{2}]$, 
	assumptions \eqref{eq:ass:alpha}, \eqref{eq:ass:beta}, \eqref{eq:ass:gamma} 
	and \eqref{eq:ass:stab} allow to choose 
	the algorithmic steering parameters $L\in\bbN$ 
	and $M_1,\ldots,M_L\in\bbN$ so that \eqref{eq:epsilon} holds 
	with cost \eqref{eq:cC}.  
	
	By Theorem~\ref{thm:MLMC-k} 
	and by assumptions 
	\eqref{eq:ass:alpha}, 
	\eqref{eq:ass:beta},  
	\eqref{eq:ass:stab} 
	we obtain the estimate 
	\begin{align*} 
	&\mathrm{err}^{k, \sf ML}_{q,\eps_s}(X)
	\leq 
	\bigl\| \Mk[X] - \Mk[X_L] \bigr\|_{\eps_s}  
	+ 
	\CML  
	\sum_{\ell=1}^L   
	\biggl[ 
	M_\ell^{-\left( 1-\frac{1}{p} \right)}  
	\norm{ X_\ell - X_{\ell-1} }{L_{kq}(\Omega;E)} 
	\\
	&\quad\hspace{18mm} \cdot  
	\sum_{i=0}^{k-1} 
	\Bigl[ \Bigl(		
	{\textstyle\binom{k}{i+1}} 
	\norm{ X_\ell - X_{\ell-1} }{L_{kq}(\Omega;E)}^i 
	+ 
	\norm{ X_\ell }{L_k(\Omega;E)}^i 
	\Bigr) 
	\norm{ X_{\ell-1} }{L_{kq}(\Omega;E)}^{k-i-1} 
	\Bigr] 
	\biggr] 
	\\
	&\leq C_\alpha N_L^{-\alpha}  
	+ 
	\CML C_\beta
	\sum_{\ell=1}^L   
	\biggl[ 
	M_\ell^{-\left( 1-\frac{1}{p} \right)}  
	N_\ell^{-\beta}
	\sum_{i=0}^{k-1} 
	\Bigl[ \Bigl(		
	{\textstyle\binom{k}{i+1}} 
	C_\beta^i  
	+ 
	\Cstab^i 
	\Bigr) 
	\Cstab^{k-i-1} 
	\Bigr] 
	\biggr] 
	\\
	&\leq 
	C_\alpha N_L^{-\alpha}  
	+ 
	C_{\star} 
	\sum_{\ell=1}^L   
	\Bigl[ 
	M_\ell^{-\nicefrac{1}{p'} }  
	N_\ell^{-\beta}
	\Bigr] , 
	\end{align*} 
	where 
	$\CML\in(0,\infty)$ is as in 
	Theorem~\ref{thm:MLMC-k}, and 
	$C_{\star} = C_\star(k,p,q,C_\beta,\Cstab)\in(0,\infty)$ 
	is defined by 
	\[
	C_{\star} 
	:= 
	\CML C_\beta 
	\bigl( 
	k \, \Cstab^{k-1} 
	+ 
	2^k
	\max\bigl\{C_\beta^{k-1},\Cstab^{k-1} \bigr\} 
	\bigr), 
	\]
	since 
	\begin{align*} 
	\sum_{i=0}^{k-1} 
	\Bigl[ \Bigl(		
	{\textstyle\binom{k}{i+1}} 
	C_\beta^i  
	+ 
	\Cstab^i 
	\Bigr) 
	\Cstab^{k-i-1} 
	\Bigr] 
	&\leq 
	k \, \Cstab^{k-1} 
	+ 
	\max\bigl\{C_\beta^{k-1},\Cstab^{k-1} \bigr\} 
	\sum_{i=0}^{k-1} 
	{\textstyle\binom{k}{i+1}} 
	\\
	&\leq 
	k \, \Cstab^{k-1} 
	+ 
	2^k
	\max\bigl\{C_\beta^{k-1},\Cstab^{k-1} \bigr\} . 
	\end{align*} 
	Choose $L\in\bbN$ as the smallest integer such that 
	$N_L^{-\alpha} < 
	\min\{ C_\alpha^{-1},1\} \, \tfrac{\epsilon}{2}$ holds 
	and, for every $\ell\in\{1,\ldots,L\}$, let $M_\ell\in\bbN$ be 
	defined as the smallest integer satisfying 
	\[
	M_\ell 
	\geq 
	C_{\star}^{p'}
	N_L^{ \alpha p' }
	S_L^{p'}  
	N_\ell^{-\frac{ (\beta+\gamma) p' }{p'+1}}, 
	\qquad 
	\text{where} 
	\qquad 
	S_L
	:= 
	\sum_{\ell=1}^L 
	N_{\ell}^{\frac{ \gamma - \beta p' }{p'+1}}. 
	\] 
	Note that the magnitude of $S_L$ 
	behaves asymptotically (for $L$ large) 
	as 
	\begin{equation}\label{eq:theta-L} 
	S_L 
	= 
	\sum_{\ell=1}^L 
	N_{\ell}^{\frac{\gamma-\beta p'}{p'+1}} 
	\eqsim_{(\beta,\gamma,A,p)} 
	\begin{cases} 
	1 & 
	\text{if}\quad  \beta p' > \gamma, 
	\\
	L & 
	\text{if}\quad
	\beta p' = \gamma, 
	\\
	N_L^{\frac{\gamma-\beta p'}{p'+1}} 
	&\text{if}\quad 
	\beta p' < \gamma. 
	\end{cases}
	\end{equation}
	For this choice of $L$ and 
	$M_1,\ldots,M_L$, we can bound the error 
	as follows, 
	\begin{align*} 
	\mathrm{err}^{k, \sf ML}_{q,\eps_s}(X)
	&< 
	C_\alpha 
	C_\alpha^{-1} 
	\frac{\epsilon}{2}
	+ 
	C_{\star} 
	\sum_{\ell=1}^{L}   
	\biggl[ 
	C_{\star}^{-1} 
	N_L^{-\alpha} 
	S_L^{-1} 
	N_\ell^{\frac{\beta+\gamma}{p'+1}} 
	N_\ell^{-\beta}
	\biggr] 
	\\
	&=  
	\frac{\epsilon}{2} 
	+ 
	N_L^{-\alpha} 
	S_L^{-1} 
	\sum_{\ell=1}^L 
	N_\ell^{\frac{\gamma-\beta p' }{p'+1}}
	=  
	\frac{\epsilon}{2} 
	+ 
	N_L^{-\alpha} 
	<  
	\epsilon . 
	\end{align*} 
	For the total cost, 
	we first compute 
	\begin{align*} 
	\cC^{k, \sf ML}_{q,\eps_s}&(X) 
	\eqsim
	\sum_{\ell=1}^L \cC_\ell  M_\ell 
	\leq  
	C_\gamma 
	\sum_{\ell=1}^L N_\ell^\gamma 
	\biggl( 
	1 + 
	C_{\star}^{p'} 
	N_L^{ \alpha p' }
	S_L^{p'}  
	N_\ell^{-\frac{(\beta+\gamma) p' }{p'+1}} 
	\biggr) 
	\\
	&\leq 
	C_\gamma 
	\sum_{\ell=1}^L N_\ell^\gamma 
	+ 
	C_\gamma 
	C_\star^{p'} 
	N_L^{ \alpha p' }
	S_L^{p'}  
	\sum_{\ell=1}^L 
	N_\ell^{\frac{\gamma-\beta p'}{p'+1}} 
	= 
	C_\gamma 
	\sum_{\ell=1}^L N_\ell^\gamma 
	+ 
	C_\gamma 
	C_\star^{p'} 
	N_L^{ \alpha p' }
	S_L^{p'+1} \!.  
	\end{align*}
	By the choice of $L$ 
	we have  
	$A^L \eqsim  N_L \eqsim \epsilon^{-\nicefrac{1}{\alpha}}$ 
	and, since $\epsilon\in(0,\nicefrac{1}{2}]$, 
	we find that $L \eqsim_{\alpha} |\log_A \epsilon|$. 
	Thus, using \eqref{eq:theta-L} 
	we conclude that the computational cost, 
	\begin{align*} 
	&\cC^{k, \sf ML}_{q,\eps_s}(X) 
	\lesssim_{(\beta,\gamma,A,p,q)}  
	\begin{cases} 
	N_L^\gamma 
	+ 
	N_L^{\alpha p'}  
	&\text{if } 
	\beta p' > \gamma, \\
	N_L^\gamma 
	+ 
	N_L^{\alpha p'} 
	L^{p'+1}  
	&\text{if }  
	\beta p' = \gamma, \\
	N_L^\gamma 
	+ 
	N_L^{\alpha p'+\gamma-\beta p'} 
	&\text{if }  
	\beta p' < \gamma, 
	\end{cases}
	\end{align*} 
	 in terms of the accuracy 
	$\epsilon$ behaves as follows, 
	\begin{align*} 
	&\cC^{k, \sf ML}_{q,\eps_s}(X) 
	\lesssim_{(\alpha,\beta,\gamma,A,p,q)}  
	\begin{cases} 
	\epsilon^{-\frac{\gamma}{\alpha}} 
	+ 
	\epsilon^{-p'} 
	&\text{if } 
	\beta p' > \gamma, \\
	\epsilon^{-\frac{\gamma}{\alpha}} 
	+ 
	\epsilon^{-p'} 
	|\log_A \epsilon|^{p'+1} 
	&\text{if }  
	\beta p' = \gamma, \\
	\epsilon^{-\frac{\gamma}{\alpha}} 
	+ 
	\epsilon^{-p' - \frac{\gamma - \beta p'}{\alpha}} 
	&\text{if }  
	\beta p' < \gamma, 
	\end{cases}
	\end{align*} 
	which completes the proof of the assertion. 
\end{proof} 

\begin{remark}[Strong convergence implies~\eqref{eq:ass:alpha}]
	\label{rem:strong-alpha}
	Lemma~\ref{lem:Mk-diff} 
	shows that 
	under the stability condition~\eqref{eq:ass:stab}, 
	assumption~\eqref{eq:ass:alpha} 
	is satisfied whenever 
	there exists a constant $\widetilde{C}_\alpha\in(0,\infty)$ 
	such that 
	$\norm{X - X_\ell}{L_k(\Omega;E)} 
	\leq \widetilde{C}_\alpha 
	N_\ell^{-\alpha}$ 
	holds 
	for all $\ell\in\bbN$. 
\end{remark} 

\begin{remark}[Comparison with single-level Monte Carlo]
	\label{rem:comparison-SLMC} 
	Under the assumptions \eqref{eq:ass:alpha}, 
	\eqref{eq:ass:gamma}, \eqref{eq:ass:stab} 
	the single-level Monte Carlo approach 
	of Corollary~\ref{cor:MC-k} 
	requires to choose the level 
	$L$ and the number of samples $M_L$ such that 
	\[
	N_L \eqsim \epsilon^{-\nicefrac{1}{\alpha}} 
	\qquad 
	\text{and}
	\qquad 
	M_L \eqsim \epsilon^{-p'}\! , 
	\]
	in order to achieve a 
	target accuracy 
	$\mathrm{err}^{k, \sf SL}_{q,\eps_s}(X) = \epsilon\in(0,\infty)$. 
	Thus, the single-level Monte Carlo method 
	to estimate $\Mk[X]$ 
	causes computational cost 
	of the order  
	\[
	\cC^{k, \sf SL}_{q,\eps_s}(X) 
	\eqsim 
	\cC_L M_L 
	\lesssim 
	N_L^\gamma M_L 
	\eqsim 
	\epsilon^{-\frac{\gamma}{\alpha} - p'} \!. 
	\]
\end{remark} 

\begin{remark}[Comparison with MLMC in Hilbert spaces]
	In the case that $E$ is a Hilbert space, 
	we have that $p=p'=2$ and the computational 
	costs in \eqref{eq:cC} coincide e.g.\ with those 
	of \cite[Theorem~1]{CliffeGilesScheichlTeckentrup2011} 
	for the two cases when $\beta p'\neq \gamma$. 
	In the critical case $\beta p' = 2\beta =\gamma$, 
	we obtain an additional log-factor $|\log_A \epsilon|$. 
	This is due to the fact that 
	we do not assume independence 
	across the levels $\ell\in\{1,\ldots,L\}$. 
	Note that this independence 
	can be exploited only if 
	\begin{enumerate*}[label=(\roman*)] 
		\item 
		$E$ is a  Hilbert space, 
		and 
		\item the error is measured 
		in the $L_2$-norm with respect to 
		$(\Omega,\cA,\bbP)$. 
	\end{enumerate*}
\end{remark} 

\begin{remark}[Full injective tensor norm]
	On the symmetric injective tensor product 
	space $\xkseps E$, 
	the full and symmetric injective tensor norms, 
	$\norm{\,\cdot\,}{\eps}$ 
	and 
	$\norm{\,\cdot\,}{\eps_s}$, 
	are equivalent,  
	see \eqref{eq:eps-norm-equivalence}. 
	Therefore, 
	the results 
	of Subsections~\ref{subsec:slmc} 
	and~\ref{subsec:mlmc} 
	on convergence of Monte Carlo methods 
	of standard (Theorem~\ref{thm:MC-k}), 
	single-level (Corollary~\ref{cor:MC-k}) 
	and multilevel type (Theorems~\ref{thm:MLMC-k} 
	and~\ref{thm:alpha-beta-gamma}) 
	hold also with respect to  
	the stronger norm on 
	$L_q(\Omega;\xkeps E)$, 
	with the additional constant $\tfrac{k^k}{k!}$. 
\end{remark} 

\section{Applications} 
\label{section:applications}   

In this section 
we illustrate the preceding, abstract theory by several
examples of stochastic equations, 
where the need for the presently
developed modifications of the standard 
Monte Carlo theory 
is entailed either 
by problem-specific constraints on the choices of 
non-Hilbertian function spaces for well-posedness 
or by the interest in error estimates  
in norms on Banach spaces (such as H\"older norms). 

Specifically, Subsections~\ref{subsec:EllPDERnDFor} 
and~\ref{subsec:EllPDERndCoef} 
are concerned with 
the $k$th moment 
MLMC finite element convergence 
analysis for explicit, linear, 
second-order elliptic PDEs 
with random forcing (in dimensions $d\in\{2,3\}$)
or random diffusion coefficient (for $d=1$), 
respectively. 
Here, the right-hand side is assumed to 
be an element of 
(or taking values in) 
$L_p(\dom)$ for some $p\in(1,\infty)$, 
where $\dom\subset\bbR^d$ denotes the spatial domain. 
To obtain well-posed problems, 
the case $p\in(1,2)$ necessitates 
variational formulations on Banach spaces, 
whereas for $p\in(2,\infty)$ 
such formulations may be advantageous 
to derive error estimates in H\"older norms 
via Sobolev embeddings. 

In Subsection~\ref{subsec:appl-hoelder} 
we discuss the MLMC approximation 
of higher-order moments 
for vector-valued stochastic 
processes ${X\from [0,T]\times\Omega\to E}$ 
in tensor norms of 
H\"older spaces $C^\delta([0,T];E)$ 
for problem-specific 
H\"older exponents $\delta\in[0,1)$. 
These results  
are applicable to many 
semi-discrete or fully discrete 
numerical schemes 
for SDEs and stochastic PDEs 
and we give some explicit examples. 

\subsection{Linear elliptic PDEs with random forcing} 
\label{subsec:EllPDERnDFor}

Let $\dom\subset \bbR^d$ with $d\in\{2,3\}$ 
be an open, bounded, 
polytopal Lipschitz domain (with closure $\overline{\dom}$)  
and, for $p\in[1,\infty]$ and $m\in\bbN$, 
let $L_p(\dom)$ and $W^m_p(\dom)$  
denote the standard Lebesgue 
and Sobolev spaces 
of real-valued functions on~$\dom$. 

We write ${\Wo}^1_p(\dom)$ for the 
closure of $C^\infty_c(\dom)$ 
(the space of smooth functions 
with compact support inside~$\dom$) 
with respect to the norm 
on $W^1_p(\dom)$, 
and $W^{-1}_p(\dom)$ 
for the dual space of $\Wo^1_{p'}(\dom)$, 
where 
$\frac{1}{p} + \frac{1}{p'}=1$.  

\subsubsection{Deterministic model problem} 

We assume given deterministic,  
continuous diffusion coefficients
$a_{ij}\in C^0(\overline{\dom})$, 
$a_{ij} =a_{ji}$, $1\leq i,j\leq d$,  
which are uniformly positive definite. 
Thus, there exist constants 
$0< \underline{a} \leq \overline{a} < \infty$ 
such that, for all $x\in \overline{\dom}$, 
\begin{equation}\label{eq:rf:CoefBdd}
\forall \phi,\psi \in \bbR^d: 
\quad 
\sum\limits_{i,j=1}^d a_{ij}(x)\phi_i \phi_j 
\geq 
\underline{a} \, \norm{ \phi }{\bbR^d}^2 ,  
\;\;   
\sum\limits_{i,j=1}^d a_{ij}(x)\phi_i \psi_j 
\leq 
\overline{a} \, \norm{ \phi }{\bbR^d} \norm{ \psi }{\bbR^d}
, 
\end{equation}
where $\norm{ \,\cdot\, }{\bbR^d}$ denotes 
the Euclidean norm on $\bbR^d$. 

For $p\in(1,\infty)$ and a given 
source term $f$ in $L_p(\dom)$ 
(which in the sequel shall be generalized to be random), 
we consider the following 
variational formulation of a homogeneous 
Dirichlet boundary value problem: 
Find  
\begin{equation}\label{eq:rf:weak-form} 
u\in \Wo^1_p(\dom):
\;\;  
B(u,v) = \duality{ f,v }  
\quad 
\forall v\in  \Wo^1_{p'}(\dom). 
\end{equation}
Here, $\duality{ \,\cdot\,, \,\cdot\, }$ denotes 
the $W^{-1}_p(\dom)\times \Wo^1_{p'}(\dom)$ duality pairing,
and 
the bilinear form~$B$ 
is given by 
\[ 
B\from \Wo^1_p(\dom) \times \Wo^1_{p'}(\dom)\to \bbR, 
\qquad 
B(w,v) 
:= 
\int_{\dom} \sum_{i,j = 1}^d 
a_{ij}(x) 
\tfrac{\partial }{\partial x_i} w(x) 
\tfrac{\partial }{\partial x_j} v(x) \,\rd x . 
\] 
Evidently, 
H\"older's inequality implies continuity 
of $B$ on $\Wo^1_p(\dom) \times \Wo^1_{p'}(\dom)$. 
However, as opposed to the Hilbert space 
case $p=p'=2$, the 
uniform strong ellipticity assumption  
in \eqref{eq:rf:CoefBdd} 
is in general not sufficient to guarantee   
that the mapping 
$\Wo^1_p(\dom) \ni u \mapsto B(u,\,\cdot\,) \in W^{-1}_p(\dom)$ 
is an isomorphism. 
For the case of the Laplace operator 
(i.e., $a_{ij}(x) = \delta_{ij}$)
and every $p\in(1,\infty)$, 
an inf-sup condition and 
hence well-posedness of \eqref{eq:rf:weak-form} 
have been shown in \cite[Theorem~6.1]{Simader1972}, 
see also \cite[Equation~(8.6.5)]{BrennerScott3rdEd}. 
Following the arguments 
used in \cite[Section~8.6]{BrennerScott3rdEd} 
this can be generalized to 
diffusion coefficients $(a_{ij})_{i,j=1}^d$ 
satisfying \eqref{eq:rf:CoefBdd}, 
provided that $p$ is sufficiently close to $2$. 
In what follows, we will \emph{require} 
for an appropriate range of 
integrability indices $p\in(1,\infty)$ that 
\eqref{eq:rf:weak-form} has a unique solution 
and, moreover, that this solution 
is $W^2_p(\dom)$-regular. 
This is summarized in the next  
assumption. 

\begin{assumption}\label{assumption:regularity}
	There exists $p_0\in(d,\infty)$
	such that, for all $p\in(1,p_0)$ 
	and every $f\in L_p(\dom)$, 
	the variational problem 
	\eqref{eq:rf:weak-form} 
	admits a unique solution $u\in\Wo^1_p(\dom)$, and 
	\begin{equation}\label{eq:RegAss}
	\forall p\in(1,p_0) 
	\quad 
	\exists C_p \in (0,\infty) 
	\quad 
	\forall f\in L_p(\dom) : 
	\quad 
	\norm{ u }{W^2_p(\dom)} 
	\leq 
	C_p 
	\norm{ f }{L_p(\dom)} 
	. 
	\end{equation}
\end{assumption} 

Sufficient conditions for the 
$W^2_p(\dom)$-regularity \eqref{eq:RegAss}
to hold for the Laplace problem  
in polygons (i.e., $d=2$) 
can, for instance, be found 
in~\cite[Theorem~4.3.2.4]{Grisvard2011}.

Since $p_0 > d$ and $d \in\{2,3\}$ are assumed, 
for a given $q \in [p_0,\infty)$,  
we may choose   
${p = \frac{qd}{d+q} \in ( 1, \min\{d,q\} )}$   
in \eqref{eq:RegAss} 
and conclude by continuity of the 
Sobolev embedding $W^2_p(\dom)\subseteq W^1_q(\dom)$ 
and H\"older's inequality that 
$\norm{ u }{W^1_q(\dom)} 
\lesssim_{(q,\dom)} 
\norm{ u }{W^2_p(\dom)} 
\lesssim_{p} 
\norm{ f }{L_p(\dom)} 
\lesssim_{(q,\dom)}
\norm{ f }{L_q(\dom)}$. 
For $q\in(1,p_0)$ this estimate 
trivially holds by \eqref{eq:RegAss} 
so that we obtain the 
following stability estimate 
for all $p\in(1,\infty)$: 
\begin{equation}\label{eq:rf:stab-u}
\forall p\in(1,\infty): 
\quad 
\norm{ u }{W^1_p(\dom)} 
\lesssim_{(p,\dom)} 
\norm{ f }{L_p(\dom)}. 
\end{equation}

\subsubsection{Finite element approximation} 
\label{subsubsec:rf:FEM} 

For the numerical approximation, 
we use a conforming finite element method (FEM) 
for \eqref{eq:rf:weak-form} 
based on continuous, piecewise 
first-order Langrangean basis functions on $\overline{\dom}$:
on a regular, simplicial triangulation $\cT_h$ of $\overline{\dom}$ 
with mesh size $h\in(0,\infty)$, 
we consider the finite-dimensional space
\[ 
S^1_0(\dom;\cT_h) 
:= 
\bigl\{ v \in C^0(\overline{\dom}) : 
v|_{\partial \dom} =0, \; v|_T \in \bbP_1 \; \forall T\in \cT_h \bigr\},
\]
where $\bbP_1$ denotes the space of polynomials 
of degree at most one. 
The corresponding Galerkin discretization 
of \eqref{eq:rf:weak-form} reads: 
Find  
\begin{equation}\label{eq:rf:FEMPbm} 
u_h \in S^1_0(\dom;\cT_h): 
\;\; 
B( u_h, v_h ) 
= 
\duality{ f,v_h } 
\quad 
\forall v_h\in S^1_0(\dom;\cT_h).
\end{equation}
Evidently, this finite-dimensional problem is equivalent to 
a linear system of equations, with matrix that is symmetric
and, by \eqref{eq:rf:CoefBdd}, positive definite, so that there
exists a unique solution 
$u_h \in  S^1_0(\dom;\cT_h)$ of \eqref{eq:rf:FEMPbm}.

Under Assumption~\ref{assumption:regularity} and 
the additional condition that 
\begin{equation}\label{eq:ass:aij} 
\forall i,j\in\{1,\ldots,d\} : 
\quad 
a_{ij} \in W^1_q(\dom) 
\quad 
\text{for some} 
\quad 
\begin{cases}
q > 2 &\text{if } d = 2 ,\\
q\geq \frac{12}{5} &\text{if } d=3, 
\end{cases} 
\end{equation}
it is shown 
in~\cite[Theorem~8.5.3]{BrennerScott3rdEd} 
that, for all $p\in(1,\infty)$,  
the Galerkin projection~$u_h$ in \eqref{eq:rf:FEMPbm} 
is bounded in $\Wo^1_p(\dom)$:
There exists a mesh size    
$h_0 \in(0,\infty)$ such that 
\begin{equation}\label{eq:rf:uh-stab-u}
\forall h\in(0,h_0): 
\quad 
\norm{ u_h }{W^1_p(\dom)} 
\lesssim_{(p,\dom)} 
\norm{ u }{W^1_p(\dom)}. 
\end{equation}
Combining \eqref{eq:rf:stab-u} and \eqref{eq:rf:uh-stab-u} 
implies stability of both  
the exact solution $u$ and 
its approximation $u_h$: 
For all $p\in(1,\infty)$, there exist 
a mesh width 
$h_0\in(0,\infty)$ 
and a constant 
$\widetilde{C}_{\sf stab} 
\in (0,\infty)$ such that 
\begin{equation}\label{eq:rf:stab} 
\forall h\in(0,h_0) : 
\quad 
\max\bigl\{ 
\norm{ u }{W^1_p(\dom)} ,  
\norm{ u_h }{W^1_p(\dom)} 
\bigr\} 
\leq 
\widetilde{C}_{\sf stab} 
\norm{ f }{L_p(\dom)} . 
\end{equation} 
Moreover, by \cite[Equation~(8.5.4)]{BrennerScott3rdEd} 
$u_h$ is quasi-optimal in $\Wo^1_p(\dom)$ 
for all $p\in(1,\infty)$: 
There exist $h_0,C_{\sf opt}\in(0,\infty)$ 
(which may depend on $\dom$ and $p$)
such that 
\[ 
\forall h\in(0,h_0) : 
\quad 
\norm{ u - u_h }{W^1_p(\dom)} 
\leq 
C_{\sf opt} 
\inf_{v_h \in S^1_0(\dom;\cT_h)} 
\norm{ u - v_h }{W^1_p(\dom)} .
\] 
Therefore, 
under Assumption~\ref{assumption:regularity} 
and~\eqref{eq:ass:aij}, 
for every quasi-uniform family of 
triangulations $(\cT_h)_{h\in\cH}$, 
standard approximation properties
of the corresponding finite element 
spaces $S^1_0(\dom;\cT_h)$, 
$h\in\cH\subseteq(0,\infty)$, 
show that, for all $p\in(0,p_0)$, 
\begin{equation}\label{eq:rf:FEM-err}
\forall h\in\cH\cap(0,h_0) : 
\quad 
\norm{ u - u_h }{W^1_p(\dom)} 
\lesssim_{(p,\dom)} 
h \, \norm{ u }{W^2_p(\dom)} 
\lesssim_{(p,\dom)} 
h \, \norm{ f }{L_p(\dom)} , 
\end{equation}
where we also used the 
assumed regularity~\eqref{eq:RegAss}.

\subsubsection{Random forcing and MLMC-FEM} 

Suppose the setting of the previous 
subsections. In particular, 
the coefficients $(a_{ij})_{i,j=1}^d$ satisfy 
\eqref{eq:rf:CoefBdd} and \eqref{eq:ass:aij}, 
and Assumption~\ref{assumption:regularity} holds 
for some $p_0\in(d,\infty)$. 
In this subsection we fix $p\in(1,p_0)$,~which 
determines the spatial integrability 
of a given random forcing. 
Random forcing in \eqref{eq:rf:weak-form} 
amounts to assuming that 
the right-hand side $f$ is an element of  
$L_r(\Omega;L_p(\dom))$  
for some suitable integrability index ${r\in[1,\infty)}$ 
with respect to the probability space~$(\Omega,\cA,\bbP)$, i.e., 
we seek a $\Wo^1_p(\dom)$-valued random variable~$u$ 
such that 
\begin{equation}\label{eq:rf:weak-form-stoch} 
B(u(\omega),v) = \duality{ f(\omega),v }  
\quad 
\forall v\in  \Wo^1_{p'}(\dom),
\;\;\; 
\text{for almost all }
\omega\in\Omega. 
\end{equation}

Under the above mentioned regularity 
requirements, see Assumption~\ref{assumption:regularity}, 
we may argue for almost all $\omega\in\Omega$  
to establish the existence and uniqueness 
of a (stochastic) solution~$u\in L_r(\Omega;\Wo^1_p(\dom))$ 
satisfying \eqref{eq:rf:weak-form-stoch}, 
with
\[ 
u \in L_r(\Omega; W^2_p(\dom)), 
\qquad 
\norm{ u }{L_r(\Omega; W^2_p(\dom)) }  
\lesssim_{(p,\dom)} 
\norm{ f }{L_r(\Omega; L_p(\dom)) }  . 
\] 

\emph{Multilevel finite element discretizations} 
of \eqref{eq:rf:weak-form-stoch} will be based on the 
discrete variational problem~\eqref{eq:rf:FEMPbm}, 
considered $\bbP$-a.s.
To this end, we denote by $\{ \cT_\ell \}_{\ell\in\bbN}$ 
a nested sequence of regular, simplicial triangulations 
$\cT_\ell$ of~$\overline{\dom}$, 
with corresponding sequence of mesh sizes  
$\{ h_\ell \}_{\ell \in\bbN}$. 
We assume that 
$\cT_{\ell+1}$ is obtained from $\cT_\ell$ 
via uniform red refinement. 
Then, $h_{\ell+1} \lesssim h_\ell /2$  
and, without loss of generality, 
we may assume that $h_1 < h_0$, 
with $h_0$ as in \eqref{eq:rf:uh-stab-u}--\eqref{eq:rf:FEM-err}.
The corresponding sequence of 
Galerkin solutions 
$u_{h_\ell} \in S^1_0(\dom;\cT_\ell)$ 
shall be denoted by $u_\ell$ 
(with slight abuse~of~notation). 

In the next corollary we verify 
that all assumptions 
of the ``$\alpha\beta\gamma$ theorem'', 
see Theorem~\ref{thm:alpha-beta-gamma}, 
are satisfied to bound the computational 
costs of the MLMC-FEM estimator for $\Mk[u]$ 
for a given accuracy 
and provide an upper bound 
for these~costs. 

\begin{corollary}\label{cor:MLMCForc}
	Let \eqref{eq:rf:CoefBdd} and \eqref{eq:ass:aij}  
	be satisfied and suppose that 
	Assumption~\ref{assumption:regularity} 
	holds for some $p_0\in(d,\infty)$. 
	Assume that $p\in(1,p_0)$, $q\in[\min\{p,2\},\infty)$, 
	and~$k\in\bbN$.  
	For $f\in L_{kq}(\Omega;L_p(\dom))$, 
	let $u\in L_{kq}(\Omega;\Wo^1_p(\dom))$ 
	be the solution to \eqref{eq:rf:weak-form-stoch}. 
	Furthermore, let 
	the FEM approximations 
	$(u_\ell)_{\ell\in\bbN}$  
	be constructed as described above. 
	Then, for $E:=W^1_p(\dom)$ and 
	$N_\ell := \dim\bigl(S_0^1(\dom;\cT_\ell) \bigr) \eqsim h_\ell^{-d}\eqsim A^\ell$\!, 
	with $ A:= 2^d$\!,
	all conditions of 
	Theorem~\ref{thm:alpha-beta-gamma} are fulfilled, 
	\begin{align*} 
	\text{\eqref{eq:ass:alpha}}&
	\quad 
	\forall\ell\in\bbN:& 
	\bigl\| \Mk[u] - \Mk[u_\ell] \bigr\|_{\eps_s} 
	&\lesssim_{(k,p,\dom,f)} N_\ell^{-\nicefrac{1}{d}} ,
	&\hspace*{-3mm}\text{i.e., } 
	\alpha &= d^{-1},  
	\\
	\text{\eqref{eq:ass:beta}}&
	\quad 
	\forall\ell\in\bbN:& 
	\norm{ u_\ell - u_{\ell-1} }{L_{kq}(\Omega;W^1_p(\dom))} 
	&\lesssim_{(k,p,q,\dom,f)} N_\ell^{-\nicefrac{1}{d}} ,
	&\hspace*{-3mm}\text{i.e., } 
	\beta &= d^{-1}, 
	\\
	\text{\eqref{eq:ass:gamma}}&
	\quad 
	\forall\ell\in\bbN:& 
	\cC_\ell 
	&\lesssim N_\ell^k ,
	&\hspace*{-3mm}\text{i.e., } 
	\gamma &= k, 
	\\
	\text{\eqref{eq:ass:stab}}&
	\quad 
	\forall\ell\in\bbN:& 
	\max\bigl\{ 
	\norm{u}{L_k(\Omega;W^1_p(\dom))}, 
	\| u_\ell &\|_{L_{kq}(\Omega;W^1_p(\dom))} 
	\bigr\}
	\leq 
	\Cstab , 
	& 
	\end{align*}
	for some constant $\Cstab\in(0,\infty)$ depending 
	only on $p,\dom$ and $\norm{f}{L_{kq}(\Omega;L_p(\dom))}$. 
	
	Furthermore, the $L_q$-accuracy 
	$\mathrm{err}_{q,\eps_s}^{k,\sf ML}(u) < \epsilon\in (0,\nicefrac{1}{2}]$ 
	of the multilevel Monte Carlo estimator for $\Mk[u]$ 
	can be achieved at computational 
	costs of the order 
	\begin{equation}\label{eq:rf:cost} 
	\cC^{k, \sf ML}_{q,\eps_s}(u)
	\lesssim_{(k,p,q,\dom,f)}  
	\begin{cases} 
	\epsilon^{-\bar{p}'}  
	& \text{if}\quad \bar{p}' > kd, 
	\\ 
	\epsilon^{-kd} 
	\, |\log_2 \epsilon|^{kd+1} 
	&\text{if}\quad \bar{p}' = kd, 
	\\
	\epsilon^{-kd} 
	& \text{if}\quad 
	\bar{p}' < kd, 
	\end{cases}
	\end{equation}
	where $\bar{p}'\in[2,\infty)$ is such that 
	$\tfrac{1}{\min\{p,2\}} + \tfrac{1}{\bar{p}'} = 1$. 
\end{corollary} 

\begin{proof} 
	First note that  
	we may apply the deterministic 
	stability estimate \eqref{eq:rf:stab} 
	for almost all $\omega\in\Omega$, 
	showing that, for all $\ell\in\bbN$, 
	\[ 
	\max\bigl\{ 
	\norm{ u }{L_k(\Omega;W^1_p(\dom))} , 
	\norm{ u_\ell }{L_{kq}(\Omega;W^1_p(\dom))} 
	\bigr\} 
	\leq 
	\widetilde{C}_{\sf stab} 
	\norm{ f }{L_{kq}(\Omega;L_p(\dom))} 
	=: 
	\Cstab.
	\] 
	
	The integrability of 
	$f\in L_{kq}(\Omega; L_p(\dom))$ combined with 
	the deterministic FEM convergence result \eqref{eq:rf:FEM-err} 
	implies strong convergence, 
	\[
	\forall \ell\in\bbN: 
	\quad 
	\norm{u - u_\ell}{L_{kq}(\Omega;W^1_p(\dom))}
	\lesssim_{(p,\dom)} 
	h_\ell 
	\, 
	\norm{f}{L_{kq}(\Omega;L_p(\dom))}. 
	\]
	Since $h_\ell \eqsim N_\ell^{-\nicefrac{1}{d}}$\!,
	we conclude that the 
	conditions \eqref{eq:ass:alpha} and \eqref{eq:ass:beta} 
	of Theorem \ref{thm:alpha-beta-gamma}
	are satisfied with $\alpha=\beta=d^{-1}$\!, 
	where we also have used  
	Remark~\ref{rem:strong-alpha} for \eqref{eq:ass:alpha}
	and the triangle inequality for \eqref{eq:ass:beta}. 
	
	Assuming a linear complexity solver (as, e.g., multigrid), 
	the cost $\cC_{\ell,1}$ for computing one sample 
	of $u_\ell$ in~\eqref{eq:rf:FEMPbm} 
	is bounded by  
	$\cC_{\ell,1} \leq C_{\gamma,1} N_\ell$ 
	with some constant $C_{\gamma,1}\in(0,\infty)$ 
	independent of $\ell$.
	Since the computation of the $k$th Kronecker product 
	of a vector of length $N_\ell$ causes computational 
	cost of the magnitude $N_\ell^k$\!, 
	the total cost $\cC_\ell$ for computing one sample 
	of the random variable 
	$\xi_\ell = \xk u_\ell - \xk u_{\ell-1}$ 
	is of the order $N_\ell^k$, and 
	condition \eqref{eq:ass:gamma} 
	is satisfied with $\gamma=k$. 
	
	Thus, the assumptions 
	\eqref{eq:ass:alpha}, 
	\eqref{eq:ass:beta}, 
	\eqref{eq:ass:gamma} and  
	\eqref{eq:ass:stab} 
	of Theorem \ref{thm:alpha-beta-gamma} 
	are satisfied, and the upper bound for the 
	computational costs in \eqref{eq:rf:cost} 
	follows upon applying  
	Theorem~\ref{thm:alpha-beta-gamma}, 
	since the Banach space $E=W^1_p(\dom)$ 
	has type $\min\{p,2\}$. 
\end{proof} 

\begin{remark} 
	In the Hilbert space case, 
	it is in general not optimal 
	to obtain a convergence rate bound for \eqref{eq:ass:alpha} 
	by combining strong convergence with  
	stability~\eqref{eq:ass:stab}, as outlined 
	in Remark~\ref{rem:strong-alpha}. 
	For instance, the error analysis 
	of Galerkin approximations for 
	generalized Whittle--Mat\'ern fields 
	in \cite[Proposition~4]{CoxKirchner2020}  
	reveals that the corresponding 
	approximations of the covariance function  
	converge more than twice as fast 
	in the $L_2(\dom{\times}\dom)$-norm 
	as the corresponding Gaussian random field approximations 
	in the strong $L_q(\Omega;L_2(\dom))$-sense. 
	However, it is not obvious if and how 
	this behavior generalizes to random variables 
	with values in Banach spaces.  
\end{remark} 

\subsection{Linear elliptic PDEs with log-Gaussian coefficient}
\label{subsec:EllPDERndCoef}

We next consider a linear, second-order elliptic PDE 
with mixed Dirichlet--Neumann boundary conditions 
and right-hand side in $L_p(\dom)$. 
As opposed to Subsection~\ref{subsec:EllPDERnDFor}, 
we now assume that the 
diffusion coefficient $a$ is random. 
More specifically, 
as e.g.\ in \cite{Charrier2012, 
	CharrierScheichlTeckentrup2013,
	TeckentrupScheichlGilesUllmann2013}, 
we suppose that 
$a$ is log-Gaussian, 
i.e., $a(x) = \exp(g(x))$ 
for almost all $x\in\dom$, $\bbP$-a.s.,  
for some Gaussian random field
$g\from\dom\times\Omega\to\bbR$. 
The spatial domain is assumed to be a bounded interval 
$\dom = I = (0,b)$ of length $b\in(0,\infty)$. 
The restriction of the spatial dimension to $d=1$ 
facilitates an explicit expression of the inf-sup constant 
of the bilinear form, 
appearing in the corresponding 
variational formulation, 
depending on 
$\underline{a}:=\operatorname{ess} \inf_{x\in I} a(x)$.  
In this setting, $\underline{a}$ is a random variable 
satisfying $\underline{a} \in L_q(\Omega;\bbR)$ 
for all $q\in[1,\infty)$. 
Similarly as in~\cite{Charrier2012}, 
this, in turn, yields well-posedness 
of the variational problem 
and strong convergence 
of finite element approximations. 

\subsubsection{Deterministic model problem} 
\label{subsubsec:rc:det-model} 

For $p\in[1,\infty]$, $m\in\bbN$, 
we recall the Lebesgue and Sobolev 
spaces $L_p(I)$ and $W^m_p(I)$ 
from Subsection~\ref{subsec:EllPDERnDFor}. 
We furthermore note that,  
since $d=1$, for every $p\in[1,\infty]$, 
elements $v$ in $W^1_p(I)$ coincide  
(upon a modification on a subset of $\overline{I}$ 
of zero Lebesgue measure) 
with a unique function which 
is continuous on $\overline{I} = [0,b]$, 
denoted by $\widetilde{v}\in C^0(\overline{I})$,  
and we define the subspace 
\[
\Wo_{p, \{0\} }^1(I) 
:= 
\bigl\{v\in W^1_p(I): \widetilde{v}(0) = 0 \bigr\}. 
\]
In virtue of the Poincar\'e inequality, 
on this subspace 
the map  
$v \mapsto \seminorm{ v }{W^1_p(I)}$,  
$\seminorm{ v }{W^1_p(I)} := \norm{ v' }{L_p(I)}$  
defines a norm,  
where $v'$ denotes the weak derivative of~$v$. 
In addition, 
we write $p'\in[1,\infty]$ for 
the H\"older conjugate of $p\in[1,\infty]$, 
and we let 
$W^{-1}_{p,\{0\}}(I)$ 
be the dual space of 
$\Wo^1_{p'\!,\{0\}}(I)$, 
equipped with the norm 
\[
\norm{f}{W^{-1}_{p, \{0\}}(I)} 
:= 
\sup_{0 \neq v\in \Wo^1_{p'\!,\{0\}}(I) } 
\frac{\duality{f,v}}{ \seminorm{v}{W^1_{p'}(I) } }. 
\]

We assume given a 
finite partition $\cP = \{ J_i \}_{i = 1}^{n_{\cP} }$
of pairwise disjoint, open subintervals $J_i$ of $I$ 
such that 
$\overline{J_1}\cup \ldots \cup \overline{J_{n_\cP}} 
= \overline{I} = [0,b]$. 
Furthermore, 
we suppose that the scalar diffusion coefficient 
satisfies 
\[ 
a\in W^1_\infty(I;\cP), 
\quad 
\text{where} 
\quad 
W^1_\infty(I;\cP) 
:= 
\bigl\{ a \in L_{\infty}(I) \, | \, \forall J\in \cP : \, a|_J \in W^1_\infty(J) \bigr\} , 
\] 
and it is positive in the sense that there 
exist constants $\underline{a}$, $\overline{a}$ such that 
\[
0 < \underline{a} \leq a(x) \leq \overline{a} < \infty 
\qquad 
\text{for almost all }x\in I. 
\]
For $p\in(1,\infty)$ and $f\in L_p(I)$  
we then consider the following boundary value problem, 
with mixed (Dirichlet--Neumann) boundary conditions: 
Find  
\begin{equation}\label{eq:Elliptic}
u \from \overline{I}\to \bbR: 
\quad\; 
-(a(x)u'(x))' = f(x),  
\quad 
\text{a.a. } x\in I, 
\quad\;\; 
u(0) = a(b)u'(b) = 0.
\end{equation}
The weak formulation of \eqref{eq:Elliptic}  
reads: Find 
\begin{equation}\label{eq:rc:weak-form}
u\in \Wo^1_{p,\{0\}}(I): 
\;\; 
B_a(u,v) = \duality{ f, v } 
\quad  
\forall v\in \Wo^1_{p'\!,\{0\}}(I), 
\end{equation}
where the bilinear form $B_a$ 
is defined by  
$B_a(w,v) := \int_0^b a(x) w'(x) v'(x) \, \rd x$, 
for all $w\in \Wo^1_{p,\{0\}}(I)$ 
and $v \in \Wo^1_{p'\!,\{0\}}(I)$. 
For every $p\in(1,\infty)$, 
existence and uniqueness of 
a solution $u$ to \eqref{eq:rc:weak-form} 
follow from continuity of~$B_a$, 
\begin{equation}\label{eq:Bcont} 
\forall w\in \Wo^1_{p,\{0\}}(I) 
\;\;\; 
\forall v\in \Wo^1_{p'\!,\{0\}}(I) : 
\quad 
|B_a(w,v)| 
\leq 
\overline{a} \, 
\seminorm{ w }{W^1_p(I)} 
\seminorm{ v }{W^1_{p'}(I)}
,
\end{equation}
and the following \emph{inf-sup condition}: 
\begin{equation}\label{eq:infsup}
\inf_{0 \neq w \in \Wo^1_{p,\{0\} }(I) } 
\sup_{0\neq v \in \Wo^1_{p'\!,\{0\} }(I) } 
\frac{B_a(w,v)}{ 
	\seminorm{w}{W^1_p(I)} \seminorm{v}{W^1_{p'}(I)} } 
\geq  
\underline{a} . 
\end{equation}
For the homogeneous Dirichlet boundary value problem 
($u(0)=u(b)=0$), 
a constructive proof for 
the inf-sup condition 
on $\Wo_p^1(I)\times \Wo^1_{p'}(I)$ 
has been given 
in \cite[Proof of Theorem~3.1]{BabuskaRheinboldt1981}. 
We adjust the argument 
from \cite{BabuskaRheinboldt1981}, 
to derive~\eqref{eq:infsup} 
for the problem \eqref{eq:Elliptic} 
with mixed boundary conditions. 
To this end, 
let  
$w\in \Wo^1_{p,\{0\}}(I) \setminus\{ 0 \}$ 
be arbitrary but fixed, 
and define
\begin{equation}\label{eq:def:vw}
v_w(x) 
:= 
\int_0^x \operatorname{sign}( w'(t) ) |w'(t)|^{p-1} \, \rd t, 
\quad  
x \in\overline{I} = [0,b].
\end{equation}
This function satisfies  
$v_w(0) = 0$, 
and it is weakly differentiable with 
weak derivative 
\[
v_w'(x) 
= 
\operatorname{sign}( w'(x) ) |w'(x)|^{p-1} 
\qquad 
\text{for almost all }x\in I. 
\]
We furthermore obtain that, 
for almost all $x\in I$,  
$|v_w'(x)| 
=  
|w'(x)|^{p-1} 
=
|w'(x)|^{p/p'}$\!, 
and conclude that $v_w\in \Wo^1_{p'\!,\{0\}}(I)$, 
with  
\[ 
\seminorm{ v_w }{ W^1_{p'}(I)} 
= 
\norm{ v_w' }{L_{p'}(I)} 
=  
\norm{ w' }{L_p(I)} ^{p/p'} 
=  
\seminorm{ w }{W^1_p(I)}^{p/p'} 
= 
\seminorm{ w }{W^1_p(I)}^{p-1} . 
\] 
The continuity \eqref{eq:Bcont}
of $B_a$ implies that $B_a(w, v_w)$ is finite, 
and we find that 
\[
B_a(w,v_w) 
= 
\int_0^b a(x) |w'(x)|^p \,\rd x 
\geq 
\underline{a} \, 
\seminorm{ w }{W^1_p(I)}^p 
= 
\underline{a} \, 
\seminorm{ w }{W^1_p(I)} 
\seminorm{ v_w }{W^1_{p'}(I)}. 
\]
Since $w\in \Wo^1_{p,\{0\}}(I)\setminus\{0\}$ was arbitrary, 
\eqref{eq:infsup} follows. 

The inf-sup condition \eqref{eq:infsup} 
(together with its symmetric counterpart 
which can be shown in the same fashion)
implies that, 
for every $f\in W^{-1}_{p,\{0\}}(I)$,
the variational problem \eqref{eq:rc:weak-form} 
admits a unique solution
$u\in \Wo^1_{p,\{0\}}(I)$. 
Furthermore, 
the linear data-to-solution mapping 
$W^{-1}_{p,\{0\}}(I) \ni f\mapsto u\in \Wo^1_{p,\{0\}}(I)$
is an isomorphism with 
\begin{equation}\label{eq:rc:stab-u}
\norm{ u' }{L_p(I)} 
=
\seminorm{ u }{W^1_p(I)} 
\leq 
\underline{a}^{-1} 
\norm{f}{W^{-1}_{p,\{0\}}(I)} . 
\end{equation}

In the case that $f\in L_p(I)$, 
this solution is more regular:
Considering the differential equation \eqref{eq:Elliptic} 
in weak sense on $J\in\cP$ implies, 
for $a\in W^1_\infty(I;\cP)$ and $f\in L_p(I)$, 
that  the second weak 
derivative of $u|_J$ 
restricted to $J\subseteq I$ 
exists and 
\[ 
-u|_J''(x) 
= 
a(x)^{-1} 
\left[ f(x) + a|_J'(x)u'(x) \right] 
\quad 
\text{for almost all } x\in J. 
\] 
Taking here the $L_p(J)$-norm yields 
with elementary estimates that, 
for every $J\in \cP$, 
\begin{align*}
\norm{ u|_J'' }{L_p(J)} 
&\leq  
\norm{ a^{-1} }{L_\infty(J)} 
\left[ 
\norm{ f }{L_p(J)} 
+ 
\norm{ a|_J' }{L_\infty(J)} 
\norm{ u' }{L_p(J)} 
\right]  
\\
&\leq 
\underline{a}^{-1} 
\Bigl[ 
\norm{ f }{L_p(I)} 
+ 
\underline{a}^{-1} 
\norm{ a|_J' }{L_\infty(J)} 
\norm{ f }{W^{-1}_{p,\{0\}}(I)} 
\Bigr]  
\leq 
C^{\sf\, reg}_{a,p} \,
\norm{ f }{L_p(I)} , 
\end{align*}
where the constant 
$C^{\sf\, reg}_{a,p}\in(0,\infty)$ 
is given by 
\[
C^{\sf\, reg}_{a,p} 
:= 
\underline{a}^{-1} 
\Bigl[ 1 + 
\underline{a}^{-1}  
\max_{J\in\cP} \norm{ a|_J' }{L_\infty(J)} \,  
C_{L_p \to W^{-1}_{p,\{0\}}}  
\Bigr], 
\]
and 
$C_{L_p \to W^{-1}_{p,\{0\}}}
:= 
\sup\nolimits_{f \in B_{L_p(I)} } 
\norm{f}{W^{-1}_{p,\{0\}}(I)}$
denotes the norm of the continuous 
embedding $L_p(I) \subset W^{-1}_{p,\{0\}}(I)$. 

Hence, for every $f\in L_p(I)$, 
the unique weak solution to \eqref{eq:rc:weak-form} 
satisfies 
\begin{equation}\label{eq:rc:SolReg} 
u\in W^2_p(I;\cP)\cap \Wo^1_{p,\{0\}}(I), 
\qquad 
\max_{J\in\cP} \norm{ u|_J'' }{L_p(J)}
\leq 
C^{\sf\, reg}_{a,p} 
\, \norm{f}{L_p(I)}, 
\end{equation}
where 
$W^2_p(I;\cP) 
:=
\bigl\{ v\in W^1_p(I) \,|\, \forall J\in \cP:  \, v|_J \in W^2_p(J) \bigr\} $
is the space of functions in $W^1_p(I)$ 
which are piecewise in $W^2_p$ 
on the partition $\cP$ of $I$.

\subsubsection{Finite element approximation} 

For the numerical approximation of the 
solution $u\in \Wo^1_{p,\{0\}}(I)$ to \eqref{eq:rc:weak-form}
we use a similar conforming finite element discretization 
as in Subsection~\ref{subsubsec:rf:FEM}. 
That is, we use  continuous, 
piecewise affine-linear
functions on a partition $\cT_h$ of $\overline{I}$
with mesh size $h\in(0,\infty)$, 
\[ 
S^1_{0,\{0\}}(I;\cT_h) 
:= 
\bigl\{ v\in C^0(\overline{I})  : 
v(0)=0,  \; v|_T \in \bbP_1 \; \forall T\in \cT_h \bigr\}. 
\] 
Evidently, 
$S^1_{0,\{0\}}(I;\cT_h) \subset 
\Wo^1_{p,\{0\}}(I) \cap \Wo^1_{p'\!,\{0\}}(I)$ and 
$\dim\bigl( S^1_{0,\{0\}}(I;\cT_h) \bigr) = \#(\cT_h)$. 
For given ${f\in L_p(I)}$, 
the Galerkin discretization of \eqref{eq:rc:weak-form} reads: 
Find 
\begin{equation}\label{eq:rc:FEM} 
u_h \in S^1_{0,\{0\}}(I;\cT_h): 
\;\; 
B_a( u_h, v_h ) 
= 
\duality{ f,v_h } 
\quad 
\forall v_h\in S^1_{0,\{0\}}(I;\cT_h).
\end{equation}

Unique solvability of \eqref{eq:rc:FEM}
follows from the ($h$-uniform) 
discrete inf-sup condition: 
\begin{equation}\label{eq:dinfsup}
\inf_{ 0\neq w_h \in S^1_{0,\{0\}}(I;\cT_h) }
\sup_{ 0\neq v_h \in S^1_{0,\{0\}}(I;\cT_h) }
\frac{B_a(w_h,v_h)}{ 
	\seminorm{ w_h }{W^1_p(I)} \seminorm{ v_h }{W^1_{p'}(I)} }
\geq
\underline{a} . 
\end{equation}
To verify \eqref{eq:dinfsup}, 
note that the proof 
of \eqref{eq:infsup} carries over
to the discrete case: 
Given $w_h \in S^1_{0,\{0\}}(I;\cT_h)$, 
one checks that the 
expression \eqref{eq:def:vw} 
yields an element $v_h$ in $S^1_{0,\{0\}}(I;\cT_h)$,
and that 
all steps in the proof of \eqref{eq:infsup} 
may be repeated verbatim. 

The discrete inf-sup condition \eqref{eq:dinfsup} 
and the continuity \eqref{eq:Bcont} 
imply that \eqref{eq:rc:FEM} 
admits a unique solution
$u_h\in  S^1_{0,\{0\}}(I;\cT_h)$ 
with 
\begin{equation}\label{eq:rc:stab-uh} 
\norm{ u_h' }{L_p(I)}
=
\seminorm{ u_h }{W^1_p(I)} 
\leq 
\underline{a}^{-1} 
\norm{ f }{W^{-1}_{p,\{0\}}(I)}, 
\end{equation} 
which is, furthermore, quasi-optimal:
\begin{equation}\label{eq:QuasiOpt}
\seminorm{ u - u_h }{W^1_p(I)} 
\leq 
\left( 1+\tfrac{\overline{a}}{\underline{a}} \right) 
\inf_{v_h \in S^1_{0,\{0\}}(I;\cT_h)} 
\seminorm{ u - v_h }{W^1_p(I)} . 
\end{equation}
Therefore, 
for every quasi-uniform family 
of grids  $(\cT_h)_{h\in\cH}$ on $\overline{I}$ 
which is such that, 
for every $h\in\cH$, 
the grid $\cT_h$ is compatible 
with the partition $\cP$, 
the quasi-optimality \eqref{eq:QuasiOpt} and 
the regularity \eqref{eq:rc:SolReg} 
imply the error bound
\begin{align}
&| u - u_h |_{W^1_p(I)} 
\textstyle\leq 
\left( 1+\tfrac{\overline{a}}{\underline{a}} \right) 
\seminorm{ u - \cI_h  u }{W^1_p(I)} 
= 
\left( 1+\tfrac{\overline{a}}{\underline{a}} \right) 
\biggl[ \sum\limits_{J\in\cP}
\seminorm{ u - \cI_h  u }{W^1_p(J)}^p 
\biggr]^{\nicefrac{1}{p}} 
\notag 
\\
&\textstyle 
\leq 
C_{b,p} 
\left( 1+\tfrac{\overline{a}}{\underline{a}} \right) 
h \, 
\biggl[\sum\limits_{J\in \cP} \norm{ u|_J'' }{L_p(J)}^p 
\biggr]^{\nicefrac{1}{p}} 
\leq 
C_{b,p} \, n_\cP^{\nicefrac{1}{p}}
\left( 1+\tfrac{\overline{a}}{\underline{a}} \right) 
C_{a,p}^{\sf \, reg} 
\, 
h 
\, \norm{f}{L_p(I)}	, 
\label{eq:rc:FEMErr}
\end{align}
upon choosing $v_h$ in \eqref{eq:QuasiOpt}  
as the nodal interpolant $\cI_h u$ of $u$ 
in $S^1_{0,\{0\}}(I;\cT_h)$. 
Here, the constant $C_{b,p}\in(0,\infty)$ 
is independent of $a$, $h$ and $u$. 

\subsubsection{Log-Gaussian random coefficient and MLMC-FEM} 

The a-priori stability and discretization error bounds 
\eqref{eq:rc:stab-u}, \eqref{eq:rc:stab-uh} and \eqref{eq:rc:FEMErr}
are explicit in 
the dependence on the 
coefficient $a$.
They allow to consider \eqref{eq:Elliptic} with deterministic source 
term $f\in L_p(I)$ for some $p\in(1,\infty)$, 
and with random coefficient 
$a\from I \times \Omega \to \bbR$ 
whose logarithm 
$g\from I \times \Omega \to \bbR$ 
is a Gaussian random field. 

More specifically, 
we assume that 
the mapping 
$\Omega\ni\omega\mapsto g(\,\cdot\,, \omega)$ 
is a vector-valued random variable 
taking values in 
$W^1_\infty(I;\cP)$, 
where we note that 
$W^1_\infty(I;\cP)$, 
equipped with the norm 
\[
\norm{v}{W^1_\infty(I;\cP)}
:= 
\norm{ v }{L_\infty(I)} 
+ 
\max_{J\in\cP} 
\norm{ v|_J' }{L_\infty(J)} 
=
\operatorname{ess}\sup_{x\in I} | v(x) |
+ 
\max_{J\in\cP} \,
\operatorname{ess}\sup_{x\in J} | v|_J'(x) |, 
\]
is a Banach space. 
Furthermore, 
$g$ is assumed to be centered Gaussian, 
i.e., for any finite collection 
$( f_1,\ldots,f_n )$ in the dual space 
$[W^1_\infty(I;\cP)]'$ 
the distribution of 
$(\duality{f_1,g},\ldots,\duality{f_n,g})$ 
is multivariate Gaussian 
with zero mean. 
In other words, 
the law $\mu$ of $g$, defined 
for every set $B$ in the Borel 
$\sigma$-algebra $\cB(W^1_\infty(I;\cP))$ by 
\[
\mu(B) 
= 
\bbP( \{\omega \in\Omega : g(\,\cdot\, , \omega) \in B\}),
\]
satisfies that 
$\mu\circ f^{-1}$ is a centered Gaussian measure on $\bbR$ 
for any $f \in [W^1_\infty(I;\cP)]'$\!, 
see e.g.\ \cite[Definition~2.2.1]{Bogachev1998}.

Under these assumptions we have, for almost all $\omega\in\Omega$, 
\begin{equation}\label{eq:a-log-gauss}
a(\,\cdot\,,\omega) =\exp(g(\,\cdot\,,\omega)), 
\quad\text{with} 
\quad 
g(\,\cdot\,,\omega)\in W^1_\infty(I;\cP), 
\end{equation}
and the trajectories of $a$ are $\bbP$-a.s.\ 
continuous on each subinterval $J_1,\ldots,J_{n_{\cP}} \subseteq I$ 
of the partition $\cP$. 
We thus may define, for almost all $\omega\in\Omega$, 
\begin{equation}\label{eq:def:a-rvs} 
\begin{split} 
\underline{a}(\omega) 
&:= 
\min\nolimits_{J\in\cP} \inf\nolimits_{x\in J}  
a(x,\omega) 
= 
\min\nolimits_{J\in\cP}  
\exp\bigl( \inf\nolimits_{x\in J} g(x,\omega) \bigr), 
\\ 
\overline{a}(\omega) 
&:= 
\max\nolimits_{J\in\cP}\sup\nolimits_{x\in J} 
a(x,\omega) 
= 
\max\nolimits_{J\in\cP} 
\exp\bigl( \sup\nolimits_{x\in J} g(x,\omega) \bigr),  
\end{split} 
\end{equation}
so that we obtain, for almost all $\omega\in\Omega$, 
\[
0 <
\exp( - \norm{ g(\,\cdot\,,\omega) }{L_\infty(I)} ) 
\leq 
\underline{a}(\omega) 
\leq 
\overline{a}(\omega) 
\leq 
\exp( \norm{ g(\,\cdot\,, \omega) }{L_\infty(I)}) 
< \infty. 
\]
For stability and strong convergence of 
finite element approximations 
of the solution~$u$ to \eqref{eq:rc:weak-form}  
with the log-Gaussian coefficient $a=\exp(g)$, 
integrability of $\underline{a}^{-1}$\!, $\overline{a}$ 
and of $\max_{J\in\cP}\norm{a|_J'}{L_\infty(J)}$ 
with respect to the 
probability space $(\Omega,\cA,\bbP)$ will be crucial. 
This is summarized in the next lemma. 

\begin{lemma}\label{lem:integrability-a-rvs} 
	The in \eqref{eq:def:a-rvs} $\bbP$-a.s.\ defined 
	mappings $\omega\mapsto\underline{a}(\omega)$ and  
	${\omega\mapsto\overline{a}(\omega)}$ 
	yield random variables satisfying 
	$\underline{a}^{-1}\!, \, \overline{a}\in L_r(\Omega;\bbR)$ 
	for all $r\in[1,\infty)$. 
	
	In addition, the mapping 
	$\overline{a}'\! : 
	\omega \mapsto 
	\max\nolimits_{J\in\cP} 
	\norm{ a|_J'(\,\cdot\,, \omega) }{L_\infty(J)}$
	is $\bbP$-a.s.\ well-defined and 
	$\overline{a}' \in L_r(\Omega;\bbR)$ 
	for all $r\in[1,\infty)$. 
\end{lemma}

\begin{proof}
	The centered Gaussian random field $g$ 
	takes values in the Banach space 
	$W^1_\infty(I;\cP)$.  
	Thus, for every $J\in\cP$ and almost all $\omega\in\Omega$, 
	$g$ admits a representative 
	which is continuous 
	on $\overline{J}$, 
	and  
	the proof of \cite[Proposition~2.3]{Charrier2012}
	using Fernique's theorem is applicable on each $J\in\cP$, 
	showing that 
	$\underline{a}^{-1}\!, \, \overline{a} \in L_r(\Omega;\bbR)$ 
	for all ${r\in[1,\infty)}$. 
	
	We now consider $\overline{a}'$\!. 
	To this end, we first note that by \eqref{eq:a-log-gauss} 
	$a\in W^1_\infty(I;\cP)$, $\bbP$-a.s., 
	since $\exp(\,\cdot\,)$ is smooth and 
	$g\in W^1_\infty(I;\cP)$, $\bbP$-a.s. 
	In particular, $a = \exp(g)$ is also a 
	$W^1_\infty(I;\cP)$-valued random variable.
	Therefore, for almost all ${\omega\in\Omega}$ and 
	every $J\in\cP$, 
	$\norm{ a|_J'(\,\cdot\,, \omega) }{L_\infty(J)} < \infty$. 
	It follows that $\overline{a}'$ is $\bbP$-a.s.\ well-defined 
	and measurable, since the 
	mapping 
	${W^1_\infty(I;\cP) \ni v \mapsto \max_{J\in\cP} 
		\norm{ v|_J' }{L_\infty(J)} \in\bbR}$  
	is continuous. 
	To prove the integrability of $\overline{a}'$\!, 
	we observe that, 
	for almost all $\omega\in\Omega$ and every $J\in\cP$, 
	the weak derivative of $a|_J$ is given by 
	$a|_J'(\,\cdot\,, \omega)
	= 
	g|_J'(\,\cdot\,, \omega)
	\exp(g|_J(\,\cdot\,, \omega))$.
	Thus, we obtain that 
	\[ 
	\max\nolimits_{J\in\cP}
	\norm{a|_J'(\,\cdot\,, \omega)}{L_\infty(J)}
	\leq 
	\max\nolimits_{J\in\cP}
	\norm{g|_J'(\,\cdot\,, \omega)}{L_\infty(J)} 
	\exp\bigl( \norm{g(\,\cdot\,, \omega)}{L_\infty(I)} \bigr). 
	\] 
	We have 
	$\bbE\bigl[ 
	\max\nolimits_{J\in\cP}
	\norm{g|_J'(\,\cdot\,, \omega)}{L_\infty(J)}^q \bigr] < \infty$ 
	for all $q\in(0,\infty)$, 
	since the distribution of $g$ 
	is Gaussian, with values in $W^1_\infty(I;\cP)$. 
	For every $r\in[1,\infty)$, 
	${\overline{a}'\in L_r(\Omega;\bbR)}$ 
	can then be derived along the lines 
	of the proof of \cite[Proposition~2.3]{Charrier2012},  
	using Fernique's theorem 
	which shows that also 
	$\bbE\bigl[ 
	\exp\bigl( q \, \norm{g}{L_\infty(I)} \bigr) 
	\bigr] < \infty$ holds for all ${q\in(0,\infty)}$.
\end{proof}

We now consider the model problem introduced 
in Subsection~\ref{subsubsec:rc:det-model} 
with a log-Gaussian coefficient~$a$  
as in \eqref{eq:a-log-gauss}. 
That is, given a deterministic source 
$f\in L_p(I)$ for some $p\in(1,\infty)$, 
we seek 
$u\from I\times\Omega\to\bbR$ 
such that, for almost all $\omega\in\Omega$, 
\begin{equation}\label{eq:rc:weak-form-stoch} 
u(\,\cdot\,,\omega)\in \Wo^1_{p,\{0\}}(I): 
\;\; 
B_{a(\,\cdot\,,\omega)}(u(\,\cdot\,,\omega),v)
= 
\duality{ f,v }  
\quad 
\forall v\in  \Wo^1_{p'\!,\{0\}}(I), 
\end{equation}
where the bilinear form is as in \eqref{eq:rc:weak-form}. 
The following proposition addresses 
well-posedness of \eqref{eq:rc:weak-form-stoch} 
and regularity of its solution  
in $L_r(\Omega)$-sense.  

\begin{proposition}\label{prop:rc:stab-u} 
	The variational 
	problem~\eqref{eq:rc:weak-form-stoch} 
	admits a solution 
	that is $\bbP$-a.s.\ unique, 
	and belongs to $L_r\bigl(\Omega; \Wo^1_{p,\{0\}}(I) \bigr) 
	\cap L_r\bigl( \Omega;W^2_p(I;\cP) \bigr)$ 
	for all $r\in[1,\infty)$, 
	with 
	\begin{align} 
	\bigl( \bbE\bigl[ \seminorm{u}{W^1_p(I) }^r \bigr] 
	\bigr)^{\nicefrac{1}{r}} 
	&\leq 
	\bigl\| \underline{a}^{-1} \bigr\|_{L_r(\Omega; \bbR)} 
	\norm{ f }{W^{-1}_{p,\{0\}}(I) }, 
	\label{eq:rc:stab-u-stoch}
	\\ 
	\Bigl( 
	\bbE\Bigl[ \bigl(\max_{J\in\cP}\norm{u|_J''}{L_p(J) } \bigr)^r 
	\Bigr] 
	\Bigr)^{\nicefrac{1}{r}}  
	&\leq  
	C_{a,p,r}^{\sf\, reg} 
	\norm{f}{L_p(I)}  ,
	\label{eq:rc:SolReg-stoch}
	\end{align} 
	where 
	$C_{a,p,r}^{\sf\, reg}  := 
	\norm{ \underline{a}^{-1} }{L_r(\Omega; \bbR)} 
	+ 
	\norm{ \underline{a}^{-1} }{L_{4r}(\Omega; \bbR)}^2
	\norm{ \overline{a}' }{L_{2r}(\Omega; \bbR)} 
	C_{L_p\to W^{-1}_{p,\{0\}}}\in(0,\infty)$.
\end{proposition}

\begin{proof} 
	Since $f\in L_p(I)$ is deterministic and 
	$a\in W^1_\infty(I;\cP)$ holds $\bbP$-a.s., 
	existence of a solution to \eqref{eq:rc:weak-form-stoch}, 
	which is $\bbP$-a.s.\ unique, 
	follows by arguing 
	via the well-posedness in the deterministic case 
	(see Subsection~\ref{subsubsec:rc:det-model})
	for almost all $\omega\in\Omega$. 
	Furthermore, for every $r\in[1,\infty)$, 
	the deterministic 
	stability bound \eqref{eq:rc:stab-u} combined 
	with the integrability 
	$\underline{a}^{-1}\in L_r(\Omega;\bbR)$, 
	see Lemma~\ref{lem:integrability-a-rvs}, 
	imply \eqref{eq:rc:stab-u-stoch}. 
	
	We now show the regularity estimate 
	\eqref{eq:rc:SolReg-stoch}. 
	Recalling again the random variables 
	$\underline{a},\overline{a},\overline{a}'$ 
	from Lemma~\ref{lem:integrability-a-rvs}, 
	by \eqref{eq:rc:SolReg} we find that, 
	for almost all $\omega\in\Omega$: 
	\[
	\max_{J\in\cP} 
	\norm{ u|_J''(\,\cdot\,, \omega) }{L_p(J)} 
	\leq 
	\Bigl[ \underline{a}(\omega)^{-1}
	+ \underline{a}(\omega)^{-2} \, \overline{a}'(\omega) 
	\, C_{L_p\to W^{-1}_{p,\{0\}}}
	\Bigr] 
	\norm{f}{L_p(I)} . 
	\]
	Taking the $L_r(\Omega;\bbR)$-norm, 
	and using the Minkowski and H\"older inequalities 
	completes the proof of \eqref{eq:rc:SolReg-stoch}, 
	and the constant $C_{a,p,r}^{\sf\, reg}>0$ 
	is finite by Lemma~\ref{lem:integrability-a-rvs}. 
\end{proof} 

We are now ready to formulate 
the ``$\alpha\beta\gamma$'' theorem 
for multilevel approximations of 
moments of the random solution 
to \eqref{eq:rc:weak-form-stoch}. 
To this end, let $\{ u_\ell \}_{\ell \in\bbN}$ 
be a sequence 
of Galerkin approximations 
$u_\ell := u_{h_\ell}$, see \eqref{eq:rc:FEM}, 
on partitions $\cT_\ell$ of $\overline{I}$ 
corresponding to mesh sizes $h_\ell \eqsim 2^{-\ell}$. 
For example, 
$\cT_\ell$ may be obtained 
by $\ell$-fold bisection 
of the initial partition $\cT_1 := \cP$. 
Note that then 
$N_\ell = 
\dim\bigl( S^1_{0,\{0\}}(I;\cT_\ell) \bigr) 
\eqsim h_\ell^{-1} \eqsim 2^\ell$.

\begin{corollary}\label{cor:MLMCCoeff}
	Let $p\in(1,\infty)$, $q\in[\min\{p,2\},\infty)$, 
	and $k\in\bbN$. 
	For $f\in L_p(I)$, 
	let $u\in L_{kq}\bigl( \Omega;\Wo^1_{p,\{0\}}(I) \bigr)$ 
	be the solution to \eqref{eq:rc:weak-form-stoch}. 
	Assume further that 
	the Galerkin approximations 
	$(u_\ell)_{\ell\in\bbN}$  
	to \eqref{eq:rc:weak-form-stoch} 
	are constructed as described above. 
	Then, for the Banach space  
	$(E, \norm{\,\cdot\,}{E} ) 
	:= 
	\bigl(\Wo^1_{p,\{0\}}(I), \seminorm{\,\cdot\,}{W^1_p(I)}\bigr)$ 
	and with 
	$N_\ell \eqsim 2^\ell$\!, 
	all conditions of 
	Theorem~\ref{thm:alpha-beta-gamma} are fulfilled, 
	\begin{align*} 
	\text{\eqref{eq:ass:alpha}}&
	\quad 
	\forall\ell\in\bbN:& 
	\bigl\| \Mk[u] - \Mk[u_\ell] \bigr\|_{\eps_s} 
	&\lesssim_{(k,b,p,\cP,a,f)} N_\ell^{-1} ,
	&\text{i.e., } 
	\alpha &= 1, 
	\end{align*} 
	\begin{align*} 
	\text{\eqref{eq:ass:beta}}&
	\quad 
	\forall\ell\in\bbN:& 
	\norm{ u_\ell - u_{\ell-1} }{L_{kq}(\Omega;\Wo^1_{p,\{0\}}(I))} 
	&\lesssim_{(k,b,p,q,\cP,a,f)} N_\ell^{-1} ,
	&\text{i.e., } 
	\beta &= 1, 
	\\
	\text{\eqref{eq:ass:gamma}}&
	\quad 
	\forall\ell\in\bbN:& 
	\cC_\ell 
	&\lesssim N_\ell^k ,
	&\text{i.e., } 
	\gamma &= k, 
	\\
	\text{\eqref{eq:ass:stab}}&
	\quad 
	\forall\ell\in\bbN:& 
	\max\bigl\{ 
	\norm{u}{L_k(\Omega;\Wo^1_{p,\{0\}}(I))}, 
	\| u_\ell &\|_{L_{kq}(\Omega;\Wo^1_{p,\{0\}}(I))} 
	\bigr\}
	\leq 
	\Cstab , 
	\hspace*{-1cm} 
	& 
	\end{align*} 
	for some constant $\Cstab\in(0,\infty)$ depending 
	only on 
	$\norm{\underline{a}^{-1}}{L_{kq}(\Omega;\bbR)}$ 
	and 
	$\norm{f}{W^{-1}_{p,\{0\}}(I)}$. 
	
	The $L_q$-accuracy 
	$\mathrm{err}_{q,\eps_s}^{k,\sf ML}(u) < \epsilon\in(0,\nicefrac{1}{2}]$ 
	of the multilevel Monte Carlo estimator for $\Mk[u]$ 
	can be achieved at computational 
	costs of the order \eqref{eq:rf:cost} with $d=1$. 
\end{corollary} 

\begin{proof} 
	We first note that \eqref{eq:rc:stab-u-stoch} 
	combined with the deterministic 
	discrete stability estimate \eqref{eq:rc:stab-uh} 
	and the fact that 
	$\underline{a}^{-1}\in L_{kq}(\Omega;\bbR)$ imply 
	\eqref{eq:ass:stab}: 
	\[
	\forall\ell\in\bbN: 
	\;\;\;  
	\max\Bigl\{ 
	\bigl( \bbE\bigl[ \seminorm{u}{W^1_p(I)}^k \bigr] 
	\bigr)^{\nicefrac{1}{k}}\! , \,
	\bigl( \bbE\bigl[ \seminorm{u_\ell}{W^1_p(I)}^{kq} \bigr] 
	\bigr)^{\nicefrac{1}{kq}} 
	\Bigr\}
	\leq 
	\norm{ \underline{a}^{-1} }{L_{kq}(\Omega;\bbR)} 
	\norm{ f }{W^{-1}_{p,\{0\}}(I)}. 
	\]
	
	Next, we observe strong convergence of the finite 
	element approximations $( u_\ell)_{\ell\in\bbN}$: 
	For all $r\in[1,\infty)$, we obtain 
	by exploiting the deterministic error estimate 
	\eqref{eq:rc:FEMErr} for almost all $\omega\in\Omega$ that 
	\begin{align*} 
	\bigl( \bbE\bigl[ | u - u_\ell |_{W^1_p(I)}^r \bigr]
	\bigr)^{\nicefrac{1}{r}}  
	&\lesssim_{(b,p,\cP)} 
	h_\ell \,   
	\Bigl\| 
	\bigl( 1 + \tfrac{\overline{a}}{\underline{a}} \bigr) 
	\max_{J\in\cP} \norm{ u|_J'' }{L_p(J)} 
	\Bigr\|_{L_r(\Omega;\bbR)} 
	\\
	&\leq 
	h_\ell \, 
	\bigl( 1 + \norm{ \overline{a} }{L_{4r}(\Omega;\bbR)} 
	\norm{ \underline{a}^{-1} }{L_{4r}(\Omega;\bbR)} \bigr) 
	\Bigl\| \max_{J\in\cP} \norm{ u|_J'' }{L_p(J)} \Bigr\|_{L_{2r}(\Omega;\bbR)} 
	\\ 
	&\leq
	h_\ell \, 
	\bigl( 1 + \norm{ \overline{a} }{L_{4r}(\Omega;\bbR)} 
	\norm{ \underline{a}^{-1} }{L_{4r}(\Omega;\bbR)} \bigr) 
	\, C_{a,p,2r}^{\sf\, reg} \, \norm{f}{L_p(I)}, 
	\end{align*} 
	where we also used \eqref{eq:rc:SolReg-stoch} 
	of Proposition~\ref{prop:rc:stab-u}. 
	Thus, the conditions 
	\eqref{eq:ass:alpha} and \eqref{eq:ass:beta} 
	are satisfied with 
	$\alpha=\beta=1$ 
	by Remark~\ref{rem:strong-alpha} 
	and the triangle inquality, respectively.  
	
	Finally, the 
	complexity of computing the Galerkin approximation  
	$u_\ell$ in \eqref{eq:rc:FEM} per one realization  
	of the Gaussian random field 
	$g(\,\cdot\,, \omega)=\log(a(\,\cdot\,, \omega))$
	(assumed given) 
	at discretization level $\ell\in\bbN$ 
	scales linearly with 
	$N_\ell = \dim\bigl(S^1_{0,\{0\}}(I;\cT_\ell) \bigr)$: 
	Observe that the linear system 
	of $N_\ell$ equations for the $N_\ell$ unknowns 
	corresponding to \eqref{eq:rc:FEM} 
	for each sample of $a(\,\cdot\,,\omega)$ 
	is tridiagonal and symmetric, positive definite
	when the standard Courant (``hat functions'')  
	basis is adopted in \eqref{eq:rc:FEM}. 
	The formation of the $k$th order 
	(full, algebraic) tensor product $\otimes^k u_\ell$ 
	then entails the cost bound 
	$\cC_\ell \lesssim \cC_\gamma N_\ell^k$
	for one sample of the random 
	variable $\xi_\ell = \xk u_\ell - \xk u_{\ell-1}$. 
\end{proof} 

\subsection{Approximation of stochastic processes in H\"older spaces}
\label{subsec:appl-hoelder}

In this subsection we let $T \in (0,\infty)$ 
be a time horizon and consider   
approximating higher-order moments of 
vector-valued stochastic processes 
${X\from [0,T]\times \Omega \to E}$ 
by means of multilevel Monte Carlo methods. 
In particular, we derive explicit convergence rates  
for the error of the corresponding approximation 
in injective tensor norms of H\"older spaces 
$C^\delta([0,T];E)$.  
We furthermore detail the implications 
of this general result 
for the Euler--Maruyama 
method for stochastic ordinary differential equations, 
and give an overview of further possible applications 
including approximations of 
stochastic partial differential equations. 

In order to properly define the relevant H\"older spaces, 
we introduce for a Banach space $(F,\norm{\,\cdot\,}{F})$ 
and $\delta \in (0,1)$ the mappings 
\[
\seminorm{\,\cdot\,}{C^\delta([0,T];F)}, \,
\norm{\,\cdot\,}{C^\delta([0,T];F)}
\from C([0,T];F) \to [0,\infty] 
\]
on the Banach space 
\[
\bigl( C([0,T];F), \norm{\,\cdot\,}{C([0,T];F)} \bigr), 
\qquad 
\norm{f}{C([0,T];F)} 
:= 
\sup_{t \in [0,T]} \norm{f(t)}{F}, 
\]
of continuous 
functions from $[0,T]$ to 
$(F,\norm{\,\cdot\,}{F})$
via 
\begin{align*}
\seminorm{f}{C^\delta([0,T];F)}
&:=  
\sup_{\substack{s,t\in [0,T] \\ s\neq t}} \frac{\norm{f(s)-f(t)}{F}}{|s-t|^{\delta}}, 
\\
\norm{f}{C^\delta([0,T];F)}
&:=  
\sup_{t \in [0,T]} \norm{f(t)}{F} 
+ 
\seminorm{f}{C^\delta([0,T];F)} . 
\end{align*} 
We note that the norm 
$\norm{\,\cdot\,}{C^\delta([0,T];F)}$ 
renders the subspace 
\[ 
C^{\delta}([0,T];F) 
= 
\left\{ f \in C([0,T];F) : 
\norm{f}{C^{\delta}([0,T];F)} < \infty \right\}
\subset 
C([0,T];F) 
\] 
of $F$-valued, 
$\delta$-H\"older continuous 
functions a Banach space. 
For brevity, we also use the 
notation $C^0([0,T];F) := C([0,T];F)$ 
to include the case $\delta=0$. 

We now consider the 
setting of \cite[Section~5]{CoxEtAl2021}, 
that is, we are given  
a stochastic process 
$X\from[0,T]\times \Omega \to E$
with continuous sample paths 
satisfying the following regularity assumption: 
There exists a constant $\bar{\beta}\in(0,1]$ 
such that  
\begin{equation}\label{eq:ass:X-beta-regularity} 
\forall \beta\in [0, \bar{\beta}) 
\quad 
\forall 
q\in[1,\infty): 
\quad 
X\in C^{\beta}([0,T];L_q(\Omega;E)). 
\end{equation}
Recall from Subsection~\ref{subsec:notation} 
that $(E,\norm{\,\cdot\,}{E})$  
is assumed to be a real Banach space. 
Additionally, we let 
$(Y^N)_{N\in\bbN}$ 
be a sequence 
of approximations 
$Y^N\!\from[0,T]\times\Omega\to E$ 
to the process~$X$ 
with continuous sample paths, 
which is known to converge at 
the nodes of the temporal partitions    
\begin{align*}
\Theta^N &:= \bigl\{ t_0^N , t_1^N, \ldots, t_{\#(\Theta^N)-1}^N \bigr\}, 
\quad  
\#(\Theta^N) < \infty, 
\qquad 
N\in\bbN, 
\\
0 &\textcolor{white}{:}=: t_0^N 
<  t_1^N 
< \ldots 
< t_{\#(\Theta^N)-2}^N  
< t_{\#(\Theta^N)-1}^N 
:= T , 
\end{align*}
in the strong sense 
essentially at the rate $\bar{\beta}\in(0,1]$, 
i.e., 
\begin{equation}\label{eq:ass:conv-at-nodes}
\forall \beta\in [0,\bar{\beta}) 
\quad 
\forall q \in [1,\infty): 
\quad 
\sup\limits_{t\in\Theta^N} 
\bigl\| X(t) - Y^N\!(t) \bigr\|_{L_q(\Omega;E)} 
\lesssim_{(\bar{\beta},q,T)} 
\bigl| \triangle t^N_{\max} \bigr|^{\beta}\! , 
\end{equation} 
where 
$\triangle t^N_{\max} 
:= 
\max_{j\in\{0,\ldots,\#(\Theta^N)-2\}} 
| t_{j+1}^N - t_j^N |$. 
These partitions 
do not necessarily 
have to be equidistant or nested. 
We only require the following 
quasi-uniformity: 
\begin{equation}\label{eq:ass:grid} 
\sup_{N\in\bbN} \, \frac{ \triangle t^N_{\max}  }{ 
	\triangle t^N_{\min}  } < \infty, 
\qquad 
\text{and} 
\qquad 
\lim_{N\to \infty} \triangle t^N_{\max} = 0, 
\end{equation} 
where 
$\triangle t^N_{\min}$ 
is defined as 
$\triangle t^N_{\max}$ 
with the maximum replaced by the minimum. 
Furthermore, we assume that, 
for every $N\in\bbN$, 
the approximation 
$Y^N$ is linearly interpolated 
on the partition $\Theta^N$\!, i.e., 
for all $j \in \{0,\ldots,\#(\Theta^N)-2\}$, 
\begin{equation}\label{eq:def:linear-interpol} 
Y^N\!(s) 
= 
\frac{ \bigl(t_{j+1}^N - s \bigr) \, Y^N\! \bigl( t_j^N \bigr) }{
	t_{j+1}^N - t_j^N } 
+ 
\frac{ \bigl(s - t_j^N \bigr) \, Y^N\! \bigl(t_{j+1}^N \bigr) }{ 
	t_{j+1}^N - t_j^N }, 
\qquad 
s \in \bigl[ t_j^N, t_{j+1}^N \bigr]. 
\end{equation} 

This general setting facilitates 
combining the abstract 
multilevel Monte Carlo results 
of Subsection~\ref{subsec:mlmc} 
with~\cite[Corollary~2.11]{CoxEtAl2021} 
and, thus, quantifying  
the convergence of the MLMC estimator 
for $\Mk[X]$ based on  
approximations $Y^{N_1}, \ldots, Y^{N_L}$ 
in the norm on 
$L_q(\Omega; \xkseps C^\delta([0,T]; E))$  
for $q\in[p,\infty)$ and $\delta\in[0,\bar{\beta})$, 
see Theorem~\ref{thm:mlmc-hoelder} below. 
To this end, the following proposition  
which readily follows  
from~\cite[Corollary~2.11]{CoxEtAl2021}  
will be crucial. 

\begin{proposition}\label{prop:hoelder-conv} 
	Let $(\Theta^N)_{N\in\bbN}\subset [0,T]$ 
	be a sequence of partitions 
	fulfilling~\eqref{eq:ass:grid}. 
	Assume that  
	${X, Y^N \!\from [0,T]\times\Omega \to E}$, 
	$N\in\bbN$, 
	are stochastic processes 
	with continuous sample paths, 
	such that, 
	for all $N\in\bbN$ and $t\in[0,T]$, 
	the random variables 
	$X(t), Y^N\!(t)\from\Omega\to E$  
	are Bochner measurable and 
	there exists $\bar{\beta}\in(0,1]$ 
	such that \eqref{eq:ass:X-beta-regularity} 
	and \eqref{eq:ass:conv-at-nodes} hold. 
	In addition, for every $N\in\bbN$, 
	let $Y^N$ be linearly interpolated on 
	the partition~$\Theta^N$\!, 
	see~\eqref{eq:def:linear-interpol}. 
	
	Then, we have for every $q\in[1,\infty)$, 
	$\delta\in[0,\bar{\beta})$, 
	and all $\epsilon\in(0,\infty)$, 
	\begin{gather*} 
	\norm{X}{L_q(\Omega; C^\delta([0,T];E))} 
	+
	\sup_{N\in\bbN} 
	\bigl\| Y^N \bigr\|_{L_q(\Omega; C^\delta([0,T];E))} 
	<\infty, 
	\\  
	\bigl\| X - Y^N \bigr\|_{L_q(\Omega; C^\delta([0,T];E))} 
	\lesssim_{(\bar{\beta},\delta,\epsilon,q,T)} 
	\bigl|\triangle t^N_{\max} \bigr|^{\bar{\beta} - \delta - \epsilon} 
	\!.  
	\end{gather*} 
\end{proposition}

\begin{theorem}\label{thm:mlmc-hoelder} 
	Suppose that all assumptions 
	of Proposition~\ref{prop:hoelder-conv} 
	are fulfilled. 
	In addition, let
	$(E,\norm{\,\cdot\,}{E})$ be of 
	Rademacher type $p\in[1,2]$,  
	${k,L\in\bbN}$ 
	and ${\{N_\ell\}_{\ell\in\bbN}\subseteq\bbN}$ be a  
	strictly increasing sequence of integers. 
	Assume further that, 
	for all ${\ell\in\{1,\ldots,L\}}$, 
	$M_\ell\in\bbN$ and 
	$\xi_{\ell,1},\ldots,\xi_{\ell,M_\ell}$ 
	are independent 
	copies of the 
	random variable 
	\[
	\xk Y^{N_\ell} - \xk Y^{N_{\ell-1}} 
	\from  \Omega \to \xkseps C([0,T];E) , 
	\qquad 
	Y^{N_0} := 0 \in C([0,T];E). 
	\]
	Then, for every 
	${q\in[p,\infty)}$, 
	$\delta\in[0, \bar{\beta})$, 
	and all $\epsilon\in(0,\infty)$, 
	we have that 
	\begin{align*} 
	\mathrm{err}_{q, \eps_s}^{k,\sf ML} (X; \delta) 
	&:= 
	\biggl\| 
	\Mk[X] 
	-
	\sum_{\ell = 1}^{L} 
	\frac{1}{M_\ell} 
	\sum_{j=1}^{M_\ell}  
	\xi_{\ell,j}
	\biggr\|_{L_q(\Omega;\xkseps C^\delta([0,T];E))} 
	\\ 
	&\,\lesssim_{(\bar{\beta},\delta,\epsilon,k,p,q,T)} 
	\bigl| \triangle t^{N_L}_{\max} \bigr|^{\bar{\beta}-\delta-\epsilon} 
	+ 
	\sum_{\ell=1}^{L}   
	M_\ell^{-\left( 1-\frac{1}{p} \right)}  
	\bigl| \triangle t^{N_{\ell-1}}_{\max} \bigr|^{\bar{\beta}-\delta-\epsilon} \!. 
	\end{align*}
\end{theorem}

\begin{proof} 
	We fix ${q\in[p,\infty)}$, 
	$\delta\in[0, \bar{\beta})$, 
	$\epsilon\in(0,\infty)$,  
	and first observe that by the triangle inequality 
	on $L_q(\Omega;\xkseps C^\delta([0,T];E))$, 
	\begin{align*}
	\mathrm{err}_{q, \eps_s}^{k,\sf ML} (X;\delta) 
	&\leq 
	\bigl\| 
	\Mk[X] 
	- 
	\Mk\bigl[Y^{N_L}\bigr] 
	\bigr\|_{\xkseps C^\delta([0,T];E)} 
	\\
	&\;\;\; +
	\biggl\| 
	\Mk\bigl[Y^{N_L}\bigr] 
	-
	\sum_{\ell = 1}^{L} 
	\frac{1}{M_\ell} 
	\sum_{j=1}^{M_\ell}  
	\xi_{\ell,j}
	\biggr\|_{L_q(\Omega;\xkseps C^\delta([0,T];E))} 
	=:
	\text{(A)} + \text{(B)}. 
	\end{align*} 

	Term (A) 
	can be bounded by 
	combining Lemma~\ref{lem:Mk-diff} with 
	the stability and convergence results 
	of Proposition~\ref{prop:hoelder-conv}, 
	showing that 
	\begin{align*} 
	\text{(A)} 
	&\leq 
	\bigl\| X - Y^{N_L} \bigr\|_{L_k(\Omega;C^\delta([0,T];E))} 
	\sum_{i=0}^{k-1} 
	\Bigl[ 
	\norm{X}{L_k(\Omega;C^\delta([0,T];E))}^i 
	\bigl\| Y^{N_L} \bigr\|_{L_k(\Omega;C^\delta([0,T];E))}^{k-i-1} 
	\Bigr]
	\\
	&\lesssim_{(\bar{\beta},\delta,\epsilon,k,T)} 
	\bigl| \triangle t^{N_L}_{\max} \bigr|^{\bar{\beta}-\delta-\epsilon} \!. 
	\end{align*} 

	To bound term (B), we may without loss of 
	generality assume that $\epsilon < \bar{\beta} - \delta$ 
	and define 
	$\beta := \delta + \frac{\epsilon}{2} \in (\delta,\bar{\beta})$. 
	We then exploit continuous embeddings, 
	similarly as in the proof of 
	\cite[Corollary~5.15]{CoxEtAl2021}: 
	There are constants 
	$C_1, C_2\in(0,\infty)$, 
	depending only on $\beta,\delta,T$, 
	such that 
	\[ 
	\norm{f}{C^\delta([0,T];E)} 
	\leq 
	C_1 
	\norm{f}{W^{\bar{s}}_{\bar{p}}((0,T);E)} 
	\leq 
	C_2
	\norm{f}{C^\beta([0,T];E)} 
	\quad 
	\forall 
	f \in C^\beta([0,T];E), 
	\]
	where $\bar{s}:= \tfrac{\beta+\delta}{2}
	\in(\delta,\beta)\subset(0,1)$ 
	and 
	$\bar{p} := \tfrac{4}{\beta-\delta}\in(4,\infty)$. 
	Here, for $s\in(0,1)$ and $q\in[1,\infty)$, 
	the space 
	$W^s_q((0,T);E)$ denotes the vector-valued 
	fractional Sobolev space, 
	see 
	e.g.~\cite[Definition~2.5.16]{AnalysisInBanachSpacesI2016}. 
	Continuous embeddings 
	are preserved under (full or symmetric) 
	injective tensor products 
	and, thus, we also have that 
	\[
	\norm{u}{\xkseps C^\delta([0,T];E)} 
	\lesssim_{(\beta,\delta,k,T)}
	\norm{u}{\xkseps W^{\bar{s}}_{\bar{p}}((0,T);E)} 
	\quad 
	\forall 
	u \in \xkseps W^{\bar{s}}_{\bar{p}}((0,T);E). 
	\]
	
	In addition, 
	we note that 
	$\bar{E} := W^{\bar{s}}_{\bar{p}}((0,T);E)$ has type 
	$\min\{p,\bar{p}\} = p$: 
	This observation follows from 
	the fact that both 
	$L^{\bar{p}}((0,T);E)$ 
	and 
	$W^1_{\bar{p}}((0,T);E)$ 
	have type $\min\{p,\bar{p}\}$ 
	(see \cite[Proposition~7.1.4]{AnalysisInBanachSpacesII2017})
	combined with the property that 
	\[
	\bar{E} 
	= 
	W^{\bar{s}}_{\bar{p}}((0,T);E) 
	= 
	\bigl( L^{\bar{p}}((0,T);E), 
	W^1_{\bar{p}}((0,T);E) \bigr)_{\bar{s},\bar{p}}
	\]
	is the real interpolation space 
	between $L^{\bar{p}}((0,T);E)$ 
	and 
	$W^1_{\bar{p}}((0,T);E)$ 
	(cf.~\cite[Theorem~2.5.17]{AnalysisInBanachSpacesI2016})
	and the specification of the type of 
	interpolation spaces 
	\cite[Proposition~7.1.3]{AnalysisInBanachSpacesII2017}.  
	Thus,  
	we may conclude 
	with Theorem~\ref{thm:MLMC-k} 
	and Proposition~\ref{prop:hoelder-conv} 
	that 
	\begin{align*} 
	&\text{(B)} 
	\lesssim_{(\beta,\delta,k,T)}
	\biggl\| 
	\Mk\bigl[Y^{N_L}\bigr] 
	-
	\sum_{\ell = 1}^{L} 
	\frac{1}{M_\ell} 
	\sum_{j=1}^{M_\ell}  
	\xi_{\ell,j}
	\biggr\|_{L_q(\Omega;\xkseps \bar{E} ) } 
	\\
	&\;\lesssim_{(k,p,q)}
	\sum_{\ell=1}^{L}   
	\Bigl[ 
	M_\ell^{-\left( 1-\frac{1}{p} \right)}  
	\bigl\| Y^{N_\ell} - Y^{N_{\ell-1}} 
	\bigr\|_{L_{kq}(\Omega; \bar{E} )} 
	\max_{1\leq\ell\leq L} 
	\bigl\| Y^{N_\ell} \bigr\|_{L_{kq}(\Omega; \bar{E} )}^{k-1}
	\Bigr]
	\\ 
	&\;\lesssim_{(\beta,\delta,k,T)} 
	\sup_{N\in\bbN} 
	\bigl\| Y^N \bigr\|_{L_{kq}(\Omega;C^\beta([0,T];E))}^{k-1}
	\sum_{\ell=1}^{L}   
	\Bigl[ 
	M_\ell^{-\left( 1-\frac{1}{p} \right)}  
	\bigl\| X - Y^{N_{\ell-1}} \bigr\|_{L_{kq}(\Omega;C^\beta([0,T];E))} 
	\Bigr] 
	\\ 
	&\;\lesssim_{(\bar{\beta},\delta,\epsilon,\widetilde{\epsilon},k,q,T)} 
	\sum_{\ell=1}^{L}   
	M_\ell^{-\left( 1-\frac{1}{p} \right)}  
	\bigl| \triangle t^{N_{\ell-1}}_{\max} 
	\bigr|^{\bar{\beta}-\beta-\widetilde{\epsilon}}
	=
	\sum_{\ell=1}^{L}   
	M_\ell^{-\left( 1-\frac{1}{p} \right)}  
	\bigl| \triangle t^{N_{\ell-1}}_{\max} 
	\bigr|^{\bar{\beta}-\delta-\frac{\epsilon}{2}-\widetilde{\epsilon}}
	\end{align*} 
	holds for all $\widetilde{\epsilon}\in(0,\infty)$, 
	and the claim follows 
	for the choice 
	$\widetilde{\epsilon}:=\frac{\epsilon}{2}$. 
\end{proof}   

\begin{example}[Euler--Maruyama method for SDEs] 
	Let $(\cF_t)_{t\in[0,T]}$ be a normal filtration on 
	$(\Omega,\cA,\bbP)$ and let 
	$B\from[0,T]\times\Omega\to\bbR^m$ be 
	an $m$-dimensional 
	$(\cF_t)_{t\in[0,T]}$-Brownian motion 
	(with continuous sample paths). 
	For Lipschitz continuous functions 
	$\mu\from\bbR^d\to \bbR^d$ and 
	$\sigma\from\bbR^d \to\bbR^{d\times m}$\!, 
	consider the $(\cF_t)_{t\in[0,T]}$-adapted 
	stochastic process 
	$X\colon[0,T]\times\Omega\to\bbR^d$ 
	with continuous sample paths 
	that satisfies 
	\[
	X(t)  
	= 
	X(0) 
	+ 
	\int_0^t \mu(X(s)) \, \rd s 
	+ 
	\int_0^t \sigma(X(s)) \, \rd B(s), 
	\quad 
	\text{$\bbP$-a.s.}, 
	\]
	as well as the Euler--Maruyama 
	approximations 
	$(Y^N)_{N\in\bbN}$ 
	to $X$, defined 
	with respect to equidistant partitions 
	of size 
	$\triangle t^N_{\max} = 
	\triangle t^N_{\min} = \nicefrac{T}{N}$
	as follows: $Y^N\!(0) := X(0)$, 
	and for $j\in\{0,\ldots,N-1\}$ 
	and $s\in\bigl( \tfrac{jT}{N}, \tfrac{(j+1)T}{N} \bigr]$, 
	\[
	Y^N\!(s) 
	:= 
	Y^N\!\bigl( \tfrac{jT}{N} \bigr) 
	+ 
	\bigl( s - \tfrac{jT}{N} \bigr) 
	\mu\bigl( Y^N\!\bigl(\tfrac{jT}{N} \bigr) \bigr) 
	+ 
	\bigl( \tfrac{sN}{T} - j \bigr)
	\sigma\bigl( Y^N\!\bigl( \tfrac{jT}{N} \bigr) \bigr)
	\bigl( B\bigl( \tfrac{(j+1)T}{N} \bigr) 
	- B\bigl(\tfrac{jT}{N} \bigr) \bigr).  
	\]
	Then, using the notation of 
	Theorem~\ref{thm:mlmc-hoelder} 
	with $E:=\bbR^d$\!, 
	we conclude that 
	for every $q\in[2,\infty)$, $\delta\in[0,\nicefrac{1}{2})$ and 
	all $\epsilon\in(0,\infty)$, 
	\[ 
	\mathrm{err}_{q, \eps_s}^{k,\sf ML} (X;\delta) 
	\lesssim_{(\delta,\epsilon,k,q,T)} 
	N_L^{-\left(\frac{1}{2}-\delta-\epsilon\right)} 
	+ 
	\sum_{\ell=1}^{L}   
	M_\ell^{- \frac{1}{2} }  
	N_{\ell-1}^{-\left( \frac{1}{2} - \delta - \epsilon \right) } \!. 
	\] 
	For the choice 
	$N_\ell := 2^\ell$ 
	and $M_\ell := 2^{L-\ell}$\!, 
	this yields the error bound 
	\[ 
	\mathrm{err}_{q,\eps_s}^{k,\sf ML} (X;\delta) 
	\lesssim_{(\delta,\epsilon,k,q,T)} 
	2^{-L\left(\frac{1}{2}-\delta-\epsilon\right)} 
	+ 
	\sum_{\ell=1}^{L}   
	2^{- \frac{L-\ell}{2} }  
	2^{-\ell\left( \frac{1}{2} - \delta - \epsilon \right) }  
	\lesssim 
	2^{-L\left(\frac{1}{2}-\delta-\epsilon\right)} \!.  
	\] 
	Since $\xkeps C([0,T]; \bbR) 
	= C([0,T]^k; \bbR)$, 
	see~\cite[Section~3.2, p.~50]{Ryan2002}, 
	this error estimate holds in particular 
	also on $C([0,T]^k; \bbR)$. 
\end{example} 

\begin{remark} 
	We note that Theorem~\ref{thm:mlmc-hoelder} 
	is applicable 
	to a variety of numerical schemes  
	developed for stochastic evolution problems 
	(such as SDEs and stochastic PDEs), 
	for which the regularity~\eqref{eq:ass:X-beta-regularity} 
	of the solution process is known, and 
	strong convergence rates are available at 
	the nodes of the temporal partitions 
	in $L_q(\Omega;E)$-sense 
	for any $q\in[1,\infty)$. 
	This list includes for instance: 
	\begin{enumerate}[leftmargin=1cm, label=(\alph*)] 
		\item SDEs with coefficients that 
		are not globally Lipschitz continuous, 
		see e.g.~\cite[Theorem~3.1]{GyongyRasonyi2011},  
		\cite[Theorem~1.1]{HutzenthalerJentzenKloeden2012} 
		and 
		\cite[Theorem~4.5]{MullerGronbachYaroslavtseva2022}; 
		\item fully discrete (in space and time) 
		approximations for linear or semilinear 
		parabolic stochastic PDEs, 
		see e.g.~\cite[Theorem~3.14]{Kruse2014}; 
		\item fully discrete  
		approximations for non-linear stochastic PDEs, such as the 
		stochastic Allen--Cahn equation, 
		see e.g.~\cite[Theorem~1.1]{BeckerEtAl2022}. 
	\end{enumerate}
\end{remark}

\section{Conclusions} 
\label{section:conclusions}   

We have analyzed the convergence of Monte Carlo sampling   
for higher-order moments of 
Banach space valued random variables. 
Specifically, for every $k\in\bbN$, 
we have derived explicit, $k$-independent 
strong convergence rates  
in the injective tensor norm 
for approximating the $k$th moment $\Mk[X]$ 
of a random variable $X\from\Omega\to E$,
taking values in a Banach space $E$, 
by means of 
\begin{enumerate}[leftmargin=10mm, label=\Roman*.] 
	\item 
	\emph{standard Monte Carlo} sampling, 
	involving no further numerical approximation, 
	see Theorem~\ref{thm:MC-k}; 
	\item 
	the \emph{single-level Monte Carlo} method, 
	combining Monte Carlo sampling 
	with an approximation  
	$X_1\from\Omega\to E_1$ of $X$ 
	to generate 
	samples in a (usually finite-dimensional) 
	subspace $E_1\subseteq E$, 
	see Corollary~\ref{cor:MC-k}; 
	\item the \emph{multilevel Monte Carlo} method, 
	combining Monte Carlo sampling 
	with a hierarchy of approximations  
	$X_\ell\from\Omega\to E_\ell$, 
	$\ell\in\{1,\ldots,L\}$, 
	in (usually nested, finite-dimensional)
	subspaces $E_\ell \subseteq E$, 
	see Theorems~\ref{thm:MLMC-k} 
	and~\ref{thm:alpha-beta-gamma}. 	
\end{enumerate} 
These findings  
extend the numerical analysis of 
Monte Carlo based algorithms in 
computational uncertainty quantification 
to a broad range of mathematical models 
beyond the classical theory in Hilbert spaces, 
which relies on assumed 
square-integrability 
and bias-variance decompositions. 
Several examples have illustrated 
the wide scope of 
the presently developed theory: 
linear, second-order elliptic PDEs with data 
affording well-posedness
in $W^{1}_{p}$, 
and
stochastic evolution equations 
with almost sure path regularity
in H\"older spaces. 

The results of 
Subsections~\ref{subsec:slmc} and~\ref{subsec:mlmc}
are essential for the error analysis of 
Monte Carlo approximations of $k$-point correlations for
every operator equation with random input data
which, 
due to modeling or physical constraints, 
does not admit a 
well-posed formulation 
in Hilbert spaces. 
We indicate some further applications, 
where this is of relevance:
In~\cite{KoleyRisebroSchwabWeber2017} 
Monte Carlo finite difference 
discretizations for 
scalar, degenerate convection-diffusion equations 
with random initial data were considered. 
In that case, the 
particular structure of the degeneracy in the diffusion coefficient,
imposed from physical properties of the underlying model, 
mandated a mathematical formulation 
in Banach spaces of type $p<2$.
Assuming random initial data, 
the corresponding Monte Carlo error analysis 
for mean values of 
the solution therefore 
required a setting in Banach spaces 
as in Corollary~\ref{cor:MC-1}. 
With the abstract MLMC results of 
Theorems~\ref{thm:MLMC-k} and~\ref{thm:alpha-beta-gamma}, 
the MLMC finite difference  
convergence analysis 
for first-order moments 
of~\cite{KoleyRisebroSchwabWeber2017} 
generalizes to spatiotemporal 
$k$-point correlations with $k\geq 2$. 

Another application is related to fluid flows: 
For the compressible Navier--Stokes equations
with spatially periodic solutions, 
the (isentropic) equation of state 
relates the pressure $P$ 
to the fluid density $\varrho$ via 
$P(\varrho) = a \varrho^\gamma$\!, 
where $a>0$ and $\gamma>1$ 
are \emph{physical constants}. 
In well-posed variational formulations~\cite{LionsVol2}, 
the density $\varrho(t,\,\cdot\,)$ and 
the corresponding momentum $\boldsymbol{m}(t,\,\cdot\,)$ 
at time $t\in[0,T]$ 
take values in 
$L_\gamma\bigl(\bbT^d \bigr)$ and 
$L_{\frac{2\gamma}{\gamma+1}}\bigl(\bbT^d;\bbR^d \bigr)$, 
respectively, 
where $d\in\{2,3\}$ 
and $\bbT^d$ denotes 
the $d$-dimensional torus. 
With random data, 
this entails a Banach space setting of type 
$p=\min\{\gamma,2\}$ for the density 
and 
$p=\frac{2\gamma}{\gamma+1} \in (1,2)$ 
for the momentum. 
The convergence of \emph{single-level}
Monte Carlo finite volume approximations  
for higher-order moments 
of $\varrho(t,\,\cdot\,)$ 
and $\boldsymbol{m}(t,\,\cdot\,)$ 
has been discussed 
in the recent work~\cite{FereislEtAl22}, 
using the isomorphic identification 
\begin{equation}\label{eq:Lp-tensor-identification} 
	\otimes_{\gamma}^k 
	L_\gamma\bigl(\bbT^d\bigr) 
	\cong 
	L_{\gamma}\bigl(\bbT^{kd}\bigr)  
\end{equation}
(and similarly for the $d$ components 
of the momentum). 
Here, $\otimes_{\gamma}^k$ indicates   
the appropriate Chevet--Saphar
tensor product space, 
see e.g.\ \cite[Chapter~6]{Ryan2002}. 
The present, abstract MLMC results 
apply directly to the 
setting of~\cite{FereislEtAl22},
implying corresponding convergence results for
multilevel Monte Carlo approximations. 

An interesting topic for future 
research is to investigate 
whether the Monte Carlo 
convergence results, 
derived for injective tensor 
product spaces in this work, 
hold also with respect to stronger 
cross norms.
In particular, the identification of 
tensor products of 
$L_p$-spaces 
as in \eqref{eq:Lp-tensor-identification} 
raises the question if 
it is possible to use 
one of the $p$th 
Chevet--Saphar tensor norms $d_p$ or $g_p$
(see \cite[p.~135]{Ryan2002})
if the Banach space~$E$ has type $p$. 
The Chevet--Saphar norms 
and the Hilbert tensor norm $w_2$ 
are unified by the 
tensor norms $\{\alpha_{p,q}\}_{1\leq p,q\leq \infty}$ 
due to Laprest\'{e}: 
$g_p = \alpha_{p,1}$,
$d_p = \alpha_{1,p}$, 
and 
$w_2 = \alpha_{2,2}$, 
see \cite[Sections 12.5--12.8]{DeFlrtBook} 
and the references there. 
However, in this generality, 
there do not seem to be symmetric versions 
of these tensor norms available in the literature.  
A corresponding (ML)MC convergence analysis 
would thus have to be based on  
considerably different arguments.  

Furthermore, this work 
may be extended to 
\emph{sparse tensor approximations}
as considered in the 
Hilbert space setting in \cite{BarthSchwabZollinger2011}. 
Specifically, 
we analyzed the MLMC approach 
for approximating the $k$th moment $\Mk[X]$ 
using samples of the (exact, full) tensor 
product $\otimes^k X_\ell$ 
on levels $\ell\in\{1,\ldots,L\}$. 
The formation of this $k$-fold  
tensor product $\otimes^k X_\ell$
on level $\ell$ 
typically entails costs in work and memory 
of the order $\cC_\ell \lesssim N_\ell^{\gamma}$ 
with $\gamma = \max\{\gamma_1,k\}$, 
assuming that $\gamma_1$ is the exponent 
in the asymptotic cost bound 
for computing one sample of $X_\ell$, 
and that one computed sample of $X_\ell$ 
requires storage of order $N_\ell$. 
As it is well-known 
in the Hilbert space case, 
various consistent  
sparse tensor product approximations 
allow to reduce this complexity considerably. 
For example, for the applications discussed 
in Subsections~\ref{subsec:EllPDERnDFor} 
and~\ref{subsec:EllPDERndCoef}, 
the sparse tensor product approach 
for the MLMC approximation 
of $k$th moments 
proposed in \cite{BarthSchwabZollinger2011} 
for Hilbert spaces 
can be leveraged to reduce the parameter $\gamma$ in 
Theorem~\ref{thm:alpha-beta-gamma} 
from $\max\{\gamma_1,k\}$
to $\gamma_1 + \delta$ for some 
(arbitrarily small) $\delta>0$. 
Yet, in this setting
the error analysis of Theorem~\ref{thm:MLMC-k} 
and, consequently, 
also of Theorem~\ref{thm:alpha-beta-gamma} 
does not readily apply. 

Beyond the MLMC estimation 
of $k$th moments $\Mk[\xi] = \bbE[\xk \xi]$, 
one may consider 
\emph{anisotropic} $k$-fold 
correlations of the form 
$\bbE[\xi_1\otimes\cdots\otimes\xi_k]$.  
Here, the vector-valued random variables $\xi_1,\ldots,\xi_k$ 
entering the anisotropic, 
injective tensor product formation 
may take values in Banach spaces 
$E_1,\ldots,E_k$ of (possibly different) types 
$p_1,\ldots,p_k\in[1,2]$. 
This rather general setting has numerous applications, 
and can be analyzed with the techniques in the present paper, 
in conjunction with the \emph{multi-index Monte Carlo} 
approach from \cite{MIMC2016}.
Details shall be reported elsewhere.  

\appendix 

\section{Tensor norms of symmetric elements of Hilbert spaces}
\label{appendix:norm-computation} 

In this section we consider a real separable 
Hilbert space $(H, \scalar{\,\cdot\,, \,\cdot\,}{H})$  
and explicitly compute the 
projective and injective tensor norms 
of Subsection~\ref{subsec:k-tensor} 
for symmetric elements in $\otimes^{2,s} H$ 
of the form $\sum_{j=1}^n \lambda_j\, e_j \otimes e_j$, 
where $e_1,\ldots,e_n$ 
are orthonormal in $H$,  
and $\lambda_1,\ldots,\lambda_n\in\bbR$. 

For this purpose, 
we need the notion of the 
\emph{$k$-fold Hilbert tensor product space} 
$\xk_{w_2} H$,  
which is defined as the closure 
of the (full) $k$-fold algebraic tensor product space 
$\xk H$ with respect to the norm 
which is induced by the inner product 
\begin{equation}\label{eq:Hilbert-tensor-IP} 
	\biggl( \sum_{j=1}^n 
	\bigotimes_{\nu=1}^k 
	x_{j,\nu}, 
	\sum_{i=1}^{\widetilde{n}} 
	\bigotimes_{\nu'=1}^k 
	y_{i,\nu'}
	\biggr)_{\!w_2}
	:= 
	\sum_{j=1}^n 
	\sum_{i=1}^{\widetilde{n}} 
	\prod_{\nu=1}^k
	\scalar{x_{j,\nu}, y_{i,\nu}}{H}. 
\end{equation}
In particular, the tensor product space 
$(\xk_{w_2} H, \scalar{\,\cdot\,, \,\cdot\,}{w_2}  )$ 
is again a Hilbert space. 

\begin{lemma}\label{lem:diagonal-norms}
	Assume that $(H, \scalar{\,\cdot\,, \,\cdot\,}{H})$ 
	is a real separable Hilbert space 
	and $(e_j)_{j\in\bbN}$ is an orthonormal basis for~$H$. 
	Let the injective and projective tensor norms, 
	$\norm{\,\cdot\,}{\eps}$ and $\norm{\,\cdot\,}{\pi}$, 
	be defined on $\otimes^2 H$
	as in \eqref{eq:epsnorm} and \eqref{eq:pinorm},  
	and let the symmetric injective and projective tensor norms,  
	$\norm{\,\cdot\,}{\eps_s}$ and $\norm{\,\cdot\,}{\pi_s}$,  
	be defined on $\otimes^{2,s} H$ as in 
	\eqref{eq:sepsnorm} and \eqref{eq:spinorm}, respectively. 
	Let $n\in\bbN$ and $\lambda_1,\ldots,\lambda_n\in\bbR$. 
	Then, 
	\begin{align} 
		\biggl\| \sum_{j=1}^n \lambda_j \, e_j \otimes e_j \biggr\|_{\pi}
		&= 
		\biggl\| \sum_{j=1}^n \lambda_j \, e_j \otimes e_j \biggr\|_{\pi_s}
		= 
		\sum_{j=1}^n |\lambda_j|, 
		\label{eq:diagonal-pi-norms}
		\\ 
		\biggl\| \sum_{j=1}^n \lambda_j \, e_j \otimes e_j \biggr\|_{\eps}
		&= 
		\biggl\| \sum_{j=1}^n \lambda_j \, e_j \otimes e_j \biggr\|_{\eps_s}
		= 
		\max_{1\leq j\leq n} |\lambda_j|. 
		\label{eq:diagonal-eps-norms}
	\end{align} 
\end{lemma} 

\begin{proof} 
	We set 
	$U_n := \sum_{j=1}^n \lambda_j \, e_j\otimes e_j \in \otimes^{2,s} H$.  
	Then, by the definitions of the projective norms 
	in~\eqref{eq:pinorm} and \eqref{eq:spinorm}, 
	see Remark~\ref{rem:projective-norm}, 
	it follows that 
	\[
		\norm{ U_n }{\pi} 
		\leq  
		\norm{ U_n }{\pi_s} 
		\leq 
		\sum_{j=1}^n |\lambda_j| \, \norm{e_j}{H}^2 
		= 
		\sum_{j=1}^n |\lambda_j|. 
	\]
	Furthermore, 
	for any representation 
	$U_n = \sum_{i=1}^{\widetilde{n}} x_i \otimes y_i \in \otimes^2 H$ 
	of $U_n$ 
	we find that 
	\[ 
		\sum_{j=1}^n 
		| \lambda_j | 
		= 
		\sum_{j\in\bbN}
		| \scalar{ U_n , e_j \otimes e_j}{w_2} | 
		= 
		\sum_{j\in\bbN} 
		\left| 
		\biggl( 
		\sum_{i=1}^{\widetilde{n}} x_i \otimes y_i ,\, e_j \otimes e_j 
		\biggr)_{\!w_2} 
		\right|, 
	\]
	where $\scalar{\,\cdot\,, \,\cdot\,}{w_2}$ 
	is the inner product on the Hilbert tensor 
	product space~$\otimes^2_{w_2} H$,   
	defined as in \eqref{eq:Hilbert-tensor-IP} 
	for $k=2$. 
	Thus, by the triangle and Cauchy--Schwarz inequalities, 
	\begin{align*} 
		\sum_{j=1}^n 
		| \lambda_j | 
		&\leq 
		\sum_{i=1}^{\widetilde{n}}
		\sum_{j\in\bbN} 
		\bigl| \scalar{ x_i\otimes y_i , e_j\otimes e_j}{w_2} \bigr| 
		= 
		\sum_{i=1}^{\widetilde{n}}
		\sum_{j\in\bbN} 
		\bigl| \scalar{ x_i, e_j}{H} \bigr|  \bigl| \scalar{ y_i , e_j }{H} \bigr| 
		\\
		&\leq 
		\sum_{i=1}^{\widetilde{n}}
		\Biggl( \sum_{j\in\bbN} 
		\scalar{ x_i, e_j}{H}^2 
		\Biggr)^{\nicefrac{1}{2}} 
		\Biggl( \sum_{j\in\bbN} 
		\scalar{ y_i, e_j}{H}^2 
		\Biggr)^{\nicefrac{1}{2}} 
		= 
		\sum_{i=1}^{\widetilde{n}}
		\norm{x_i}{H}
		\norm{y_i}{H} . 
	\end{align*} 
	By taking the infimum over all representations of $U_n \in \otimes^2 H$ 
	we obtain the reverse inequality 
	$\norm{ U_n }{\pi}\geq \sum_{j=1}^n |\lambda_j|$ and, 
	since also $\norm{ U_n }{\pi_s} \geq \norm{ U_n }{\pi}$, 
	this proves \eqref{eq:diagonal-pi-norms}. 
	
	To show \eqref{eq:diagonal-eps-norms}, 
	let $j_\star\in\{1,\ldots,n\}$ be an index 
	such that 
	$| \lambda_{j_\star} | = \max_{1\leq j\leq M} |\lambda_j|$, 
	and recall the definitions of the injective norms from 
	\eqref{eq:epsnorm} and \eqref{eq:sepsnorm}. 
	Then, we find
	\begin{align*} 
		\norm{ U_n }{\eps} 
		&\geq 
		\norm{ U_n }{\eps_s} 
		= 
		\sup_{f\in B_{\dual{H}} } \biggl| \sum_{j=1}^n \lambda_j \duality{f, e_j}^2 \biggr| 
		\geq 
		\biggl| \sum_{j=1}^n \lambda_j \scalar{e_{j_\star}, e_j}{H}^2 \biggr| 
		= 
		|\lambda_{j_\star}| 
		= 
		\max_{1\leq j\leq n} |\lambda_j|. 
	\end{align*} 
	The reverse estimates follow again by the Cauchy--Schwarz inequality 
	combined with the Riesz representation theorem, 
	\begin{align*} 
		\norm{ U_n }{\eps} 
		&= 
		\sup_{f_1, f_2\in B_{\dual{H}} } 
		\biggl| \sum_{j=1}^n \lambda_j \duality{f_1, e_j} \duality{f_2, e_j} \biggr| 
		\leq 
		|\lambda_{j_\star}| 
		\sup_{f_1, f_2\in B_{\dual{H}} } 
		\sum_{j=1}^n | \duality{f_1, e_j} | \, |\duality{f_2, e_j}| 
		\\
		&=
		\max_{1\leq j\leq n} |\lambda_j| 
		\sup_{v_1, v_2\in B_H } 
		\sum_{j=1}^n | \scalar{v_1, e_j}{H} | \, |\scalar{v_2, e_j}{H}| 
		\leq 
		\max_{1\leq j\leq n} |\lambda_j|.
	\end{align*} 
	Thus, 
	$\norm{ U_n }{\eps_s} 
	\leq 
	\norm{ U_n }{\eps} 
	\leq 
	\max_{1\leq j\leq n} |\lambda_j|$ 
	completing the proof of \eqref{eq:diagonal-eps-norms}. 
\end{proof}

\begin{remark}[Relation of $\otimes^{2,s} H$ 
		to self-adjoint finite-rank linear operators]
	In the setting of Lemma~\ref{lem:diagonal-norms}, 
	we may associate 
	a self-adjoint linear operator on the Hilbert space $H$ 
	with the element 
	$U_n := \sum_{j=1}^n \lambda_j \, e_j\otimes e_j$ 
	in the symmetric algebraic tensor product space $\otimes^{2,s} H$. 
	More specifically, 
	we can define the self-adjoint 
	finite-rank linear operator $T_{U_n} \from H \to H$
	associated with $U_n$ by 
	$T_{U_n} x 
	:= 
	\sum_{j=1}^n \lambda_j \scalar{x, e_j}{H} \, e_j$, 
	for every $x\in H$. 
	Using this definition, the norm identities 
	in~\eqref{eq:diagonal-pi-norms} and~\eqref{eq:diagonal-eps-norms} 
	can be reformulated in terms of 
	the trace-class (or nuclear) norm, 
	$\norm{T_{U_n}}{\cL_1(H)} := \operatorname{tr} ( |T_{U_n}| )$, 
	and the operator norm, 
	$\norm{T_{U_n}}{\cL(H)} := \sup_{x\in B_H } \norm{ T_{U_n} x}{H}$, 
	of $T_{U_n}$ as follows: 
	\begin{gather*} 
		\textstyle 
		\norm{U_n}{\pi}
		= 
		\norm{U_n}{\pi_s} 
		= 
		\sum\limits_{j=1}^n |\lambda_j| 
		= 
		\norm{T_{U_n}}{\cL_1(H)} , 
		\\ 
		\norm{U_n}{\eps} 
		= 
		\norm{U_n}{\eps_s}
		= 
		\max_{1\leq j\leq n} |\lambda_j| 
		= 
		\norm{T_{U_n}}{\cL(H)} , 
	\end{gather*} 
	see e.g.\ \cite[Theorem~14.15.(1) and Theorem~8.11]{vanNeerven-FA} 
	for the operator norm identities. 
	
	More generally, 
	to every element $U \in \otimes^{2,s} H$, 
	we can associate a self-adjoint linear operator $T_U\from H \to H$, 
	whose action on $x\in H$ is defined by 
	\[
		\scalar{T_U x, y}{H} 
		= 
		\scalar{U, x\otimes y}{w_2} 
		\quad 
		\forall y \in H. 
	\]
	Here, the Riesz representation theorem 
	ensures that the linear operator 
	$T_U$ is well-defined. 
	The implied linear mapping 
	${\cI: U \mapsto T_U}$
	extends continuously 
	to an isometric isomorphism 
	between the 
	symmetric projective tensor product space $\otimes^{2,s}_{\pi_s} H$ 
	and the space of self-adjoint trace-class linear operators on $H$ 
	(respectively, between the 
	symmetric injective tensor product space $\otimes^{2,s}_{\eps_s} H$ 
	and the space of self-adjoint compact linear operators on~$H$). 
\end{remark} 

\section{A consequence of Slepian's inequality}
\label{appendix:details-fernique} 

In this section we restate the version of Slepian's inequality 
for finite-dimensional Gaussian processes as 
formulated by Fernique~\cite{Fernique1975}. 
We subsequently use it to derive a comparison result for 
real-valued Gaussian processes indexed by the closed unit ball 
$B_{\dual{E}}$ of the dual of a real Banach space $E$, 
see Lemma~\ref{lem:slepian}. 
This result is 
needed in Subsection~\ref{subsec:mlmc} 
to prove convergence 
of multilevel Monte Carlo methods. 

The following theorem is taken from~\cite[Theorem~2.1.2]{Fernique1975}. 
We note that a more general version of Slepian's inequality, 
which includes Fernique's formulation 
as a special case, can be found 
in~\cite[Theorem~2.8]{HoffmannJorgensen2013}.  

\begin{theorem}\label{thm:slepian} 
	Let $N\in\bbN$ and 
	$X = (X_1,\ldots,X_N)^\top$\!,  
	$Y = (Y_1,\ldots,Y_N)^\top$ be 
	two centered Gaussian 
	random vectors 
	in $\bbR^N$\!, defined 
	on a complete probability space 
	$(\Omegat,\cAt,\bbPt)$ 
	with expectation $\bbEt$. 
	Assume further that 
	\[
		\forall i,j\in\{1,\ldots,N\}: 
		\quad 
		\bbEt\bigl[ |X_i - X_j|^2 \bigr] 
		\leq 
		\bbEt\bigl[ |Y_i - Y_j|^2 \bigr] , 
	\]
	and let $G \from [0,\infty)\to [0,\infty)$ 
	be convex and increasing. Then, 
	\[
		\bbEt G \Bigl( \max_{1\leq i,j \leq N} |X_i - X_j| \Bigr) 
		\leq 
		\bbEt G \Bigl( \max_{1\leq i,j \leq N} |Y_i - Y_j| \Bigr) . 
	\]
\end{theorem}  
	
\begin{lemma}\label{lem:slepian} 
	Let $M\in\bbN$ and assume that 
	$(g_j)_{j=1}^M$ is 
	an orthogaussian family 
	on a complete probability space 
	$(\Omegat,\cAt,\bbPt)$ 
	(with expectation $\bbEt$), 
	and that $\Psi \from\bbR^2\to\bbR$ 
	is a continuous function 
	such that $\Psi(0,0)=0$. 
	Let $( x_j )_{j=1}^M, \, ( y_j)_{j=1}^M \subset E$, 
	and define the 
	centered Gaussian process 
	$\cG_1 \from B_{\dual{E}\!} \times \Omegat \to \bbR$ 
	on $(\Omegat,\cAt,\bbPt)$ 
	indexed by the closed unit ball $B_{\dual{E}\!}$ 
	in the dual space $\dual{E}$ by 
	\begin{equation}\label{eq:defG1}
		\cG_1(f) 
		:= 
		\sum_{j=1}^M 
		g_j  \Psi\bigl( f(x_j), f(y_j) \bigr) , 
		\quad 
		f\in B_{\dual{E}\!}. 
	\end{equation}
	Let $\cG_2 \from B_{\dual{E}\!} \times \Omegat \to \bbR$ 
	be a second centered Gaussian process on 
	$(\Omegat,\cAt,\bbPt)$ such that 
	\begin{equation}\label{eq:ass:slepian} 
		\forall f,h \in B_{\dual{E}\!}: 
		\quad 
		\bbEt\bigl[ |\cG_1(f) - \cG_1(h)|^2\bigr] 
		\leq 
		\bbEt\bigl[ |\cG_2(f) - \cG_2(h)|^2\bigr]. 
	\end{equation}
	Then, for all $q\in[1,\infty)$ we have that 
	\begin{equation}\label{eq:lem:slepian} 
		\bbEt\Bigl[ 
		\bigl(\sup\nolimits_{f \in B_{\dual{E}} \!} 
		|\cG_1(f) | \bigr)^q \Bigr] 
		\leq 
		2^q \, 
		\bbEt\Bigl[
		\bigl(\sup\nolimits_{f \in B_{\dual{E}} \!} 
		|\cG_2(f) | \bigr)^q \Bigr]. 
	\end{equation}
\end{lemma} 

\begin{proof} 	
	Set $N_0:=0$ 
	and $f_0 := 0 \in\dual{E}$\!.  
	Given $M\in\bbN$ and 
	$(x_j)_{j=1}^M, \, (y_j)_{j=1}^M \subset E$, 
	for $f\in \dual{E}$ and $\delta\in(0,\infty)$,
	define the subset 
	\[
		U_\delta(f) 
		:= 
		\Bigl\{ h \in \dual{E} : 
		\max_{1\leq j\leq M} | f(x_j) - h(x_j) | < \delta, \ 
		\max_{1\leq j\leq M} | f(y_j) - h(y_j) | < \delta  
		\Bigr\} 
		\subseteq  
		\dual{E}\! . 
	\]
	Then, 
	for every $f\in\dual{E}$ and all $\delta\in(0,\infty)$, 
	the set $U_\delta(f) $ is open 
	(more precisely, an open neighborhood of $f$)
	in $\dual{E}$ 
	with respect to the 
	$\text{weak}^*$-topology on the dual space~$\dual{E}$\!. 
	By the Banach--Alaoglu theorem the closed unit ball 
	$B_{\dual{E}}$ is $\text{weak}^*$-com- pact. 
	Hence, for every $n\in\bbN$, the open cover 
	\[
		\bigcup_{f\in B_{\dual{E}}} 
		U_{\frac{1}{n}}( f ) 
		\supseteq  B_{\dual{E}}
	\]
	contains a finite subcover.
	Iteratively, for every $n\in\bbN$, 
	one can find 
	an integer $N_n\in\bbN$,  
	satisfying $N_n > N_{n-1}$, 
	and elements  
	$f_{N_{n-1}+1},\ldots,f_{N_n} \in B_{\dual{E}}$  
	such that 
	\[
		B_{\dual{E}} 
		\subseteq 
		\bigcup_{\nu=1}^{N_n} 
		U_{ \frac{1}{n}}( f_{\nu} ) . 
	\]
	Note, in particular, 
	that this definition of 
	$f_1,\ldots,f_{N_n}$, $n\in\bbN$, 
	implies nestedness, 
	$( f_1,\ldots,f_{N_n} ) \subseteq ( f_1,\ldots,f_{N_m} )$ 
	for $n<m$. 
	
	Next, 
	we define for every 
	non-negative integer $\nu\in\bbN_0$   
	the real-valued centered Gaussian random variables 
	$X_\nu := \cG_1( f_\nu )$ 
	and 
	$Y_\nu := \cG_2( f_\nu )$. 
	By assumption \eqref{eq:ass:slepian} we then have 
	for all $n\in\bbN$ and every $\nu,\nu'\in\{0,\ldots,N_n\}$, 
	\begin{align*} 
		\bbEt \bigl[ | X_\nu - X_{\nu'} |^2 \bigr] 
		=
		\bbEt \bigl[ | \cG_1(f_\nu) - \cG_1(f_{\nu'}) |^2 \bigr] 
		&\leq 
		\bbEt \bigl[ | \cG_2(f_\nu) - \cG_2(f_{\nu'}) |^2 \bigr] 
		\\
		&= 
		\bbEt \bigl[ | Y_\nu - Y_{\nu'} |^2 \bigr] . 
	\end{align*} 
	By Fernique's version of Slepian's inequality, 
	see Theorem~\ref{thm:slepian}, 
	applied for the convex increasing function 
	$G(t) := t^q$, $t\geq 0$, 
	and by using the fact that $f_0 = 0\in \dual{E}$ 
	implies that $\cG_1(f_0) = 0$ 
	holds $\bbPt$-a.s., 
	we find that, for all $n\in\bbN$, $q\in[1,\infty)$, 
	\begin{align*} 
		\bbEt \Bigl[ \bigl( &\max\nolimits_{0\leq\nu\leq N_n} | \cG_1(f_\nu) | \bigr)^q \Bigr] 
		\leq
		\bbEt \Bigl[ \bigl( \max\nolimits_{0\leq\nu,\nu'\leq N_n} | \cG_1(f_\nu) 
		- \cG_1(f_{\nu'}) | \bigr)^q \Bigr] 
		\\
		&\qquad =  
		\bbEt \Bigl[ \bigl( \max\nolimits_{0\leq\nu,\nu'\leq N_n} |X_\nu - X_{\nu'} | \bigr)^q \Bigr] 
		\leq 
		\bbEt \Bigl[ \bigl( \max\nolimits_{0\leq\nu,\nu'\leq N_n} |Y_\nu - Y_{\nu'} | \bigr)^q \Bigr] 
		\\
		&\qquad \leq  
		2^q \, \bbEt \Bigl[ \bigl( \max\nolimits_{0\leq\nu\leq N_n} |Y_\nu | \bigr)^q \Bigr] 
		\leq   
		2^q \, \bbEt \Bigl[ \bigl( \sup\nolimits_{f\in B_{\dual{E}\!}} | \cG_2(f) | \bigr)^q \Bigr]. 
	\end{align*} 
	To derive \eqref{eq:lem:slepian}, 
	it remains to prove that 
	$\lim_{n\to\infty}\norm{S_n}{L_q(\Omegat;\bbR)} = \norm{S_*}{L_q(\Omegat;\bbR)}$, 
	where 
	\[
		S_n 
		:= 
		\max\limits_{0\leq \nu\leq N_n} | X_\nu | 
		= 
		\max\limits_{0\leq \nu\leq N_n} | \cG_1(f_\nu) | 
		\qquad\text{and}\qquad 
		S_* := \sup\limits_{f\in B_{\dual{E}\!}} | \cG_1(f) | . 
	\]
	By the assumptions on the process $\cG_1$ in \eqref{eq:defG1},  
	there exists a set 
	$\Omegat_0\in\cAt$ with 
	$\bbPt(\Omegat_0)=0$ 
	such that 
	$\overline{g}(\omegat) := 1 + \max_{1\leq j \leq M} |g_j(\omegat)| < \infty$ 
	for all 
	$\omegat\in\Omegat\setminus\Omegat_0$. 
	Fix $\omegat\in\Omegat\setminus\Omegat_0$ 
	and $\epsilon\in(0,1)$. 
	Then, there exists 
	$f^\epsilon = f^\epsilon(\omegat) \in B_{\dual{E}}$ such that 
	\[
		\sup_{f\in B_{\dual{E}\!}} | \cG_1(f)(\omegat) | 
		\leq  
		|\cG_1(f^\epsilon)(\omegat) | 
		+
		\frac{\epsilon}{2}. 
	\]
	In addition, 
	there exists $\delta_\epsilon(\omegat)\in(0,\infty)$, 
	such that the implication 
	\begin{equation}\label{eq:weakstar-cont} 
		h \in U_{\delta_\epsilon(\omegat)} (f^\epsilon)  
		\quad 
		\Longrightarrow 
		\quad 
		| \cG_1(f^\epsilon)(\omegat) - \cG_1(h)(\omegat) | 
		< \frac{\epsilon}{2} 
	\end{equation}
	holds. 
	Indeed, by continuity of $\Psi\from\bbR^2 \to \bbR$ 
	we may choose $\delta_\epsilon(\omegat)\in(0,\infty)$  
	such that 
	\begin{align*} 
		\max_{1\leq j\leq M}  
		&\bigl\{ 
		\max\{ | f^\epsilon(x_j) - h(x_j) |, 
		| f^\epsilon (y_j) - h(y_j) | \} 
		\bigr\} 
		< \delta_\epsilon(\omegat)   
		\\ 
		&\Longrightarrow 
		\quad 
		\max_{1\leq j\leq M} 
		\bigl| \Psi\bigl( f^\epsilon(x_j), f^\epsilon(y_j) \bigr) - \Psi\bigl( h(x_j), h(y_j) \bigr) \bigr| 
		< 
		\frac{\epsilon}{2} \, M^{-1} \overline{g}(\omegat)^{-1}\!. 
	\end{align*} 
	Furthermore, by definition of 
	the sequences 
	$(N_n)_{n\in\bbN}\subset \bbN$ and 
	$(f_\nu)_{\nu\in\bbN}\subseteq B_{\dual{E}}$, 
	there exist integers  
	$n_\epsilon = n_\epsilon(\omegat)\in\bbN$ 
	and 
	$\nu_\star = \nu_\star(\omegat)\in\{1,\ldots,N_{n_\epsilon}\}$ 
	such that 
	\[
		f^\epsilon \in U_{\delta_\epsilon(\omegat)}( f_{\nu_\star} ). 
	\]
	By combining this observation 
	with~\eqref{eq:weakstar-cont} 
	we conclude 
	that 
	\[
		\bigl| | \cG_1(f^\epsilon)(\omegat) | - | \cG_1(f_{\nu_\star})(\omegat) | \bigr| 
		\leq 
		| \cG_1(f^\epsilon)(\omegat) - \cG_1(f_{\nu_\star})(\omegat) | 
		<\frac{\epsilon}{2}, 
	\]
	and   
	\begin{align*} 
		S_{*}(\omegat) 
		&- S_{n_\epsilon} (\omegat) 
		= 
		\sup\limits_{f\in B_{\dual{E}\!}} | \cG_1(f)(\omegat) | 
		- 
		\max\limits_{0\leq \nu\leq N_{n_\epsilon}} | \cG_1(f_\nu)(\omegat) | 
		\\
		&\leq 
		\sup\limits_{f\in B_{\dual{E}\!}} | \cG_1(f)(\omegat) | 
		- 
		| \cG_1(f_{\nu_\star})(\omegat) | 
		\leq 
		| \cG_1(f^\epsilon)(\omegat) | 
		- 
		| \cG_1(f_{\nu_\star})(\omegat) | 
		+
		\frac{\epsilon}{2}
		< 
		\epsilon
	\end{align*} 
	follows. 
	This shows that, for almost all $\omegat\in\Omegat$, 
	\[
		S_*(\omegat)
		= 
		\sup_{f\in B_{\dual{E}\!}} | \cG_1(f)(\omegat) | 
		= 
		\sup_{n\in\bbN}
		\max_{0\leq \nu\leq N_n} | \cG_1(f_\nu) (\omegat)|
		= 
		\lim_{n\to\infty} 
		S_n(\omegat). 
	\]
	Since the 
	non-negative random variables $(S_n)_{n\in\bbN}$ 
	are non-decreasing in $n\in\bbN$, 
	$\bbPt$-almost surely, 
	the $L_q(\Omegat;\bbR)$-convergence 
	$\lim\limits_{n\to\infty}\norm{S_n}{L_q(\Omegat;\bbR)} = \norm{S_*}{L_q(\Omegat;\bbR)}$ 
	follows from the monotone convergence theorem. 
\end{proof}

\section*{Acknowledgments} 

This work was possible in part due to the 
visit of both authors at the 
Erwin Schr\"odinger Institute (ESI) in Vienna,
Austria, during the ESI thematic program 
\emph{Computational Uncertainty Quantification} 
in May and June 2022. 
ChS acknowledges a visit to 
Delft Institute of Applied Mathematics in July 2022, 
and KK a visit to the 
Forschungsinstitut f\"ur Mathematik (FIM) 
at ETH Z\"urich in October 2022. 

KK acknowledges 
helpful comments on 
Slepian's inequality  
and Lemma~\ref{lem:slepian} 
by Jan van Neerven 
and Mark Veraar; 
and 
fruitful discussions 
on the counterexample 
mentioned in Example~\ref{example:projective-bad} 
as well as on the
application considered 
in Subsection~\ref{subsec:appl-hoelder} 
with Sonja Cox.  
The authors furthermore thank 
an anonymous reviewer for 
valuable comments.

\section*{Funding} 

KK acknowledges support 
of the research project 
\emph{Efficient spatiotemporal statistical modelling 
	with stochastic PDEs} 
(with project number 
VI.Veni.212.021) 
by the talent programme \emph{Veni}  
which is financed by 
the Dutch Research Council (NWO). 

\bibliographystyle{siam}
\bibliography{MCBanach-bib}

\end{document}